\documentclass[a4paper,12]{article}
\usepackage{amsmath, amsfonts, amssymb, amsthm, bm, graphics, bbm, color}
\usepackage{graphicx}
\usepackage{subfigure}
\usepackage{mathtools}
\usepackage[table]{xcolor}
\usepackage{url}
\usepackage{mathabx}
\usepackage{cancel}
\usepackage{multicol}
\usepackage{float}
\usepackage{caption}
\usepackage{mathalfa}
\usepackage{hyperref}
\DeclareMathAlphabet\mathbfcal{OMS}{cmsy}{b}{n}
\newtheorem{theorem}{Theorem}[section]
\newtheorem{lemma}[theorem]{Lemma}

\theoremstyle{definition}

\makeatletter
\newcommand\makebig[2]{%
  \@xp\newcommand\@xp*\csname#1\endcsname{\bBigg@{#2}}%
  \@xp\newcommand\@xp*\csname#1l\endcsname{\@xp\mathopen\csname#1\endcsname}%
  \@xp\newcommand\@xp*\csname#1r\endcsname{\@xp\mathclose\csname#1\endcsname}%
}
\makeatother

\makebig{biggg} {3.0}
\makebig{Biggg} {3.5}
\makebig{bigggg}{4.0}
\makebig{Bigggg}{4.5}
\makebig{biggggg}{5.0}
\makebig{Biggggg}{5.5}
\makebig{bigggggg}{6.0}
\makebig{Bigggggg}{6.5}

\newcommand{\dif}{\mathrm{d}}
\newcommand{\im}{\mathrm{i}}

\begin{document}

\title{Efficient Computation of the Magnetic Polarizability Tensor Spectral Signature using POD}

\author{B.A. Wilson and P.D. Ledger\\
Zienkiewicz Centre for Computational Engineering, 
College of Engineering, \\Swansea University \\
 b.a.wilson@swansea.ac.uk, p.d.ledger@swansea.ac.uk}

\maketitle

\pagenumbering{arabic}
%%%%%%%%%%%%%%%%%%%%%%%%%%%%%%%%%%%%%%%%%%%%%%%%%%%%%%%%%
%Overview of the MPT
%\section{Overview of the MPT}

\section*{Abstract}
Our interest lies in the identification of hidden conducting permeable objects from measurements of the perturbed magnetic field in metal detection taken over range of low frequencies. The magnetic polarizability tensor (MPT) provides a characterisation of a conducting permeable object using a small number of coefficients, has explicit formula for their calculation and a well understood frequency behaviour, which we call its spectral signature. However, to compute such signatures, and build a library of them for object classification, requires repeated solution of a direct (full order) problem, which is typically accomplished using a finite element discretisation. To overcome this issue, we propose an efficient reduced order model (ROM) using a proper orthogonal decomposition (POD) for the rapid computation of MPT spectral signatures. Our ROM benefits from output certificates, which give bounds on the accuracy of the predicted outputs with respect to the full order model solutions. To further increase the efficiency of the computation of the MPT spectral signature, we provide scaling results, which enable an immediate calculation of the signature under changes in the object size or conductivity. We illustrate our approach by application to a range of homogenous and inhomogeneous conducting permeable objects.

{\bf Keywords} Metal detection; Magnetic polarizability tensor; Reduced order model; Object classification.

{\bf MSC CLASSIFICATION} 65N30; 35R30; 35B30

\section{Introduction}

There is considerable interest in using the magnetic polarizability tensor (MPT) characterisation of conducting permeable objects to classify and identify hidden targets in metal detection. The MPT is a complex symmetric rank 2 tensor, which has $6$ independent coefficients, although the number of independent coefficients for objects with rotational or reflectional symmetries is smaller~\cite{LedgerLionheart2015}. Its coefficients are a function of the exciting frequency, the object's size, its shape  as well as its conductivity and permeability.
Explicit formulae for computing the tensor coefficients have been derived~\cite{Ammari2014,LedgerLionheart2015,LedgerLionheart2018,LedgerLionheart2019} and validated against exact solutions and measurements~\cite{LedgerLionheart2016,LedgerLionheart2018}. Also, the way in which the tensor coefficients vary with the exciting frequency is theoretically well understood~\cite{LedgerLionheart2019} offering improved object classification.

The frequency (or spectral) behaviour of the MPT, henceforth called its spectral signature, has been exploited in a range of different classification algorithms including simple library classification for homogeneous~\cite{Ammari2015} and inhomogeneous objects~\cite{LedgerLionheartamad2019}, a $k$ nearest neighbours (KNN) classification algorithm~\cite{Makkonen2014} and machine learning approaches~\cite{WoutervanVerre2019}.  The MPT classification of objects has already been applied to a range of different applications including  airport security screening~\cite{marsh2014,Makkonen2014}, waste sorting~\cite{karimian2017} and anti-personal landmine detection~\cite{rehim2016}. 
The aforementioned {\em supervised} classification techniques rely on a library of MPT spectral signatures to {\em learn} how to classify the objects. The purpose of this paper is to describe an efficient computational tool for computing this library.

One approach to obtaining a library of spectral signatures is to use a metal detector or dedicated measurement device
to obtain MPT coefficients of different objects~\cite{zhao2016,zhao2014,rehim2015}, however, to do so, over a range of frequencies for a large number of objects, is time consuming and will result in unavoidable measurement errors and noise. Therefore, there is considerable interest in their automated computation. By post processing finite element method (FEM) solutions to eddy current problems using  commercial packages (e.g  with ANSYS as in~\cite{rehim2015}) MPT coefficients can be obtained, however, improved accuracy, and a better understanding, can be gained by using the available explicit expressions for MPT coefficients, which rely on computing finite element (FE) approximations to a transmission problem~\cite{LedgerLionheart2015,LedgerLionheart2018,LedgerLionheart2019}. Nevertheless, to produce an accurate MPT spectral signature,  the process must be repeated for a large number of excitation frequencies leading to potentially expensive computations for fine discretisations (with small mesh spacing and high order elements).  The present paper addresses this issue by proposing a reduced order model, in the form of a (projected) proper orthogonal decomposition (POD) scheme, that relies on full order model solutions computed using the  established open source FE package, \texttt{NGSolve}, and the recently derived alternative explicit expressions formulae for the MPT coefficients~\cite{LedgerLionheart2019}. The use of \texttt{NGSolve}~\cite{NGSolve,netgendet} ensures that the solutions to underlying (eddy current type) transmission problems are accurately computed using high order $\bm{H}(\text{curl})$ conforming  (high order edge element) discretisations (see~\cite{ledgerzaglmayr2010,SchoberlZaglmayr2005,zaglmayrphd} and references therein) and the POD technique ensures their rapid computation over sweeps of frequency.

Reduced order models (ROMs) based on POD have been successfully applied to efficiently generate solutions for new problem parameters using a small number full order model snapshots in a range of engineering applications including mechanics~\cite{niroomandi2010model,radermacher2016pod}, thermal problems \cite{wang2012comparative,bialecki2005proper}, fluid flow \cite{luo2011reduced,pettit2002application} as well as electromagnetic problems with application to integrated circuits \cite{kerler2017model}. However, they have not been applied to the computation of MPT spectral signatures. A review of current POD techniques is provided in~\cite{hesthaven2016,Chatterjee2000}.

The main novelty of the work is the application of a POD approach for the efficient and accurate computation of the MPT spectral signature and the derivation of output certificates that ensure accuracy of the reduced order predictions. This ROM approach is motivated by the previous success of POD approaches and the theoretical study~\cite{LedgerLionheart2019}, which shows the spectral behaviour of the MPT is characterised by a small number of functions and, hence, has a sparse representation. The practical computation requires only computing full order model solution snapshots at a small number of frequencies and the evaluation of the MPT spectral signature follows from solving a series of extremely small linear systems. A second novelty is the presentation of simple scaling results, which enable the MPT spectral signature to easily be computed from an existing set of coefficients under the scaling of an object's conductivity or object size. 

The paper is organised as follows: In Section~\ref{sect:eddycurrent}, the eddy current model, which applies in metal detection, and the asymptotic expansion of the perturbed magnetic field in the presence of a conducting permeable object, which leads to the explicit expression of the MPT, is briefly reviewed. Then, in Section~\ref
{sect:fullorder}, the FE model used for full order model problem is described. Section~\ref{ROM} presents the POD reduced order model scheme. This is followed, in Section~\ref{sect:scaling}, by the derivation of results that describe the scaling of the MPT under parameter changes. Sections~\ref{sect:examplespodp} and~\ref{sect:examplesscale} present numerical examples of the POD scheme for computing the frequency behaviour of the MPT and examples of the scaling of the MPT under parameter changes, respectively.

%%%%%%%%%%%%%%%%%%%%%%%%%%%%%%%%%%%%%%%%%%%%%%%%%%%%%%%%%
%Eddy current model
\section{The eddy-current model and asymptotic expansion}\label{sect:eddycurrent}

We briefly discuss the eddy-current model along with stating the asymptotic expansion that forms the basis of the magnetic polarizability description of conducting objects in metal detection.
\subsection{Eddy-current model}
The eddy current model is a low frequency approximation of the Maxwell system that neglects the displacement currents, which is valid when the frequency is small and the conductivity of the body is high. A rigorous justification of the model involves the topology of the conducting body~\cite{ammaribuffa2000}.  The eddy current model is described by the system

\begin{subequations}\label{Eddy Current}
\begin{align}
\nabla\times\bm{E}_{\alpha}&=\im \omega\mu\bm{H}_{\alpha},\\
\nabla\times\bm{H}_{\alpha}&=\bm{J}_0+\sigma\bm{E}_{\alpha}.
\end{align}
\end{subequations}
where $\bm{E}_{\alpha}$ and $\bm{H}_{\alpha}$ are the electric and magnetic interaction fields, respectively, $ \bm{J}_0$ is an external current source, $\im:=\sqrt{-1}$, $\omega$ is the angular frequency, $\mu$ is the magnetic permeability and $\sigma$ is the electric conductivity.  We will use the eddy current model for describing the forward and inverse problems  associated with metal detection.

\subsubsection{Forward problem} \label{sect:forward}
In the forward (or direct) problem, the position and materials of the conducting body $B_{\alpha}$ are known. The object has a high conductivity, $\sigma=\sigma_*$, and a permeability, $\mu=\mu_*$. For the purpose of this study, the conducting body is assumed to be buried in soil, which is assumed to be of a much lower conductivity so that $\sigma \approx 0$ and have a permeability $\mu=\mu_0:= 4 \pi \times 10^{-7}\text{H/m}$.  A background field is generated by a solenodial current source $\bm{J}_0$ with support in the air above the soil, which also has $\sigma=0$ and $\mu =\mu_0$. The region around the object is $B_{\alpha}^c\vcentcolon=\mathbb{R}^3\setminus B_{\alpha}$ as shown  in Figure \ref{metal detection}. Note that a similar model also applies in the situation of identifying hidden targets in security screening~\cite{marsh2014,Makkonen2014} and waste sorting ~\cite{karimian2017} amongst others.

%Model setup diagram
\begin{figure}[H]
\begin{center}
\includegraphics[width=0.7\textwidth, keepaspectratio]{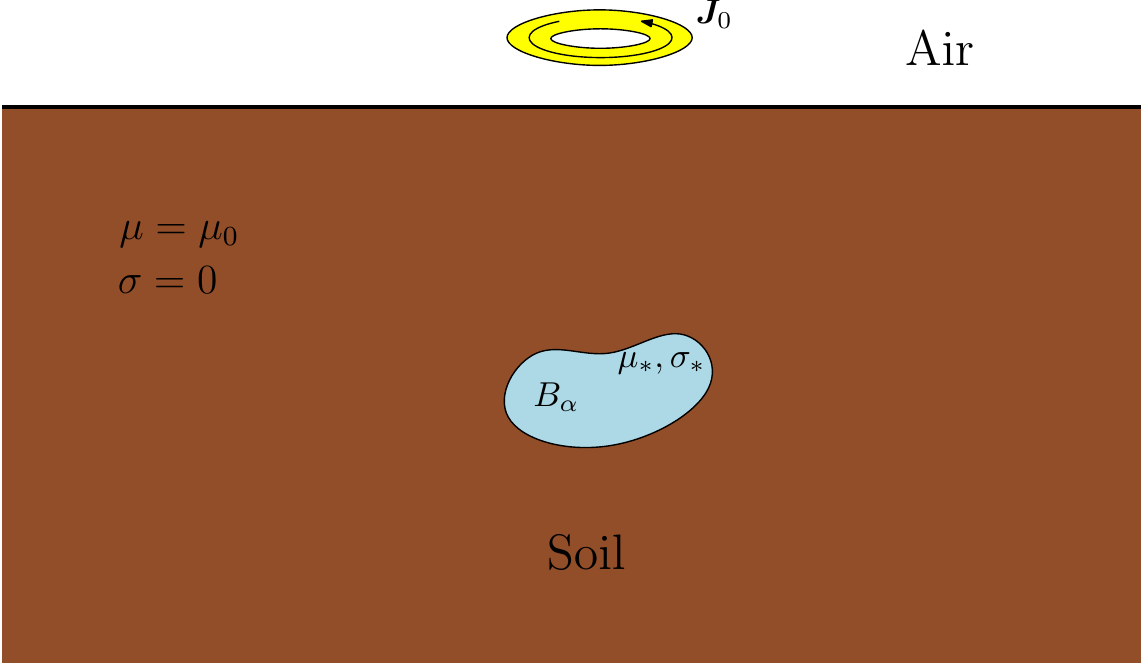}
\caption{A diagram showing a hidden conducting object $B_{\alpha}$, buried in soil, with a current source located in the air above.}
\label{metal detection}
\end{center}
\end{figure}
The forward model is described by the system (\ref{Eddy Current}),  which holds in ${\mathbb R}^3$, with
\begin{align}
\mu (\bm{x}) = \left \{  \begin{array}{ll} \mu_* & \bm{x} \in B_{\alpha} \\
\mu_0 & \bm{x} \in B_{\alpha}^c \end{array} \right . , \qquad 
\sigma (\bm{x}) = \left \{  \begin{array}{ll} \sigma_* & \bm{x} \in B_{\alpha} \\
0 & \bm{x} \in B_{\alpha}^c \end{array} \right .,
\end{align}
and the regions $B_\alpha$  and $B_\alpha^c $ are coupled by the transmission conditions
\begin{align}
\left [\bm{n} \times \bm{E}_\alpha\right ]_{\Gamma_{\alpha}}=\left [\bm{n} \times \bm{H}_\alpha\right ]_{\Gamma_{\alpha}}=\bm{0},\label{jump}
\end{align}
\noindent
which hold on  $\Gamma_\alpha:= \partial B_\alpha$. In the above, $[u ]_{\Gamma_{\alpha}}:= u| _+ - u|_-  $ denotes the jump,  the $+$ refers to just outside of $B_\alpha$ and the $-$ to just inside and $\bm{n}$ denotes a unit outward normal to  $\Gamma_{\alpha}$. 

The electric interaction field is non-physical in $B_\alpha^c$ and, to ensure uniqueness of this field, the condition $\nabla \cdot \bm{E}_\alpha =0$ is imposed in this region. Furthermore,  we also require that $\bm{E}_{\alpha}=O(1/|\bm{x}|)$ and $\bm{H}_{\alpha}=O(1/|\bm{x}|)$  as $|\bm{x} | \to \infty$, denoting that the fields go to zero at least as fast as $1/|\bm{x}|$, although, in practice, this rate can be faster.

\subsubsection{Inverse problem} \label{sect:inverseproblem}
In metal detection, the  inverse problem is to determine the location, shape and material properties ($\sigma_*$ and $\mu_*$) of the conducting object $B_\alpha$ from measurements of $(\bm{H}_\alpha - \bm{H}_0) (\bm{x})$ taken at a range of  locations $\bm{x}$ in the air. As described in the introduction, there are considerable advantages in using spectral data, i.e. additionally measuring $(\bm{H}_\alpha - \bm{H}_0) (\bm{x})$  over a range of frequencies $\omega$, within the limit of the eddy current model. Here, $\bm{H}_0$ denotes the background magnetic field and
$\bm{E}_0$ and $\bm{H}_0$ are the solutions of (\ref{Eddy Current}) with $\sigma =0$ and $\mu=\mu_0$ in ${\mathbb R}^3$. Similar to above, we also require the decay conditions $\bm{E}_{0}=O(1/|\bm{x}|)$ and $\bm{H}_{0}=O(1/|\bm{x}|)$  as $|\bm{x} | \to \infty$. Note that practical metal detectors measure a voltage perturbation, which corresponds to $\int_S \bm{n} \cdot (\bm{H}_\alpha - \bm{H}_0) (\bm{x}) \dif \bm{x}$ over an appropriate surface $S$~\cite{LedgerLionheart2018}. For very small coils, this voltage perturbation is approximated by $\bm{m} \cdot (\bm{H}_\alpha - \bm{H}_0) (\bm{x})$ where $\bm{m}$ is the magnetic dipole moment of the coil~\cite{LedgerLionheart2018}.

A traditional approach to the solution of this inverse problem involves creating a discrete set of voxels, each with unknown $\sigma$ and $\mu$, and posing the solution to the inverse problem as an optimisation process in which $\sigma $ and $\mu$ are found through minimisation of an appropriate functional e.g.~\cite{manuch2006}. From the resulting images of $\sigma$ and $\mu$ one then attempts to infer the shape and position of the object.
However,
this problem is highly ill-posed~\cite{brown2016} and presents
considerable challenges mathematically and computationally in the case of
limited noisy measurement data.

Instead, we seek an approximation of the perturbation $(\bm{H}_{\alpha}-\bm{H}_0)( \bm{x})$ at some point $\bm{x}$ exterior to $B_\alpha$, which allows objects to be characterised by a small number of coefficients in a MPT that are easily obtained from the measurements of $(\bm{H}_\alpha - \bm{H}_0) (\bm{x})$ once the object position is known, which can be found from a MUSIC algorithm for example~\cite{Ammari2014}. The object identification then reduces to a classification problem, as discussed in the introduction.

\subsection{The asymptotic expansion and MPT description}
%The forward problem described in Section~\ref{sect:forward}, implies that if we know $B_\alpha$, $ \sigma_*$ and $\mu_*$ we can solve (\ref{Eddy Current})  to uniquely determine $\bm{E}_\alpha$ and $\bm{H}_\alpha$. However, to do this repeatably for each new 
%$B_\alpha$, $ \sigma_*$ and $\mu_*$, such as required in an inverse problem cast as an optimisation problem,
%would be computationally expensive. Instead, we seek an approximation of the perturbation $(\bm{H}_{\alpha}-\bm{H}_0)( \bm{x})$ at some point $\bm{x}$ exterior to $B_\alpha$, which allows objects to be characterised by a smaller number of coefficients in a MPT. Thus, leading to classification approaches discussed in the introduction.

Following~\cite{Ammari2014,LedgerLionheart2015} we define $B_\alpha := \alpha B + \bm{z}$ where $B$ is a unit size object, $\alpha$ is the object size and $\bm{z}$ is the object's translation from the origin as shown in Figure ~\ref{BAlphaShifted}.
\begin{figure}[H]
\begin{center}
\includegraphics[width=0.8\textwidth, keepaspectratio]{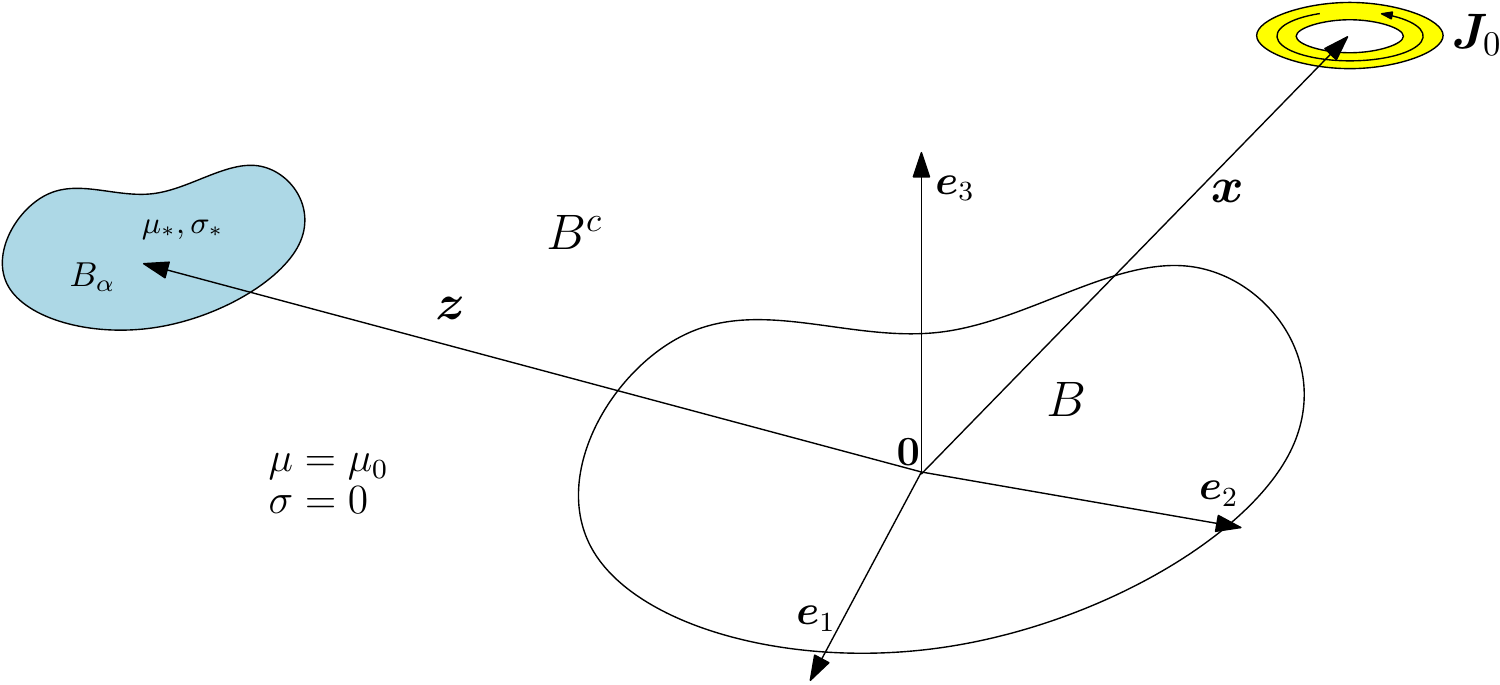}
\caption{A diagram showing the physical description of $B_{\alpha}$ with respect to the coordinate axes.}
\label{BAlphaShifted}
\end{center}
\end{figure}
\noindent Then, using the asymptotic formula obtained by Ammari, Chen, Chen, Garnier and Volkov~\cite{Ammari2014}, Ledger and Lionheart~ \cite{LedgerLionheart2015} have derived the simplified form
 %Asymptotic formula
\begin{align}
(\bm{H}_{\alpha}-\bm{H}_0)(\bm{x})_i=(\bm{D}_{\bm{x}}^2G(\bm{x},\bm{z}))_{ij}(\mathcal{M})_{jk}(\bm{H}_0(\bm{z}))_k+O(\alpha^4), \label{eqn:asymp}
\end{align}
which holds as $\alpha\to 0$ and makes the MPT explicit. The relationship between the leading order term in the above  to the dipole expansion of $(\bm{H}_{\alpha}-\bm{H}_0)(\bm{x})$ is discussed in~\cite{LedgerLionheart2018}.
 In the above, $G(\bm{x} ,\bm{z}  ) := 1/ {4\pi | \bm{x} -\bm{z} |}$ is the free space Laplace Green's function, $ \bm{D}^2_x G$ denotes the Hessian of $G$ and Einstein summation convention of the indices is implied. In addition, ${\mathcal M}=({\mathcal M})_{jk} {\bm e}_j \otimes {\bm e}_k$, where $\bm{e}_i$ denotes the $i$th orthonormal unit vector,  is the symmetric rank 2 MPT, which describes the shape and material properties of the object $B_\alpha $ and is frequency dependent, but is independent of the object's position $\bm{z}$. We will sometimes write $\mathcal{M}[ \alpha B, \omega ]$ to emphasise this. The above formulation, and the definition of ${\mathcal M}$ below, are presented for the case of a single homogenous object $B$, the extension to multiple inhomogeneous objects can be found in~\cite{LedgerLionheartamad2019,LedgerLionheart2019}.

%presented in \cite{LedgerLionheart2015} since it evaluates the object by means of a complex symmetric rank 2 tensor which is invariant relative to position in the field. The expansion also includes an error term, giving a measure of the accuracy of the approximation and is given as follows. We may evaluate the perturbation $(\bm{H}_{\alpha}-\bm{H}_0)$ at the point $\bm{x}$ caused by the object $B_{\alpha}$ as,
%
%which holds as $\alpha\rightarrow 0$. Let us also note the term $\bm{R}(\bm{x})=O(\alpha^4)$ which is the error term we mentioned before. It is also useful for us to note at this point our physical description of $B_{\alpha}$, we define $B_{\alpha}=\alpha B+\bm{z}$ a unit sized object $B$ scaled by $\alpha$ and translated by $\bm{z}$ shown in Figure \ref{BAlphaShifted}.
%B\alpha diagram

Using the derivation in~\cite{LedgerLionheart2019}, we state the explicit formulae for the computation of the coefficients of $\mathcal{M}$, which
are particularly well suited to a FEM discretisation. The earlier explicit expressions in~\cite{LedgerLionheart2015,LedgerLionheart2016,LedgerLionheart2018} are equivalent for exact fields.  We use the splitting $(\mathcal{M})_{ij}:=(\mathcal{N}^0)_{ij}+(\mathcal{R})_{ij}+\im(  \mathcal{I})_{ij}$ obtained in ~\cite{LedgerLionheart2019} with
%Tensor definitions
\begin{subequations}
\label{eqn:NRI}
\begin{align}
(\mathcal{N}^0[ \alpha B] )_{ij}&:=\alpha^3\delta_{ij}\int_{B}(1-\mu_r^{-1})\dif \bm{\xi}+\frac{\alpha^3}{4}\int_{B\cup B^c}\tilde{\mu}_r^{-1}\nabla\times\tilde{\bm{\theta}}_i^{(0)}\cdot\nabla\times\tilde{\bm{\theta}}_j^{(0)}\ \dif \bm{\xi},\\
(\mathcal{R}[\alpha B, \omega])_{ij}&:=-\frac{\alpha^3}{4}\int_{B\cup B^c}\tilde{\mu}_r^{-1}\nabla\times\bm{\theta}_j^{(1)}\cdot\nabla\times\overline{\bm{\theta}_i^{(1)}}\ \dif \bm{\xi},\\
(\mathcal{I}[\alpha B, \omega])_{ij}&:=\frac{\alpha^3}{4}\int_B\nu\Big(\bm{\theta}_j^{(1)}+(\tilde{\bm{\theta}}_j^{(0)}+\bm{e}_j\times\bm{\xi})\Big)\cdot\Big(\overline{\bm{\theta}_i^{(1)}+(\tilde{\bm{\theta}}_i^{(0)}+\bm{e}_i\times\bm{\xi})}\Big)\ \dif \bm{\xi},
\end{align}
\end{subequations}
where $\mathcal{N}^0[ \alpha B]$, $\mathcal{R}[\alpha B, \omega]$ and $\mathcal{I}[\alpha B, \omega]$ are real symmetric rank 2 tensors, which each have real eigenvalues.
In the above,
\begin{align}
\tilde{\mu}_r ( \bm{\xi} ) := \left \{ \begin{array}{ll}  \mu_r :=\mu_*/\mu_0 & \bm{\xi} \in B\\
1 & \bm{\xi} \in B^c \end{array} \right . \nonumber,
\end{align}
and $\nu:=\alpha^2\omega\mu_0\sigma_*$, $\delta_{ij}$ is the Kronecker delta and the overbar denotes the complex conjugate. The computation of (\ref{eqn:NRI}) rely on the solution of the transmission problems~\cite{LedgerLionheart2019}
\begin{subequations}
\label{eqn:Theta0}
\begin{align}
\nabla\times\tilde{\mu}_r^{-1}\nabla\times\bm{\theta}_i^{(0)}&=\bm{0} &&\textrm{in }B\cup B^c,\\
\nabla\cdot\bm{\theta}_i^{(0)}&=0 &&\textrm{in }B\cup B^c,\\
[{\bm{n}}\times\bm{\theta}_i^{(0)}]_{\Gamma}&=\bm{0} &&\textrm{on }\Gamma,\\
[{\bm{n}}\times\tilde{\mu}_r^{-1}\nabla\times\bm{\theta}_i^{(0)}]_{\Gamma}&=\bm{0} &&\textrm{on }\Gamma,\\
\bm{\theta}_i^{(0)}-{\bm{e}}_i\times\bm{\xi}&=\bm{O}(|\bm{\xi}|^{-1}) &&\textrm{as }|\bm{\xi}|\rightarrow\infty,
\end{align}
\end{subequations}
where $\Gamma:=\partial B$ and 
%Theta 1 transmission problem
\begin{subequations}
\label{eqn:Theta1}
\begin{align}
\nabla\times {\mu}_r^{-1}\nabla\times\bm{\theta}_i^{(1)}-\im \nu  (\bm{\theta}_i^{(0)}+\bm{\theta}_i^{(1)})&=\bm{0}&&\textrm{in }B,\\
\nabla\times \nabla\times\bm{\theta}_i^{(1)} &=\bm{0}&&\textrm{in }B^c,\\
\nabla\cdot\bm{\theta}_i^{(1)}&=0&&\textrm{in }B^c,\\
[{\bm{n}}\times\bm{\theta}_i^{(1)}]_{\Gamma}&=\bm{0}&&\textrm{on }\Gamma,\\
[{\bm{n}}\times\tilde{\mu}_r^{-1}\nabla\times\bm{\theta}_i^{(1)}]_{\Gamma}&=\bm{0}&&\textrm{on }\Gamma,\\
\bm{\theta}_i^{(1)}&=\bm{O}(|\bm{\xi}|^{-1})&&\textrm{as }|\bm{\xi}|\rightarrow\infty.
\end{align}
\end{subequations}
Note also that we choose to introduce $\tilde{\bm{\theta}}_i^{(0)}\vcentcolon=\bm{\theta}_i^{(0)}-{\bm{e}}_i\times\bm{\xi}$, which can be shown to satisfy the same transmission problem as (\ref{eqn:Theta0}) except with a non-zero jump condition for $[{\bm{n}}\times\tilde{\mu}_r^{-1}\nabla\times\tilde{\bm{\theta}}_i^{(0)}]_{\Gamma}$ and the decay condition $\tilde{\bm{\theta}}_i^{(0)}(\bm{\xi})=\bm{O}(|\bm{\xi}|^{-1})$ as $|\bm{\xi}|\rightarrow\infty$.

\section{Full order model}\label{sect:fullorder}
To approximate the solutions to the transmission problems (\ref{eqn:Theta0}) and (\ref{eqn:Theta1}) we truncate the unbounded domain $B^c$ at a finite distance from the object $B$ and create a bounded domain $\Omega$ containing $B$. On $\partial \Omega$, we approximate the decay conditions (\ref{eqn:Theta0}e) and (\ref{eqn:Theta1}f) by  $\bm{n} \times   \tilde{\bm{\theta}}_i^{(0)}= \bm{n} \times (  {\bm{\theta}}_i^{(0)} - \bm{e}_i \times \bm{\xi})=\bm{0}$ and  $\bm{n} \times {\bm{\theta}}_i^{(1)} =\bm{0}$, respectively. On this finite domain, we approximate the associated weak variational statements to these problems using  FEM with a $\bm{H}(\text{curl})$ conforming discretisation with mesh spacing $h$ and order elements $p$ where
\begin{equation}
\bm{H}(\text{curl}) :=\left \{ {\bm{u}} :{\bm{u}} \in  (L^2(\Omega))^3, \ \nabla \times {\bm{u}} \in  (L^2(\Omega))^3 \right \},
\end{equation}
and $L^2(\Omega)$ denotes the standard space of square integrable functions.
In Section~\ref{sect:weak} we provide their weak formulations and provide their discretisation in Section~\ref{sect:fem}. Henceforth, we call this discrete approximation the full order model.
%%%%%%%%%%%%%%%%%%%%%%%%%%%%%%%%%%%%%%%%%%%%%%%%%%%%%%%%%
%Weak form
\subsection{Weak formulation of the problem} \label{sect:weak}

Following the approach advocated in~\cite{ledgerzaglmayr2010} for magnetostatic and eddy current problems, we add a regularisation term $\varepsilon \int_{\Omega} \tilde{\bm{\theta}}_i^{(0)} \cdot \bm{\psi} \dif \bm{\xi}$, where $\varepsilon$ is a small regularisation parameter,  to the weak variational statement of (\ref{eqn:Theta0}), written in terms of $\tilde{\bm{\theta}}_i^{(0)}$, in order to circumvent the Coulomb gauge $\nabla\cdot\tilde{\bm{\theta}}_i^{(0)}=0$. 
For details of the small error induced by this approximation see~\cite{ledgerzaglmayr2010,zaglmayrphd}.
Then, by choosing an appropriate set of $\bm{H}(\text{curl})$ conforming finite element functions in $W^{(hp)} \subset \bm{H}(\text{curl})$, we obtain the following discrete regularised weak form for (\ref{eqn:Theta0}) : Find real solutions $\tilde{\bm{\theta}}_i^{(0,hp)} \in Y^{\varepsilon}\cap W^{(hp)} $ such that
%Weak form of Theta 0 (discrete)
\begin{align}\label{Weak0}
\int_{\Omega} \tilde{\mu}_r^{-1} \nabla \times \tilde{\bm{\theta}}_i^{(0,hp)} \cdot \nabla \times \bm{\psi}^{(hp)} \dif \bm{\xi}  &+ \varepsilon \int_\Omega \tilde{\bm{\theta}}_i^{(0,hp)} \cdot \bm{\psi}^{(hp)} \dif \bm{\xi}\nonumber\\
&=2 \int_B(1-\mu_r^{-1}) \bm{e}_i \cdot \nabla \times \bm{\psi}^{(hp)} \dif \bm{\xi},
\end{align}
for all $\bm{\psi}^{(hp)} \in Y^{\varepsilon} \cap W^{(hp)}$, where 
$$Y^{\varepsilon} = \Big\{ \bm{u} \in \bm{H}(\text{curl}) : {\bm{n}} \times {\bm{u}} =\bm{0} \textrm{ on } \partial \Omega  \Big\}.$$
In a similar manner, the discrete weak variational statement of (\ref{eqn:Theta1}) is:  Find complex solutions ${\bm{\theta}}_i^{(1,hp)} \in Y^{\varepsilon}\cap W^{(hp)} $ such that
%Weak form of Theta 1 (discrete)
\begin{align}\label{Weak1}
\int_{\Omega}\big(\mu_r^{-1}\nabla\times\bm{\theta}_i^{(1,hp)}\big)&\cdot\big(\nabla\times\overline{\bm{\psi}^{(hp)}}\big) \dif \bm{\xi}-\im \int_{B}\nu\bm{\theta}_i^{(1,hp)}\cdot\overline{\bm{\psi}^{(hp)}} \dif\bm{\xi}\nonumber\\
&+\varepsilon\int_{\Omega\setminus B}\bm{\theta}_i^{(1,hp)}\cdot\overline{\bm{\psi}^{(hp)}}\dif \bm{\xi}=\im \int_B\nu\bm{\theta}_i^{(0,hp)}\cdot\overline{\bm{\psi}^{(hp)}} \dif\bm{\xi},
\end{align}
for all $\bm{\psi}^{(hp)} \in Y^{\varepsilon} \cap W^{(hp)}$ where the overbar denotes the complex conjugate.

For what follows it is beneficial to restate (\ref{Weak1}) in the following form: Find ${\bm{\theta}}_i^{(1,hp)} \in Y^{\varepsilon}\cap W^{(hp)}$ such that
%Redefining in a bilinear form
\begin{equation}\label{bilinear}
a\big(\bm{\theta}_i^{(1,hp)},\bm{\psi}^{(hp)};\bm{\omega}\big)=r\big(\bm{\psi}^{(hp)}; \bm{\theta}_i^{(0,hp)} ,\bm{\omega}\big),
\end{equation}
for all $\bm{\psi}^{(hp)} \in Y^{\varepsilon} \cap W^{(hp)}$ where
\begin{subequations}\label{eqn:BilinearExpanded}
\begin{align}
a\big(\bm{\theta}_i^{(1,hp)},\bm{\psi}^{(hp)};\bm{\omega}\big)\vcentcolon&=
\left < \tilde{\mu}^{-1} \nabla\times\bm{\theta}_i^{(1,hp)}, \nabla\times {\bm{\psi}^{(hp)}} \right>_{L^2(\Omega)} \nonumber\\
&- \im \left <  \nu\bm{\theta}_i^{(1,hp)} , {\bm{\psi}^{(hp)}} \right >_{L^2(B)} \nonumber \\
&+ \varepsilon \left < \bm{\theta}_i^{(1,hp)}, \bm{\psi}^{(hp)} \right >_{L^2(\Omega\setminus B)}, \\
r\big(\bm{\psi}^{(hp)};\bm{\theta}_i^{(0,hp)}, \bm{\omega}\big)\vcentcolon&=
\im \left < \nu \bm{\theta}_i^{(0,hp)}, {\bm{\psi}^{(hp)}} \right >_{L^2(B)},
\end{align}
\end{subequations}
$ \left < \bm{u},\bm{v} \right >_{L^2(\Omega)} := \int_\Omega {\bm{u}} \cdot \overline{\bm{v}} \dif \bm {\xi}$ denotes the $L^2$ inner product over $\Omega$
and $\bm{\omega}$ indicates the list of the problem parameters $\{\omega,\sigma_*, \mu_r,\alpha\}$ that one might wish to vary. Note that $r\big(\cdot ;\cdot, \cdot \big)$ is a function of $\mu_r$ as $\bm{\theta}_i^{(0,hp)}$ depends on $\mu_r$.
%%%%%%%%%%%%%%%%%%%%%%%%%%%%%%%%%%%%%%%%%%%%%%%%%%%%%%%%%
%Finite element implementation
\subsection{Finite element discretisation} \label{sect:fem}
For the implementation of the full order model, we use \texttt{NGSolve} \cite{NGSolvecode,NGSolve,netgendet,zaglmayrphd} along with the hierarchic set of $\bm{H}(\text{curl})$ conforming basis functions  proposed by Sch\"{o}berl and Zaglmayr~\cite{SchoberlZaglmayr2005}, which are available in this software. In the following, for simplicity, we focus on the treatment of $\bm{\theta}_i^{(1,hp)}$ and drop the index $i$ as each direction can be computed in a similar way (as can $\tilde{\bm{\theta}}_i^{(0,hp)}$).
We denote these basis functions with $\bm{N}^{(k)}(\bm{\xi})\in W^{(hp)}$ leading to the expression of the solution function along with the weighting functions
%Theta in terms of q
\begin{subequations}\label{FE deconstruct}
\begin{align}
\bm{\theta}^{(1,hp)}(\bm{\xi},\bm{\omega})&\vcentcolon=\sum_{k=1}^{N_d}\bm{N}^{(k)}(\bm{\xi})\mathrm{q}_k(\bm{\omega}),\\
\bm{\psi}^{(hp)}(\bm{\xi},\bm{\omega})&\vcentcolon=\sum_{k=1}^{N_d}\bm{N}^{(k)}(\bm{\xi})\mathrm{l}_k(\bm{\omega}),
\end{align}
\end{subequations}
where $N_d$ is the number of degrees of freedom. Here, and in the following,  the bold italic font
denotes a vector  field and the bold non-italic Roman font represents a matrix (upper case) or column vector (lower case). With this distinction, we rewrite (\ref{FE deconstruct}) in matrix form as
%matrix form of theta
\begin{subequations}\label{solution breakdown}
\begin{align}
\bm{\theta}^{(1,hp)}(\bm{\xi},\bm{\omega})&=\textbf{N}(\bm{\xi})\textbf{q}(\bm{\omega}),\\
\bm{\psi}^{(hp)}(\bm{\xi},\bm{\omega})&=\textbf{N}(\bm{\xi})\textbf{l}(\bm{\omega}),
\end{align}
\end{subequations}
where $\textbf{N}(\bm{\xi})$ is the matrix constructed with the basis vectors $\bm{N}^k(\bm{\xi})$ as its columns, i.e.
%N matrix of basis functions
$$\textbf{N}(\bm{\xi})\vcentcolon=\big [\bm{N}^{(1)}(\bm{\xi}),\bm{N}^{(2)}(\bm{\xi}),...,\bm{N}^{(N_{d})} (\bm{\xi} )\big ].$$
With this, we may also rewrite (\ref{bilinear}) as follows
%sumation version of the bilinear form
\begin{equation}\label{eqn:basisbilinear}
\sum_{i=1}^{N_d}\sum_{j=1}^{N_d}\overline{l_i(\bm{\omega})}a\big(\bm{N}^{(j)}(\bm{\xi}),\bm{N}^{(i)}(\bm{\xi});\bm{\omega}\big)q_j(\bm{\omega})=\sum_{i=1}^{N_d}\overline{l_i(\bm{\omega})}r\big(\bm{N}^{(i)}(\bm{\xi}); \bm{\theta}^{(0,hp)},  \bm{\omega}\big),
\end{equation}
and, with a suitable choice of $l_i(\bm{\omega})$, we may rewrite (\ref{eqn:basisbilinear}) as the linear system of equations
%Matrix form of the linear set of equations
\begin{equation}\label{eqn:Linear}
\textbf{A}(\bm{\omega})\textbf{q}(\bm{\omega})=\textbf{r}(\bm{\theta}^{(0,hp)} , \bm{\omega}),
\end{equation}
where the coefficients of $\bf{A}(\bm{\omega})$ and ${\bf r}({\bm \theta}^{(0,hp)}, \bm{\omega})$ are defined to be
\begin{subequations}
\label{eqn:definestiffandrhs}
\begin{align}
(\textbf{A}(\bm{\omega}))_{ij}&\vcentcolon=a\big(\bm{N}^{(j)}(\bm{\xi}),\bm{N}^{(i)}(\bm{\xi});\bm{\omega}\big),\\
(\textbf{r}( \bm{\theta}^{(0,hp)}, \bm{\omega}))_{i}&\vcentcolon=r\big(\bm{N}^{(i)}(\bm{\xi}); \bm{\theta}^{(0,hp)}, \bm{\omega}\big).
\end{align}
\end{subequations}
\texttt{NGSolve} offers efficient approaches for the computational solution to (\ref{eqn:Linear}) using preconditioned iterative solvers~\cite{zaglmayrphd,ledgerzaglmayr2010}, which we exploit. Following the solution of  (\ref{eqn:Linear}), we can obtain $\bm{\theta}^{(1,hp)}(\bm{\xi},\bm{\omega})$ using (\ref{solution breakdown}) and, by repeating the process for $i=1,2,3$, we get  $\bm{\theta}_i^{(1,hp)}(\bm{\xi},\bm{\omega})$. Then $(\mathcal{M}[\alpha B, \omega, \mu_r, \sigma_*])_{ij}$, for the full order model, is found by using (\ref{eqn:NRI}).

\section{Reduced order model (ROM)}\label{ROM}

A traditional approach for the computation of the MPT spectral signature, i.e. the variation of the coefficients ${\mathcal M}[\alpha B, \omega]$  with frequency,  would involve the repeated solution of the $N_d$ sized system (\ref{eqn:Linear}) for different $\omega$. To reduce the computational cost of this, we wish to apply a  ROM in which the solution of (\ref{eqn:Linear}) is replaced by a surrogate problem of reduced size. Thus, reducing both the computation cost and time to produce a solution for each new $\omega$.  In particular, in Section~\ref{POD},  we describe a ROM based on the POD method ~\cite{Chatterjee2000,AbdiWilliams2010,hesthaven2016,Seoane2019} and, in Section~\ref{sect:podp}, apply the variant called projection based POD (which we denote by PODP), which has already been shown to work well in the analysis of magneto-mechanical coupling applied to MRI scanners~\cite{Seoane2019}. 
To emphasise the generality of the approach, the formulation is presented for an arbitrary list of problem parameters denoted by $\bm{\omega}$. 
In Section~\ref{sect:outputcert} we derive a procedure for computing certificates of accuracy on the ROM solutions with negligible additional cost.

%%%%%%%%%%%%%%%%%%%%%%%%%%%%%%%%%%%%%%%%%%%%%%%%%%%%%%%%%%
%Proper orthogonal decomposition
\subsection{Proper orthogonal decomposition}\label{POD}
Following the solution of (\ref{eqn:Linear}) for  $\mathbf{q}(\bm{\omega})$  for different values of the set of parameters, $\bm{\omega}$, we construct a matrix $\mathbf{D}\in\mathbb{C}^{N_d\times N}$ with the vector of solution coefficients as its columns in the form
%Definition of the matrix of snapshots
\begin{equation}\label{D}
\mathbf{D}\vcentcolon=\big[\mathbf{q}(\bm{\omega}_1),\mathbf{q}(\bm{\omega}_2),...,\mathbf{q}(\bm{\omega}_{N})\big],
\end{equation}
where $N\ll N_d$ denote the number of snapshots. Application of a singular value decomposition (SVD)~e.g.~\cite{bjorck,hansen} gives
%Definition of the SVD
\begin{equation}\label{eq:SVD}
\mathbf{D}=\mathbf{U\Sigma V}^*,\end{equation}
where $\mathbf{U}\in\mathbb{C}^{N_d\times N_d}$ and $\mathbf{V}^*\in\mathbb{C}^{N\times N}$ are unitary matrices and $\mathbf{\Sigma}\in\mathbb{R}^{N_d\times N}$ is a diagonal matrix enlarged by zeros so that it becomes rectangular. In the above, $\mathbf{V}^*= \overline{\mathbf{V}}^T$ is the hermitian of $\mathbf{V}$.

The diagonal entries $(\mathbf{\Sigma})_{ii}=\sigma_i$~\footnote{Note that $\sigma_*$ is used for conductivity and $\sigma_i$ for a singular value, however, it should be clear from the application as to which definition applies} are the singular values of $\mathbf{D}$ and they are arranged as  $\sigma_1>\sigma_2>...>\sigma_{N}$. Based on the sparse representation of the solutions to (\ref{eqn:Theta1}) as function of $\nu$, and hence $\omega$,
 (and hence also the sparse representation of the MPT) found in~\cite{LedgerLionheart2019},  we expect these to decay rapidly towards zero, which motivates the introduction of 
 the truncated singular value decomposition (TSVD)~e.g.~\cite{bjorck,hansen}
\begin{equation}\label{eq:truncSVD}
\mathbf{D}\approx
\mathbf{D}^M = \mathbf{U}^M\mathbf{\Sigma}^M(\mathbf{V}^M)^*,
\end{equation}
where $\mathbf{U}^M\in\mathbb{C}^{N_d\times M} $ are the first $M$ columns of $\mathbf{U}$, $\mathbf{\Sigma}^M\in{\mathbb R}^{M\times M}$ is a diagonal matrix containing the first $M$ singular values and $(\mathbf{V}^M)^*\in{\mathbb C}^{M\times N}$ are the first $M$ rows of $\mathbf{V}^*$. The computation of (\ref{eq:truncSVD}) constitutes the off-line stage of the POD.
Using (\ref{eq:truncSVD}) we can recover an approximate representation for each of our solution snapshots as follows
\begin{equation}
\mathbf{q}(\bm{\omega}_j)\approx\mathbf{U}^M\mathbf{\Sigma}^M((\mathbf{V}^M)^*)_j,
\end{equation}
where $((\mathbf{V}^M)^*)_j$ refers to the $j$th column of $(\mathbf{V}^M)^*$.
%%%%%%%%%%%%%%%%%%%%%%%%%%%%%%%%%%%%%%%%%%%%%%%%%%%%%%%%%
%%%%%%%%%%%%%%%%%%%%%%%%%%%%%%%%%%%%%%%%%%%%%%%%%%%%%%%%%
%Projection based proper orthogonal decomposition
\subsection{Projection based proper orthogonal decomposition (PODP)} \label{sect:podp}
 In the online stage of PODP, $\mathbf{q}^{PODP} ( {\bm \omega}) \approx\mathbf{q}({\bm \omega})$  is obtained  by taking a linear combination of the columns of ${\bf U}^M$ where the coefficients of this projection are contained in the vector ${\bf p}^M$. We choose to also approximate $\mathbf{l}({\bm \omega})$ in a similar way so that
\begin{subequations}\label{eqn:solution:rebreakdown}
\begin{align}
\bm{\theta}^{(1,hp)}(\bm{\xi},\bm{\omega}) \approx(\bm{\theta}^{(1,hp)})^{\text{PODP}} ({\bm \xi}, \bm{\omega)} :=& \textbf{N}(\bm{\xi})
\mathbf{q}^{PODP} ( {\bm \omega}) = 
\textbf{N}(\bm{\xi}) \mathbf{U}^M\textbf{p}^M( \bm{\omega})   \in Y^{(PODP)} , \\
\bm{\psi}^{(hp)}(\bm{\xi},\bm{\omega})\approx(\bm{\psi}^{(1,hp)})^{\text{PODP}} ({\bm \xi}, \bm{\omega)} :=&
\textbf{N}(\bm{\xi})\textbf{l}^{PODP} (\bm{\omega})=
\textbf{N}(\bm{\xi})\mathbf{U}^M\textbf{o}^M(\bm{\omega})  \in Y^{(PODP)},
\end{align}
\end{subequations}
where $ \in Y^{(PODP)}\subset Y^\varepsilon \cap W^{(hp)}$. 
Substituting these  lower dimensional representations in to (\ref{eqn:basisbilinear}) we obtain the following
\begin{align}
\sum_{i=1}^{M}\sum_{j=1}^{M}\overline{o^M_i(\bm{\omega})}a\big(\bm{N}^{(j)}(\bm{\xi})( \mathbf{U}^{M})_j &,\bm{N}^{(i)}(\bm{\xi})(\mathbf{U}^{M})_i;\bm{\omega}\big)p^M_j(\bm{\omega})\nonumber\\
&=\sum_{i=1}^{M}\overline{o^M_i(\bm{\omega})}r\big(\bm{N}^{(i)}(\bm{\xi})( \mathbf{U}^{M})_i; \bm{\theta}^{(0,hp)},\bm{\omega}\big), \nonumber \\
({\mathbf{o}^M} (\bm{\omega}))^* ( (\mathbf{U}^M)^*\mathbf{A}(\bm{\omega}) \mathbf{U}^M )\mathbf{p}^M ( \bm{\omega}) & =
({\mathbf{o}^M}(\bm{\omega}))^*(\mathbf{U}^M)^*
\mathbf{r}(\bm{\theta}^{(0,hp)},  \bm{\omega}).
\label{eqn:basisbilinearPOD}
\end{align}
Then, if we choose $\mathbf{o}^M (\bm{\omega})$ appropriately, we obtain the linear system
\begin{equation}\label{eqn:ReducedA}
\mathbf{A}^M(\bm{\omega})\mathbf{p}^M(  \bm{\omega})=\mathbf{r}^M(\bm{\theta}^{(0,hp)}, \bm{\omega}),
\end{equation}
which is of size $M\times M$ where $\mathbf{A}^M(\bm{\omega})\vcentcolon=(\mathbf{U}^M)^*\mathbf{A}(\bm{\omega})\mathbf{U}^M$ and $\mathbf{r}^M(\bm{\theta}^{(0,hp)} , \bm{\omega})\vcentcolon=(\mathbf{U}^M)^*\mathbf{r} ( \bm{\theta}^{(0,hp)}, \bm{\omega})$. Note, since $M<N \ll N_d$, this is significantly smaller than (\ref{eqn:Linear}) and, therefore, substantially computationally cheaper to solve. After solving this reduced system, and obtaining $\mathbf{p}^M(\bm{\omega})$, we obtain an approximate solution for $\bm{\theta}^{(1,hp)}(\bm{\xi},\bm{\omega})$ using (\ref{eqn:solution:rebreakdown}).

Focusing on the particular case where $\bm{\omega}=\omega$,  from (\ref{eqn:BilinearExpanded}) we observe that we can express  $\mathbf{A}$ and $\mathbf{r}$ as the simple sums
\begin{align}
\mathbf{A}(\omega)=& \mathbf{A}^{(0)} + \omega  \mathbf{A}^{(1)} , \nonumber \\
\mathbf{r}(\bm{\theta}^{(0,hp)},\omega)=&\omega \mathbf{r}^{(1)}(\bm{\theta}^{(0,hp)}) , \nonumber
\end{align}
where the definitions of $ \mathbf{A}^{(0)}$, $ \mathbf{A}^{(1)}$ and $\mathbf{r}^{(1)}(\bm{\theta}^{(0,hp)}) $ are obvious from (\ref{eqn:definestiffandrhs}),(\ref{eqn:BilinearExpanded}) and the definition of $\nu$.
Then, by computing and storing $(\mathbf{U}^M)^*\mathbf{A}^{(0)}  \mathbf{U}^M$,  $(\mathbf{U}^M)^*\mathbf{A}^{(1)}  \mathbf{U}^M$ , $(\mathbf{U}^M)^*\mathbf{r}^{(1)}(\bm{\theta}^{(0,hp)})$, which are independent of $\omega$, it follows that $\mathbf{A}^M({\omega})$ and $\mathbf{r}^M(\bm{\theta}^{(0,hp)},{\omega})$ can be efficiently calculated for each new $\omega$ from the stored data. In a similar manner, by precomputing appropriate data, the MPT coefficients in (\ref{eqn:NRI}) can also be rapidly evaluated for each new $\omega$ using the PODP solutions. This leads to further considerable computational savings. We emphasise that the PODP is only applied to obtain ROM solutions for $\bm{\theta}^{(1)}(\bm{\xi},{\omega})$ and not to $\bm{\theta}^{(0)}(\bm{\xi})$, which does not depend on $\omega$.

\subsection{PODP output certification} \label{sect:outputcert}

We follow the approach described in~\cite{hesthaven2016}, which enables us to derive and compute certificates of accuracy on the MPT coefficients obtained with PODP, with respect to those obtained with full order model, as a function of $\omega$. To do this, we set $\bm{\epsilon}_i (\omega)= \bm{\theta}_i^{(1,hp)} (\omega)- (\bm{\theta}_i^{(1,hp)})^{\text{PODP}} (\omega) \in Y^{(hp)}$, where we have reintroduced the subscript $i$, as we need to distinguish between the cases $i=1,2,3$. Although $\bm{\epsilon}_i$ also depends on $\bm{\xi}$, we  have chosen here, and in the following, to only emphasise its dependence on $\omega$. We have also introduced $Y^{(hp)}= Y^\varepsilon \cap W^{(hp)}$ for simplicity of notation, and note that this error satisfies
\begin{align}
a(\bm{\epsilon}_i (\omega), \bm{\psi}; {\omega} ) =r (\bm{\psi};\bm{\theta}_i^{(0,hp)} ,\omega) \qquad \forall \bm{\psi} \in Y^{(hp)}, \label{eqn:erroreqn}
\end{align}
which is called the error equation~\cite{hesthaven2016} and
\begin{align}
a(\bm{\epsilon}_i (\omega) , \bm{\psi} ; \omega) =0 \qquad \forall \bm{\psi} \in Y^{(PODP)},
\end{align}
which is called Galerkin orthogonality~\cite{hesthaven2016}. The Riesz representation~\cite{hesthaven2016} of $r (\cdot ;\bm{\theta}_i^{(0,hp)}, \omega)$ denoted by $\hat{\bm{r}}_i(\omega) \in Y^{(hp)}$ is such that
\begin{align}
(\hat{\bm{r}}_i (\omega) , \bm{\psi} )_{Y^{(hp)}} =r(\bm{\psi};\bm{\theta}_i^{(0,hp)},\omega) \qquad \forall \bm{\psi} \in Y^{(hp)}, \label{eqn:riesz}
\end{align}
so that 
\begin{align}
a(\bm{\epsilon}_i(\omega), \bm{\psi}; {\omega} ) =(\hat{\bm{r}}_i(\omega) , \bm{\psi} )_{Y^{(hp)}} \qquad \forall \bm{\psi} \in Y^{(hp)}.
\end{align}
Then, by using the alternative set of formulae  for the tensor coefficients~\cite{LedgerLionheart2019}
\begin{subequations} \label{eqn:tensoraltform}
\begin{align}
(\mathcal{R}[\alpha B, \omega])_{ij}&=-\frac{\alpha^3}{4} \int_{B} \nu\text{Im} ( \bm{\theta}_j^{(1,hp)} ) \cdot \bm{\theta}_i^{(0,hp)} \dif \bm{\xi}  =  -\frac{\alpha^3}{4} \left <  \nu\text{Im} ( \bm{\theta}_j^{(1,hp)} ), \bm{\theta}_i^{(0,hp)}  \right >_{L^2(B)},
\\
(\mathcal{I}[\alpha B, \omega])_{ij}&= \frac{\alpha^3}{4} \left ( 
\int_{B} \nu\text{Re} ( \bm{\theta}_j^{(1,hp)} ) \cdot \bm{\theta}_i^{(0,hp)}
 \dif \bm{\xi} + \int_{B} \nu \bm{\theta}_j^{(0,hp)} \cdot \bm{\theta}_i^{(0,hp)}\dif \bm{\xi} 
\right ) \nonumber \\
& = \frac{\alpha^3}{4}\left (  \left <  \nu\text{Re} ( \bm{\theta}_j^{(1,hp)} ), \bm{\theta}_i^{(0,hp)}  \right >_{L^2(B)} + \left <  \nu
 \bm{\theta}_j^{(0,hp)}
, \bm{\theta}_i^{(0,hp)}  \right >_{L^2(B)} \right ),
\end{align}
\end{subequations}
written in terms of the full order solutions,
we obtain the certificates for the tensor entries computed using PODP stated in the lemma below. Note that the  formulae stated in (\ref{eqn:NRI})  are used for the actual POD computation of $(\mathcal{R}^{PODP}[\alpha B, \omega])_{ij}$ and $(\mathcal{I}^{PODP}[\alpha B, \omega])_{ij}$, but the form in (\ref{eqn:tensoraltform}) is useful for obtaining certificates. Also, as $(\mathcal{N}[\alpha B])_{ij}$ is independent of $\omega$ we have 
$(\mathcal{N}^{0,PODP}[\alpha B])_{ij}= (\mathcal{N}^{0 }[\alpha B])_{ij}$ and we write $\mathcal{M}^{PODP}[\alpha B, \omega] = \mathcal{N}^{0,PODP}[\alpha B]+ \mathcal{R}^{PODP}[\alpha B, \omega]+ \im \mathcal{I}^{PODP}[\alpha B, \omega]$ for the MPT obtained by PODP.

\begin{lemma}
An error certificate for the tensor coefficients computed using PODP is
\begin{subequations}
\label{eqn:certifcate}
\begin{align}
\left | (\mathcal{R}[\alpha B, \omega])_{ij} - (\mathcal{R}^{PODP}[\alpha B, \omega])_{ij}  \right |\le & (\Delta[\omega])_{ij} ,\\
\left | (\mathcal{I}[\alpha B, \omega])_{ij} - (\mathcal{I}^{PODP}[\alpha B, \omega])_{ij}  \right | \le &(\Delta[\omega])_{ij},
\end{align}
\end{subequations}
where 
\begin{align}
(\Delta[\omega])_{ij}: =  \frac{\alpha^3}{8\alpha_{LB}}
\left (  \| \hat{\bm{r}}_i (\omega) \|_{Y^{(hp)}}^2 + \| \hat{\bm{r}}_j (\omega) \|_{Y^{(hp)}}^2 + \| \hat{\bm{r}}_i (\omega) - \hat{\bm{r}}_j (\omega) \|_{Y^{(hp)}}^2 
\right ) ,  \nonumber
\end{align}
and $\alpha_{LB}$ is a lower bound on a stability constant.
\end{lemma}

\begin{proof}
We concentrate on the proof for $ \left | (\mathcal{R}[\alpha B, \omega])_{ij} - (\mathcal{R}^{PODP}[\alpha B, \omega])_{ij}  \right |$ as the proof  for the second bound is similar and leads to the same result. Recalling the symmetry of $\mathcal{R}[\alpha B, \omega]$,  we have $  (\mathcal{R}[\alpha B, \omega])_{ij} = \frac{1}{2} \left (
 (\mathcal{R}[\alpha B, \omega])_{ij} +  (\mathcal{R}[\alpha B, \omega])_{ji}  \right )$
 so that
 \begin{align}
D:=\left | (\mathcal{R}[\alpha B, \omega])_{ij} - (\mathcal{R}^{PODP}[\alpha B, \omega])_{ij}  \right | =&
 \frac{\alpha^3}{8} \left | \left <  \nu\text{Im} ( \bm{\epsilon}_i ), \bm{\theta}_j^{(0,hp)}  \right >_{L^2(B)}+  \left <  \nu\text{Im} ( \bm{\epsilon}_j ), \bm{\theta}_i^{(0,hp)}  \right >_{L^2(B)}
 \right | \nonumber \\
=& \frac{\alpha^3}{8} \left |
  \left <  \nu\text{Im} ( \bm{\epsilon}_i ), \bm{\theta}_i^{(0,hp)}   \right >_{L^2(B)}+
  \left <  \nu\text{Im} ( \bm{\epsilon}_i ),\bm{\theta}_j^{(0,hp)} -  \bm{\theta}_i^{(0,hp)}   \right >_{L^2(B)}+ \right . \nonumber\\
  &\left .
  \left <  \nu\text{Im} ( \bm{\epsilon}_j ), \bm{\theta}_j^{(0,hp)}  \right >_{L^2(B)} + 
  \left <  \nu\text{Im} ( \bm{\epsilon}_j ), \bm{\theta}_i^{(0,hp)} -\bm{\theta}_j^{(0,hp)}  \right >_{L^2(B)}
 \right | \nonumber \\
= & \frac{\alpha^3}{8} \left |
  \left <  \nu\text{Im} ( \bm{\epsilon}_i ), \bm{\theta}_i^{(0,hp)}   \right >_{L^2(B)}+
  \left <  \nu\text{Im} ( \bm{\epsilon}_i - \bm{\epsilon}_j ),\bm{\theta}_j^{(0,hp)} -  \bm{\theta}_i^{(0,hp)}   \right >_{L^2(B)}+ \right . \nonumber\\
  &\left .
  \left <  \nu\text{Im} ( \bm{\epsilon}_j ), \bm{\theta}_j^{(0,hp)}  \right >_{L^2(B)}  \right | \nonumber ,
  \end{align}
which follows since $\nu$ and $\bm{\theta}_i^{(0,hp)}$ are real valued and  where we have dropped the dependence of $\omega$ on $\bm{\epsilon}_i$  for simplicity of presentation.
  Thus,
 \begin{align}
 D\le  \frac{\alpha^3}{8}\left ( \left |
  \left <  \nu  \bm{\epsilon}_i , \bm{\theta}_i^{(0,hp)}   \right >_{L^2(B)} \right | + \left | \left <  \nu (\bm{\epsilon}_i - \bm{\epsilon}_j ),\bm{\theta}_j^{(0,hp)} -  \bm{\theta}_i^{(0,hp)}   \right >_{L^2(B)} \right | + \left |  \left <  \nu \bm{\epsilon}_j , \bm{\theta}_j^{(0,hp)}  \right >_{L^2(B)}  \right | \right ) .\nonumber
  \end{align}
  Next, using (\ref{eqn:erroreqn}),  we make the observation that
  \begin{align}
  \left |  \left <  \nu  \bm{\epsilon}_i , \bm{\theta}_i^{(0,hp)}   \right >_{L^2(B)} \right | = \left | r(\bm{\epsilon}_i; \bm{\theta}_i^{(0,hp)} ,\omega ) \right  | = \left | a(\bm{\epsilon}_i , \bm{\epsilon}_i  ; {\omega} )  \right | = \| \bm{\epsilon}_i  \|_\omega^2.
  \nonumber 
  \end{align}
  Also, since $r( \bm{\psi};  \bm{\theta}_j^{(0,hp)} - \bm{\theta}_i^{(0,hp)} ,\omega )=  a(  \bm{\theta}_j^{(1,hp)} (\omega) - \bm{\theta}_i^{(1,hp)} (\omega), \bm{\psi};\omega) =  a(  \bm{\epsilon}_j- \bm{\epsilon}_i, \bm{\psi};\omega) $ for all $\bm{\psi} \in Y^{(hp)}$,
  we have
  $  \left  | \left <  \nu  \bm{\psi} , \bm{\theta}_j^{(0,hp)} - \bm{\theta}_i^{(0,hp)}   \right >_{L^2(B)} \right | = \left | a(  \bm{\epsilon}_j- \bm{\epsilon}_i, \bm{\psi};\omega) \right | $ so that
  \begin{align}
  \left |  \left <  \nu ( \bm{\epsilon}_i -\bm{\epsilon}_j ) , \bm{\theta}_j^{(0,hp)} - \bm{\theta}_i^{(0,hp)}    \right >_{L^2(B)} \right | = \left | r(\bm{\epsilon}_i-\bm{\epsilon}_j;   \bm{\theta}_j^{(0,hp)} - \bm{\theta}_i^{(0,hp)}  , \omega ) \right  | = \left | a(\bm{\epsilon}_j -\bm{\epsilon}_i , \bm{\epsilon}_i  - \bm{\epsilon}_j  ; {\omega} )  \right | = \| \bm{\epsilon}_i - \bm{\epsilon}_j  \|_\omega^2,
  \nonumber 
  \end{align}
  and hence
  \begin{align}
   D \le  \frac{\alpha^3}{8}\left (  \| \bm{\epsilon}_i  \|_\omega^2 + \| \bm{\epsilon}_i - \bm{\epsilon}_j  \|_\omega^2 + \|  \bm{\epsilon}_j  \|_\omega^2\right ) \label{eqn:bderrors}.
  \end{align}
 Following similar steps to~\cite[pg47-50]{hesthaven2016}, and introducing a Riesz representation in (\ref{eqn:riesz}), we  can find that  
  \begin{align}
  \| \bm{\epsilon}_i  \|_\omega^2 \le \frac{\| \hat{\bm{r}}_i (\omega) \|_{Y^{(hp)}}^2}{\alpha_{LB}} , \qquad \| \bm{\epsilon}_j \|_\omega^2 \le \frac{ \| \hat{\bm{r}}_j (\omega) \|_{Y^{(hp)}}^2}{\alpha_{LB}} ,
   \qquad \| \bm{\epsilon}_i - \bm{\epsilon}_j \|_\omega^2 \le \frac{ \| \hat{\bm{r}}_i(\omega)-  \hat{\bm{r}}_j (\omega) \|_{Y^{(hp)}}^2}{\alpha_{LB}} 
   \nonumber ,
  \end{align}
  and, combining this with (\ref{eqn:bderrors}), completes the proof.
\end{proof}

The efficient evaluation of (\ref{eqn:certifcate}) follows the approach presented in~\cite[pg52-54]{hesthaven2016},  adapted to complex matrices and with the simplification that we compute a Riesz representation $ \hat{\bm{r}}_i(\omega) \in Y^{(h0)}$ using lowest order elements for computational efficiency. The computations are split in to those performed in the off-line stage and those in the on-line stage as follows.

In the off-line stage, the following $(2M+1) \times (2M+1) $ hermitian matrices are computed
\begin{align}
\mathbf{G}^{(i,j)} = \left ( \textbf{W}^{(i)} \right )^H \textbf{M}_0^{-1} \textbf{W}^{(j)} ,  \nonumber
\end{align}
where, since $\mathbf{G}^{(j,i)} = (\mathbf{G}^{(i,j)})^H$, it follows that, in practice, only the 3 matrices 
$\mathbf{G}^{(1,1)}$, $\mathbf{G}^{(2,2)}$ and $\mathbf{G}^{(3,3)}$
are required for computing the certificates on the diagonal entries of the tensors, and the  further 3 matrices
$\mathbf{G}^{(1,2)}$, $\mathbf{G}^{(1,3)}$ and $\mathbf{G}^{(2,3)}$ are needed
 for the off-diagonal terms. In the above, $(\textbf{M}_0)_{ij} = \left <  \bm{N}^{(i)} , \bm{N}^{(j)} \right >_{L^2(\Omega)}$ are the coefficient of a real symmetric FEM mass matrix for the lowest order, with $ \bm{N}^{(i)}, \bm{N}^{(j)}\in W^{(h0)}$ being typical lowest order basis functions, and 
\begin{align}
\textbf{W}^{(i)}: =\textbf{P}_0^p \left (\begin{array}{ccc} \mathbf{r}^{(1)}( \bm{\theta}_i^{(0)})  & \mathbf{A} ^{(0)} \mathbf{U}^{(M,i)}  &  \mathbf{A} ^{(1)} \mathbf{U}^{(M,i)} \end{array} \right), \nonumber
\end{align}
where $\textbf{P}_0^p$ is a projection matrix of the FEM basis functions from order $p$ to the lowest order $0$,  $ \mathbf{U}^{(M,i)}$ is the $\mathbf{U}^M$ obtained in (\ref{eq:truncSVD}) for the $i$th direction.
The stability constant $\alpha_{LB} =\lambda_{min}\text{min}(1,\frac{\omega}{\omega'})$ is obtained from the smallest eigenvalue of an eigenvalue problem~\cite[pg56]{hesthaven2016}, which, in practice, is only performed once for smallest frequency of interest $\omega'$.

In the on-line stage, we evaluate
\begin{align}
 \| \hat{\bm{r}}_i (\omega) \|_{Y^{(hp)}}^2 =& \left ( (\mathbf{w}^{(i)}(\omega))^H \mathbf{G}^{(i,i)} (\mathbf{w}^{(i)}(\omega)) \right )^{1/2} , \nonumber \\
 \| \hat{\bm{r}}_i (\omega)- \hat{\bm{r}}_j (\omega) \|_{Y^{(hp)}}^2 =& \left ( \| \hat{\bm{r}}_i (\omega) \|_{Y^{(hp)}}^2 +  \| \hat{\bm{r}}_j (\omega) \|_{Y^{(hp)}}^2
 - 2 \text{Re}( 
  (\mathbf{w}^{i}(\omega))^H \mathbf{G}^{(i,j)} (\mathbf{w}^{(j)}(\omega)) )\right )^{1/2} , \nonumber
\end{align}
for each $\omega$ by updating the vector
\begin{equation}
 \mathbf{w}^{(i)}(\omega) =\left ( \begin{array}{c}  \omega  \\- \mathbf{p}^M( {\omega}) \\ -\omega  \mathbf{p}^M({\omega}) \end{array} \right ). \nonumber
\end{equation}
We then  apply~(\ref{eqn:certifcate}) to obtain the output certificates.

%%%%%%%%%%%%%%%%%%%%%%%%%%%%%%%%%%%%%%%%%%%%%%%%%%%%%%%%%%%%%%%%
\section{Scaling of the MPT under parameter changes}\label{sect:scaling}
%%%%%%%%%%%%%%%%%%%%%%%%%%%%%%%%%%%%%%%%%%%%%%%%%%%%%%%%%%%%%%%%
Two results that aid the computation of the frequency sweep of an MPT for an object with scaled conductivity and an object with a scaled object size from an already known frequency sweep of the MPT for the same shaped object are stated below.
%%%%%%%%%%%%%%%%%%%%%%%%%%%%%%%%%%%%%%%%%%%%%%%%%%%%%%%%%%%%%%%%
\begin{lemma} \label{lemma:condscale}
%%%%%%%%%%%%%%%%%%%%%%%%%%%%%%%%%%%%%%%%%%%%%%%%%%%%%%%%%%%%%%%%
Given the MPT coefficients for an object $ \alpha B $ with material parameters $ \mu_r$  and $ \sigma_*$ at frequency $s\omega$, the  coefficients of the MPT for an object, which has the same $B$, $\alpha $ and $\mu_r$, but with conductivity $s\sigma_*  $, at frequency $\omega$, are given by
\begin{align}
({\mathcal M}[\alpha B, \omega , \mu_r ,s\sigma_*])_{ij} =& ( {\mathcal M}[\alpha B, s\omega , \mu_r ,\sigma_*] )_{ij}, \label{eqn:condscale}
\end{align}
where $( {\mathcal M}[\alpha B, s\omega , \mu_r ,\sigma_*] )_{ij}$ denote the coefficients of the original MPT at frequency $s\omega$.
%%%%%%%%%%%%%%%%%%%%%%%%%%%%%%%%%%%%%%%%%%%%%%%%%%%%%%%%%%%%%%%%
\end{lemma}
%%%%%%%%%%%%%%%%%%%%%%%%%%%%%%%%%%%%%%%%%%%%%%%%%%%%%%%%%%%%%%%%
\begin{proof}
This result immediately follows from (\ref{eqn:NRI}) and (\ref{eqn:Theta1}) since both are written in terms of $\nu=\alpha^2 \sigma_*\mu_0 \omega$.
\end{proof}
%%%%%%%%%%%%%%%%%%%%%%%%%%%%%%%%%%%%%%%%%%%%%%%%%%%%%%%%%%%%%%%%

%%%%%%%%%%%%%%%%%%%%%%%%%%%%%%%%%%%%%%%%%%%%%%%%%%%%%%%%%%%%%%%%
\begin{lemma}\label{lemma:alphascale}
%%%%%%%%%%%%%%%%%%%%%%%%%%%%%%%%%%%%%%%%%%%%%%%%%%%%%%%%%%%%%%%%
Given the MPT coefficients for an object $ \alpha B $ with material parameters $ \mu_r$  and $ \sigma_*$ at frequency $s^2\omega$, the  coefficients of the MPT for an object $s \alpha B $, which is the same as $B$ apart from having size $s\alpha$, at frequency $\omega$, are given by
\begin{align}
({\mathcal M}[s \alpha B, \omega , \mu_r ,\sigma_*])_{ij} =& s^3 ({\mathcal M}[\alpha B, s^2\omega , \mu_r ,\sigma_*] )_{ij},\label{eqn:alphascale}
\end{align}
where $( {\mathcal M}[\alpha B, s^2\omega , \mu_r ,\sigma_*] )_{ij}$ denote the coefficients of the original MPT at frequency $s^2\omega$.
%%%%%%%%%%%%%%%%%%%%%%%%%%%%%%%%%%%%%%%%%%%%%%%%%%%%%%%%%%%%%%%%
\end{lemma}
%%%%%%%%%%%%%%%%%%%%%%%%%%%%%%%%%%%%%%%%%%%%%%%%%%%%%%%%%%%%%%%%
\begin{proof}
For the case of $\mu_r=1$ this result was proved by Ammari {\em et al.}~\cite{Ammari2015}. We generalise this to $0<\mu_r<\mu_r^{max}<\infty$  as follows: We use the splitting $(\mathcal{M})_{ij}:=(\mathcal{N}^0)_{ij}-(\mathcal{C}^{\sigma_*})_{ij} + (  \mathcal{N}^{\sigma_*})_{ij}$ presented in~\cite{LedgerLionheart2016} and let $\bm{\theta}_{i,B}^{(0)}$ denote the solution to (\ref{eqn:Theta0}).  Then, we find that 
\begin{equation}
\frac{1}{s} \bm{\theta}_{i,sB}^{(0)} (s\bm{\xi}') = \bm{\theta}_{i,B}^{(0)} (\bm{\xi}'), \nonumber
\end{equation}
where $\bm{\theta}_{i,sB}^{(0)} $ is the solution to  (\ref{eqn:Theta0}) with $B$ replaced by $sB$. If $ \bm{\theta}_{i,B}^{(1)}[s^2\nu]$ is the solution to (\ref{eqn:Theta1}) with $\nu$ replaced by $s^2\nu$, then, we find that
\begin{equation}
\frac{1}{s} \bm{\theta}_{i,sB}^{(1)}  [\nu ](s\bm{\xi}') = \bm{\theta}_{i,B}^{(1)}[ s^2\nu] (\bm{\xi}') , \nonumber
\end{equation}
where $\bm{\theta}_{i,sB}^{(1)}  [\nu ]$ is the solution to (\ref{eqn:Theta1}) with $B$ replaced by $sB$. Using the above, the definitions in  Lemma 1 of~\cite{LedgerLionheart2016},  and $\bm{\xi} = s {\bm \xi}'$ we find
\begin{align}
({\mathcal C}^{\sigma_*}[ \alpha (sB), \omega, \mu_r,\sigma_* ] )_{ij} = & -  \frac{\im  \alpha^3\nu }{4}  \int_{sB} \bm{e}_i \cdot 
\left ( \bm{\xi} \times \left (
\bm{\theta}_{i,sB}^{(1)}[\nu]+ \bm{\theta}_{i,sB}^{(0)} \right ) \right ) \dif \bm{\xi} \nonumber\\
= & \frac{\im  s^3 \alpha^3\nu }{4} \int_{B}
 \bm{e}_i \cdot
 \left ( s \bm{\xi}' \times \left (
\bm{\theta}_{i,sB}^{(1)}[\nu](s \bm{\xi}')+ \bm{\theta}_{i,sB}^{(0)} (s \bm{\xi}') \right ) \right ) \dif \bm{\xi}'\nonumber\\
= & \frac{\im  s^3 \alpha^3(s^2\nu) }{4} \int_{B}
 \bm{e}_i \cdot
 \left ( \bm{\xi}' \times \left (
\bm{\theta}_{i,B}^{(1)}[s^2\nu] + \bm{\theta}_{i,B}^{(0)}  \right ) \right ) \dif \bm{\xi}'
= s^3 ({\mathcal C}^{\sigma_*} [ \alpha B, s^2\omega, \mu_r,\sigma_*] )_{ij} , \nonumber
\end{align}
\begin{align}
({\mathcal N}^0[ \alpha (sB), \mu_r] )_{ij} = &  \frac{ \alpha^3}{2} [\tilde{\mu}^{-1} ]_\Gamma \int_{sB} \bm{e}_i \cdot \nabla_\xi \times 
\bm{\theta}_{i,sB}^{(0)}  \dif \bm{\xi} \nonumber\\
= & \frac{ s^3 \alpha^3}{2} [\tilde{\mu}^{-1} ]_\Gamma \int_{B} \bm{e}_i \cdot\frac{1}{s} \nabla_{\xi'} \times (s
\bm{\theta}_{i,B}^{(0)} ) \dif \bm{\xi}' = s^3 ({\mathcal N}^0[ \alpha B, \mu_r] )_{ij} ,  \nonumber
\end{align}
\begin{align}
({\mathcal N}^{\sigma_*}[ \alpha (sB), \omega, \mu_r,\sigma_* ] )_{ij} = &  \frac{ \alpha^3}{2} [\tilde{\mu}^{-1} ]_\Gamma \int_{sB} \bm{e}_i \cdot \nabla_\xi \times 
\bm{\theta}_{i,sB}^{(1)}[\nu]  \dif \bm{\xi} \nonumber\\
= & \frac{ s^3 \alpha^3}{2} [\tilde{\mu}^{-1} ]_\Gamma \int_{B} \bm{e}_i \cdot\frac{1}{s} \nabla_{\xi'} \times (s
\bm{\theta}_{i,B}^{(1)} [s^2 \nu] ) \dif \bm{\xi}' = s^3 ({\mathcal N}^{\sigma_*} [ \alpha B, s^2\omega, \mu_r,\sigma_*] )_{ij} , \nonumber
\end{align} 
and the quoted result immediately follows.
\end{proof}
%%%%%%%%%%%%%%%%%%%%%%%%%%%%%%%%%%%%%%%%%%%%%%%%%%%%%%%%%%%%%%%%

%%%%%%%%%%%%%%%%%%%%%%%%%%%%%%%%%%%%%%%%%%%%%%%%%%%%%%%%%%%%%%%%
\section{Numerical examples of PODP} \label{sect:examplespodp}
%%%%%%%%%%%%%%%%%%%%%%%%%%%%%%%%%%%%%%%%%%%%%%%%%%%%%%%%%%%%%%%%
The PODP algorithm has been implemented in the Python interface to the high order finite element solver \texttt{NGSolve} package led by the group of Sch\"oberl~\cite{NGSolve,zaglmayrphd,netgendet} available at \texttt{https://ngsolve.org}. The snapshots are computed by solving (\ref{Weak0}) and (\ref{bilinear}) using \texttt{NGSolve} and their $\bm{H}(\text{curl})$ conforming tetrahedral finite element basis functions of order $p$ on meshes of spacing $h$~\cite{SchoberlZaglmayr2005}. Following the solution of (\ref{eqn:Linear}), and the application of (\ref{FE deconstruct}), the coefficients of ${\mathcal M}[\alpha B, \omega]$~\footnote{
In the following, when presenting numerical results for the PODP, we frequently choose to drop the superscript PODP  on  $\mathcal{M}[\alpha B,\omega]$  $\mathcal{R}[\alpha B,\omega]$, $\mathcal{I}[\alpha B,\omega]$ and $\mathcal{N}^0[B]$, introduced in Section~\ref{sect:outputcert}, for brevity of presentation where no confusion arises.
Also, we will return to using the notation ${\mathcal M}[\alpha B, \omega,\mu_r,\sigma_*]$, which illustrates the full parameter dependence, in Section~\ref{sect:examplesscale} when considering scaling of conductivity and object size.}
 follow by simple post-processing of (\ref{eqn:NRI}). If desired, PODP output certificates can also be efficiently computed using the approach described in Section~\ref{sect:outputcert}.
 %In the following we present MPT spectral signatures in the form of the variation of the eigenvalues of its real and imaginary parts with frequency.
  The Python scripts for the computations presented can be accessed at \\\texttt{https://github.com/BAWilson94/MPT-Calculator}.

%%%%%%%%%%%%%%%%%%%%%%%%%%%%%%%%%%%%%%%%%%%%%%%%%%%%%%%%%%%%%%%%
\subsection{Conducting permeable sphere}\label{sect:ConductingSphere}
%%%%%%%%%%%%%%%%%%%%%%%%%%%%%%%%%%%%%%%%%%%%%%%%%%%%%%%%%%%%%%%%
We begin with the case where $B_{\alpha}=\alpha B$  is a permeable conducting sphere of radius $\alpha=0.01$ m and  $B$ is the unit sphere centred at the origin. The sphere is chosen to have a relative permeability $\mu_r=1.5$ and conductivity $\sigma_*=5.96\times10^6$ S/m. To produce the snapshots of the full order model, we set $\Omega$ to be a ball 100 times the radius of $B$ and discretize it using a mesh of 26\,385 unstructured tetrahedral elements, refined towards the object, and a polynomial order of $p=3$. We have chosen this discretization since it has already been found to produce an accurate representation of $\mathcal{M}[\alpha B, \omega]$ for $10^2 <\omega<10^8$ rad/s by comparing with exact solution of the MPT for a sphere~\cite{wait1951,LedgerLionheart2018}. Indeed, provided that the geometry discretisation error is under control, performing $p$-refinement of the full order model solution results in exponential convergence to the true solution~\cite{LedgerLionheart2015}. 

We follow two different schemes for choosing frequencies $\omega$ for generating the solution vectors $\mathbf{q}(\omega)$ required for $\mathbf{D}$ in (\ref{D}).
Firstly, we consider linearly spaced frequencies $\omega_{min}\le \omega_{n}\le \omega_{max}$,  $n=1,2,\ldots,N$, where, as in Section~\ref{POD}, $N$ is the number of snapshots, and
 denote this choice of samples by ``Lin" in the results. Secondly, we consider logarithmically spaced frequencies $\omega_{min}\le \omega_{n}\le \omega_{max}$ 
and  denote this regime by ``Log" in the results.
 
Considering both linearly and logarithmically spaced frequencies  with $\omega_{min}= 1\times 10^2 \text{ rad/s}$, $\omega_{max}= 1\times 10^8 \text{ rad/s}$ and $N=9,13,17$, in turn, to generate the snapshots, the  application of an SVD to  $\mathbf{D}$ in (\ref{eq:SVD}) leads to the results shown in Figure~\ref{fig:Singular} where the values have been scaled by $\sigma_1$ and are strictly decreasing.  We observe that ``Log'' case produces singular values $\sigma_i/\sigma_1$, which tend to $0$ with increasing $i$, while 
the ``Lin'' case produces $\sigma_i/\sigma_1$, which tend to a finite constant with increasing $i$. Also shown is the tolerance $TOL=1\times 10^{-3}$,  i.e. we define $M$ such that $\sigma_M/\sigma_1\leq TOL<\sigma_{M+1}/\sigma_1$ and create the matrices $\mathbf{U}^M$, $\mathbf{\Sigma}^M$ and $(\mathbf{V}^*)^M$ by taking the first $M$ columns of $\mathbf{U}$, $M$ rows of $\mathbf{V}^*$ and first $M$ rows and columns of $\mathbf{\Sigma}$.
 
 \begin{figure}[H]
\begin{center}
$$\begin{array}{cc}
\includegraphics[width=0.5\textwidth, keepaspectratio]{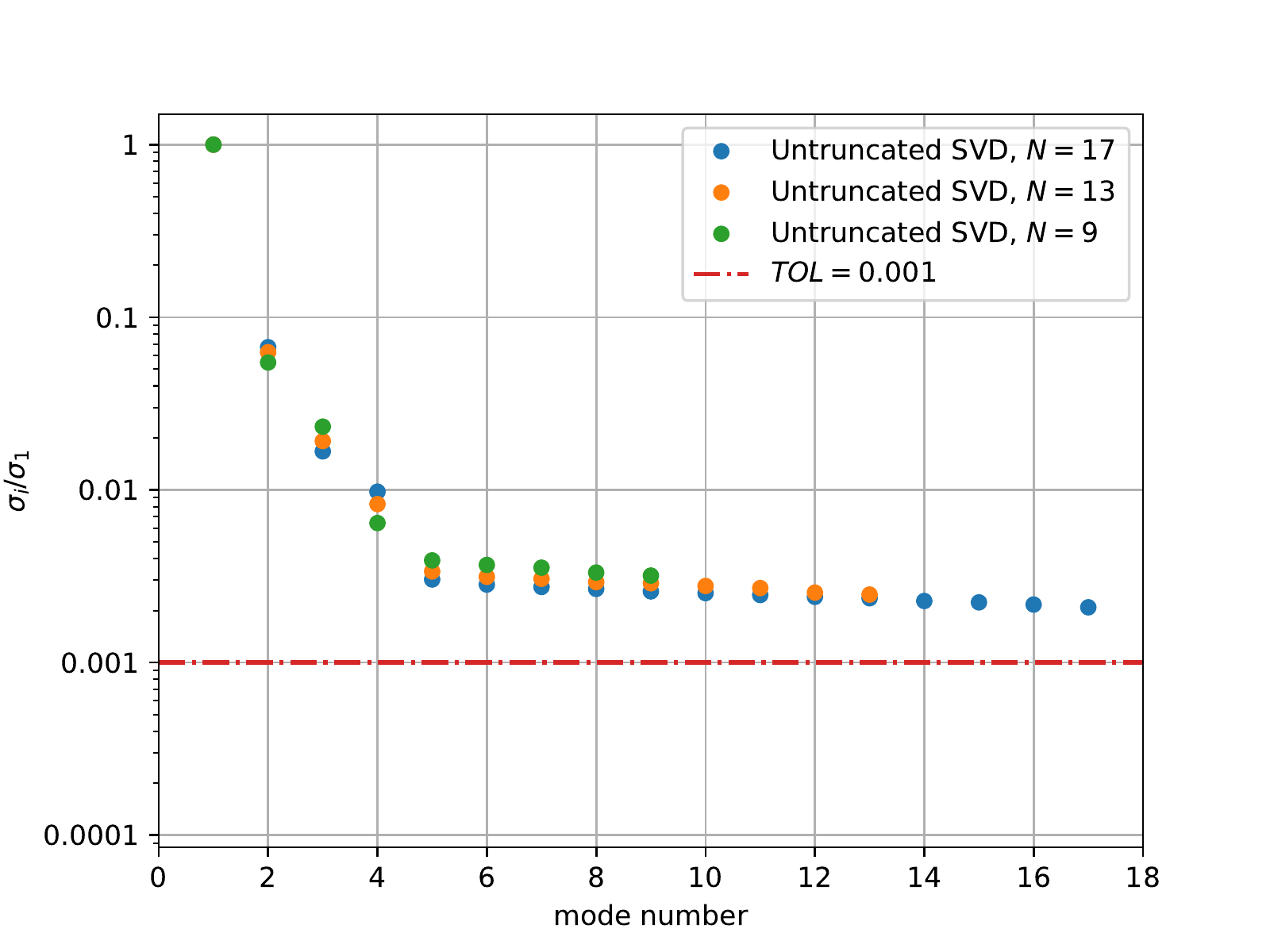} &
\includegraphics[width=0.5\textwidth, keepaspectratio]{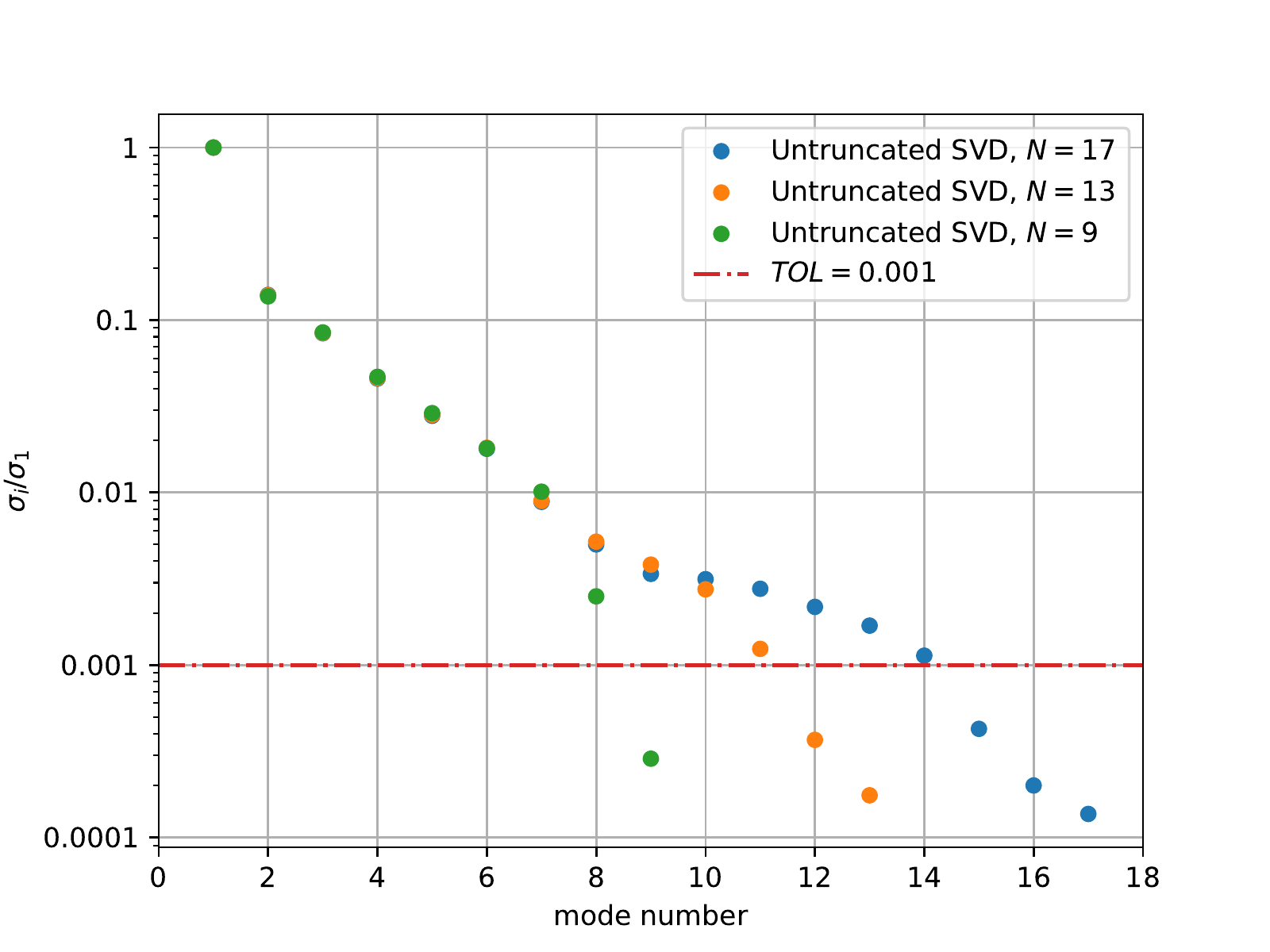} \\
\textrm{\footnotesize{(b) Linearly spaced snapshots}} &
 \textrm{\footnotesize{(a) Logarithmically spaced snapshots}}  
\end{array}$$
\caption{Sphere with $\mu_r=1.5$, $\sigma_*=5.96\times10^6$ S/m, $\alpha=0.01$ m: PODP applied to the computation of $\mathcal{M}[\alpha B, \omega]$  showing $\sigma_i/\sigma_1$ for $(a)$ linearly spaced  snapshots and  $(b)$  logarithmically spaced snapshots.}
\label{fig:Singular}
\end{center}
\end{figure}

The superior performance of logarithmically spaced frequency snapshots over those linearly spaced is illustrated in~Figure \ref{fig:LogvsLin}~$(a)$ where the variation of the condition number $\kappa(\mathbf{A}^M(\omega))$  with $\omega$ for the frequency range and snapshots presented in Figure~\ref{fig:Singular} is shown. Included in Figure~\ref{fig:LogvsLin}~$(b)$ is the
corresponding error measure  $|e(\Lambda_i(\omega))|:= |\Lambda_i^{exact}(\omega)-\Lambda_i^{PODP} (\omega)| /  |\Lambda_i^{exact}(\omega)| $ with $\omega$, where $\Lambda_i(\omega)=\lambda_i(\mathcal{R}[\alpha B, \omega]+{\mathcal N}^0[\alpha B])+\textrm{i}\lambda_i(\mathcal{I}[\alpha B ,\omega])$, $\lambda_i(\cdot)$ indicates the $i$th eigenvalue.
%,
%$\mathcal{R}^{PODP}[\alpha B, \omega]$, ${\mathcal N}^{0, PODP}[\alpha B ]={\mathcal N}^0[\alpha B ]$ and $\mathcal{I}^{PODP}[\alpha B, \omega]$ refer to the tensors obtained using PODP. 
 Note that, since the results for $i=1,2,3$ are identical on this scale, only $i=1$ is shown.  From the results shown in this figure, we see that, with the exception of $N=17$, the logarithmically spaced frequency snapshots results in a lower condition number compared to the linear ones and that all the logarithmically spaced snapshots result in a smaller error.

\begin{figure}[H]
$$\begin{array}{cc}
\includegraphics[width=0.5\textwidth, keepaspectratio]{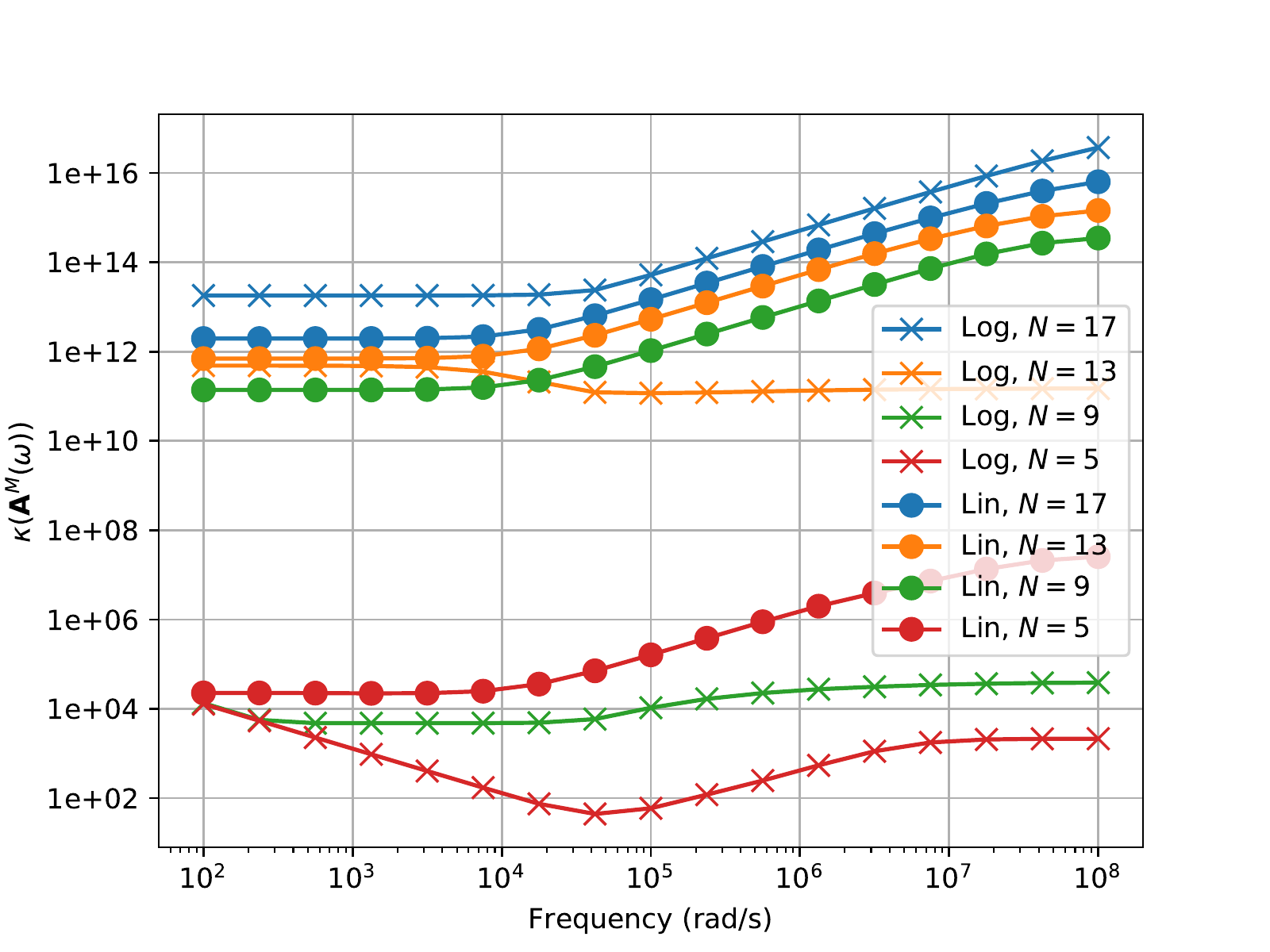} & \includegraphics[width=0.5\textwidth, keepaspectratio]{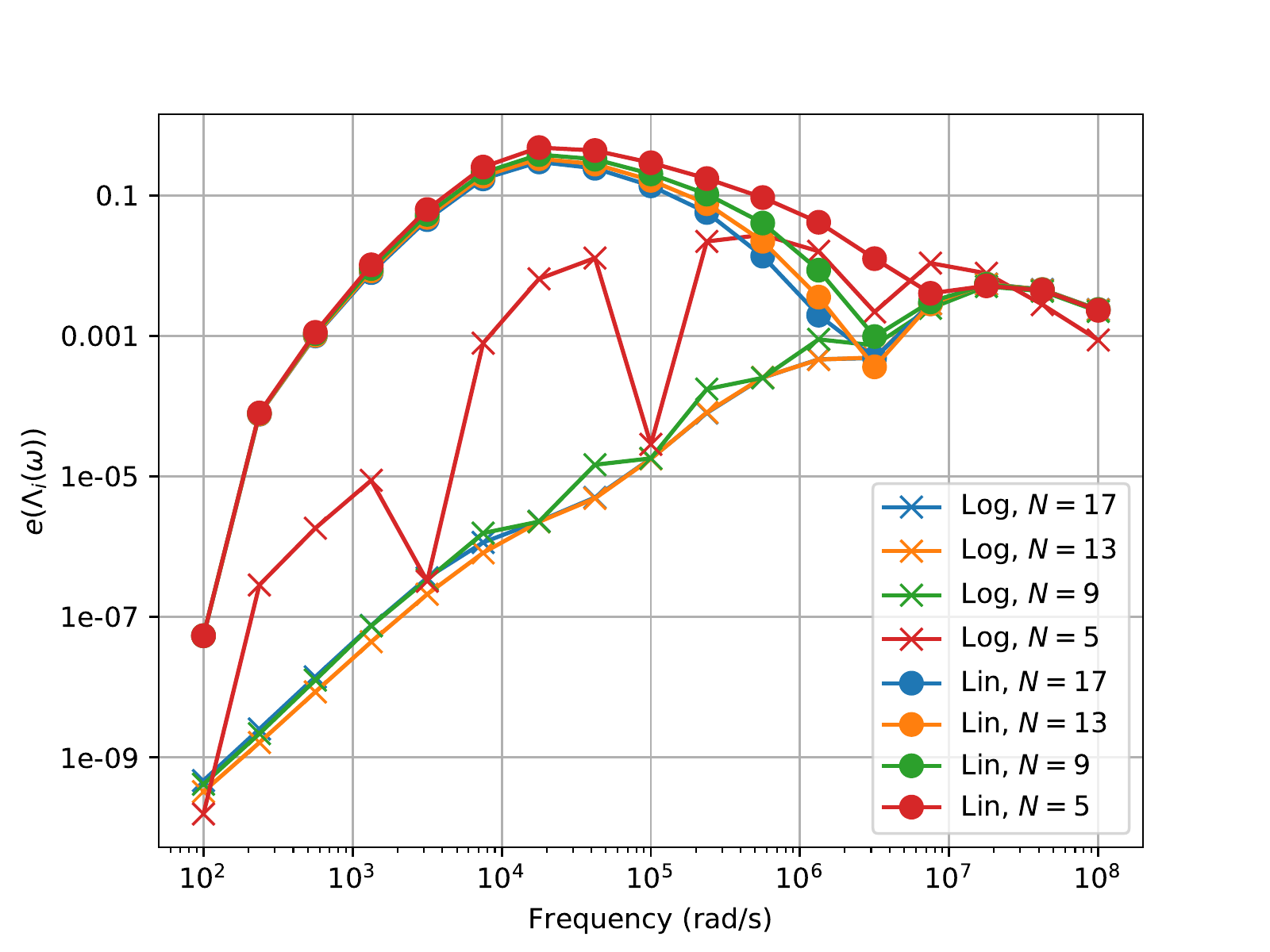}\\
\textrm{\footnotesize{(a) $\kappa(\mathbf{A}^M(\omega))$}} & \textrm{\footnotesize{(b) $e(\Lambda_i(\omega))$}}
\end{array}$$
\caption{Sphere with $\mu_r=1.5$, $\sigma_*=5.96\times10^6$ S/m, $\alpha=0.01$ m: PODP applied to the computation of $\mathcal{M}[\alpha B, \omega]$  showing
 $(a)$ Variation of $\kappa(\mathbf{A}^M(\omega))$ with $\omega$ for linearly and logarithmically spaced snapshots $(b)$ Variation of $e(\Lambda_i(\omega))$ with $\omega$ for the same snapshots.}
\label{fig:LogvsLin}
\end{figure}

Further tests reveal that the accuracy of the PODP using $N=9,13,17$ and  logarithmically spaced snapshots remains similar to that shown in Figure~\ref{fig:LogvsLin}~$(b)$ for $TOL\le 1\times 10^{-3}$ for this problem. We complete the discussion of the sphere by showing a comparison of $\lambda_i(\mathcal{R}[\alpha B, \omega])$ and $\lambda_i(\mathcal{I}[\alpha B, \omega])$, each with $\omega$, for the full order model, PODP using $N=9$ and the exact solution  in Figure~\ref{Best}. Again, the results for $i=1,2,3$ are identical and, hence, only $i=1$ is shown. In this figure, we observe excellent agreement between PODP, the full order model solution and exact solution.
% Note that here, and in the following, we frequently choose to drop the superscript PODP  on $\mathcal{R}$, $\mathcal{I}$ and $\mathcal{N}^0$ for brevity of presentation where no confusion arises.

\begin{figure}[H]
$$\begin{array}{cc}
\includegraphics[width=0.5\textwidth, keepaspectratio]{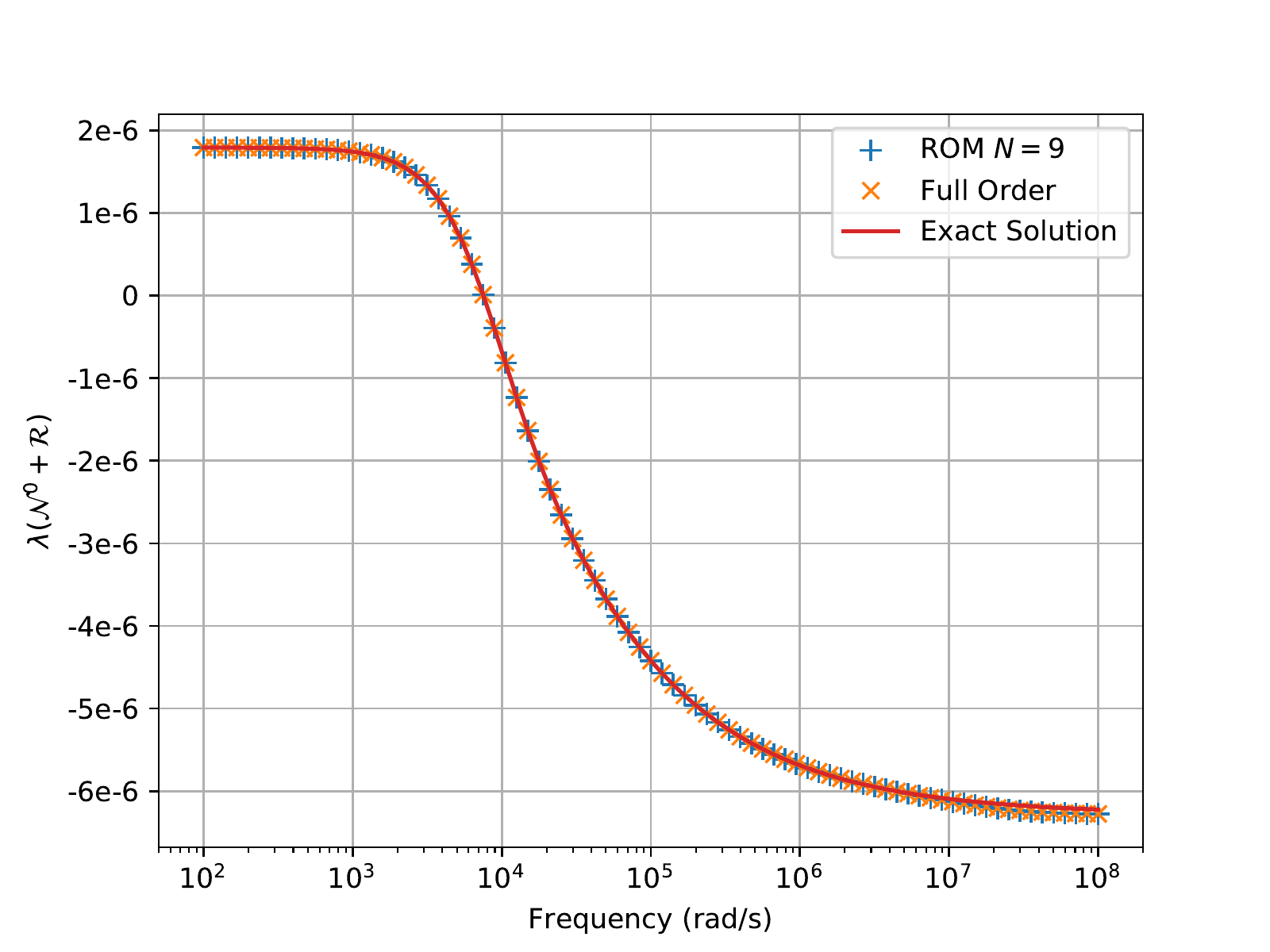} & \includegraphics[width=0.5\textwidth, keepaspectratio]{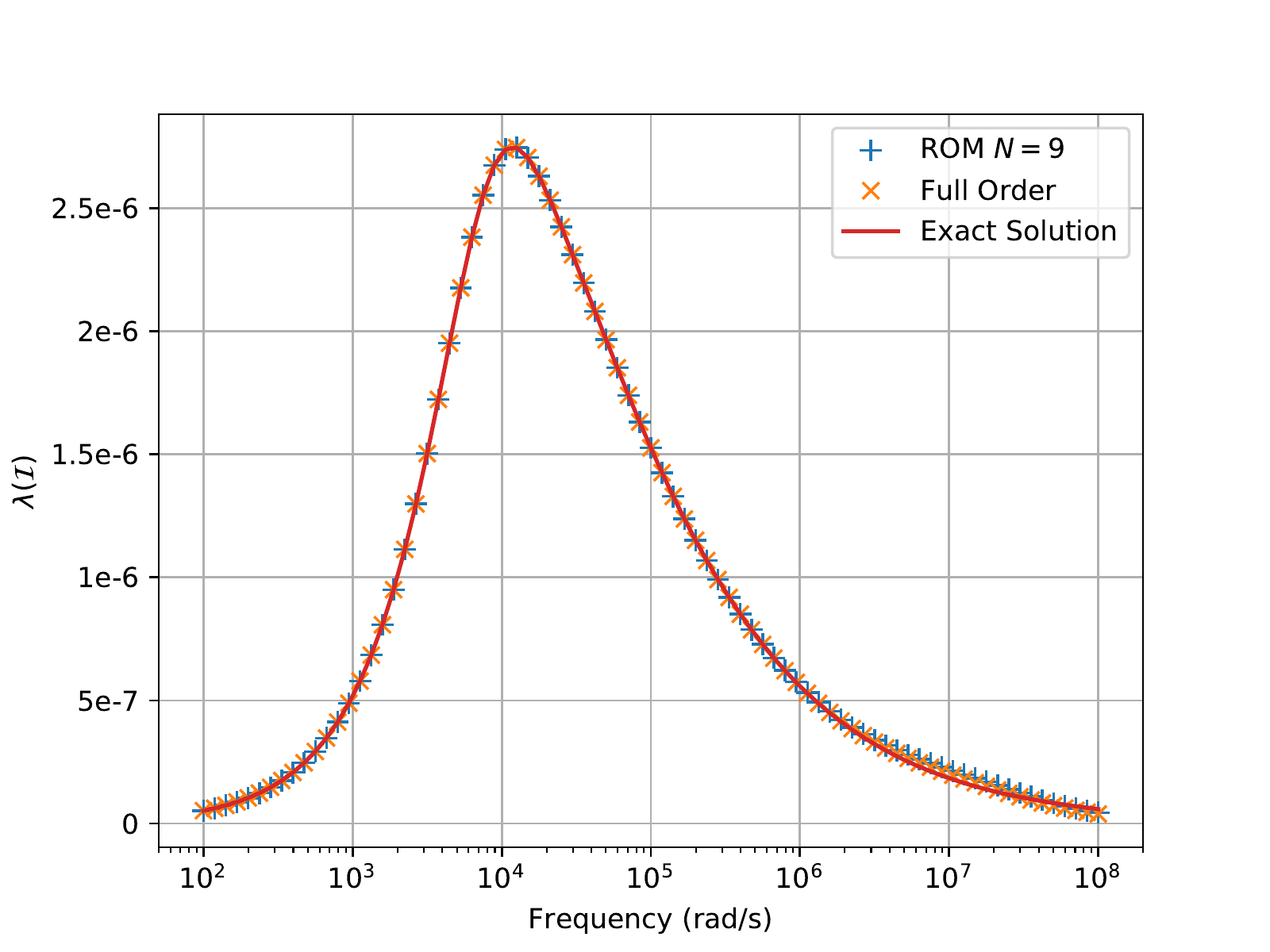}\\
\textrm{\footnotesize{(a) $\lambda_i(\mathcal{N}^0[\alpha B]+\mathcal{R}[\alpha B, \omega])$}} & \textrm{\footnotesize{(b) $\lambda_i(\mathcal{I}[\alpha B, \omega]$)}}
\end{array}$$
%\begin{figure}\label{LogvsLin}
\caption{Sphere with $\mu_r=1.5$, $\sigma_*=5.96\times10^6$ S/m, $\alpha=0.01$ m: PODP applied to the computation of $\mathcal{M}[\alpha B, \omega]$ with $N=9$ and $TOL=1\times 10^{-4}$ showing $(a)$ $\lambda_i(\mathcal{N}^0[\alpha B]+\mathcal{R}[\alpha B, \omega])$  and $(b)$ $\lambda_i(\mathcal{I}[\alpha B, \omega])$ each with $\omega$.}
\label{Best}
\end{figure}
In Figure~\ref{fig:sphereerror}, we show the output certificates $(\mathcal{R}^{PODP}[\alpha B, \omega]+{\mathcal N}^{0, PODP}[\alpha B ])_{ii}\pm (\Delta[\omega])_{ii}$ (summation of  repeated indices is not implied) and $(\mathcal{I}^{PODP}[\alpha B, \omega])_{ii}\pm (\Delta[\omega])_{ii}$, each with $\omega$,
%In Figure~\ref{fig:sphereerror}, we show the output certificates $(\mathcal{R}[\alpha B, \omega]+{\mathcal N}^{0}[\alpha B ])_{ii}\pm (\Delta[\omega])_{ii}$ (summation of  repeated indices is not implied) and $(\mathcal{I}[\alpha B, \omega])_{ii}\pm (\Delta[\omega])_{ii}$, each with $\omega$,
obtained by applying the technique described in Section~\ref{sect:outputcert} for the case where $i=1$ and with $N=17, 21$ and $TOL=1\times10^{-6}$. Similar certificates can be obtained for the other tensor coefficients. We observe that certificates are indistinguishable from the 
 the MPT coefficients obtained with PODP
 for low frequencies in both cases and the certificates rapidly tend to the MPT coefficients for all $\omega$ as $N$ is increased.
Note that  $TOL=1\times10^{-6}$ is chosen as larger tolerances lead to larger certificates, however, this reduction in tolerance does not substantially affect the computational cost of the ROM.
  Although the effectivity indices $(\Delta [\omega])_{11}/|(\mathcal{R}[\alpha B, \omega]- \mathcal{R}^{PODP}[\alpha B, \omega] )_{11}|$
  %~\footnote{Note that $(\mathcal{N}^0)^{PODP}[\alpha B] =\mathcal{N}^0[\alpha B] $.} 
  and $(\Delta [ \omega])_{11}/|(\mathcal{I}[\alpha B, \omega]- \mathcal{I}^{PODP}[\alpha B, \omega] )_{11}|$) of the PODP with respect to the full order model
are clearly larger at higher frequencies, we emphasise that they are computed at negligible additional cost, they converge rapidly to 
the MPT coefficients obtained with PODP
 as $N$ is increased and give credibility in the PODP solution without the need of performing additional full order model solutions to validate the ROM.

\begin{figure}[H]
$$\begin{array}{cc}
\includegraphics[width=0.5\textwidth, keepaspectratio]{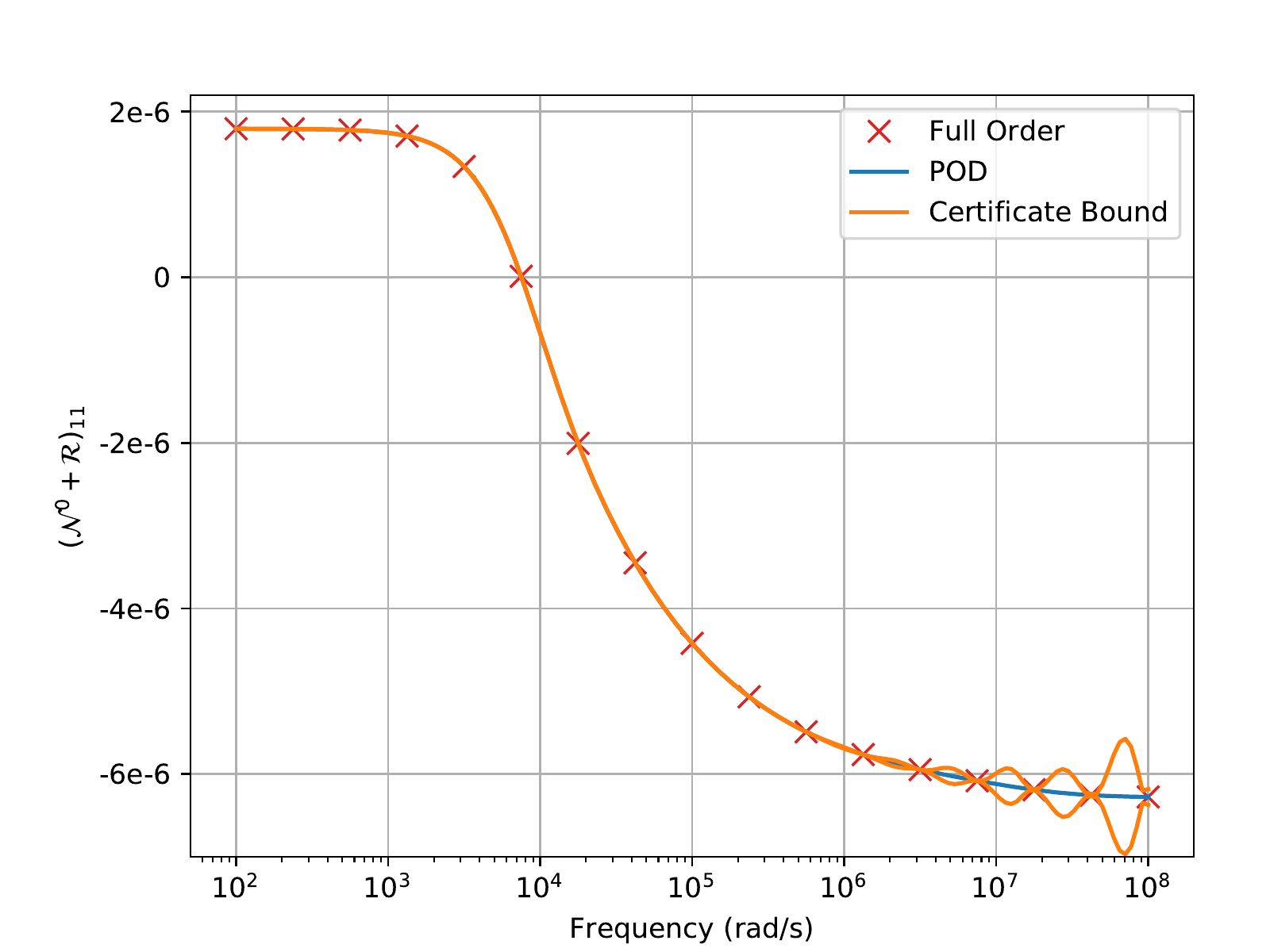} &
\includegraphics[width=0.5\textwidth, keepaspectratio]{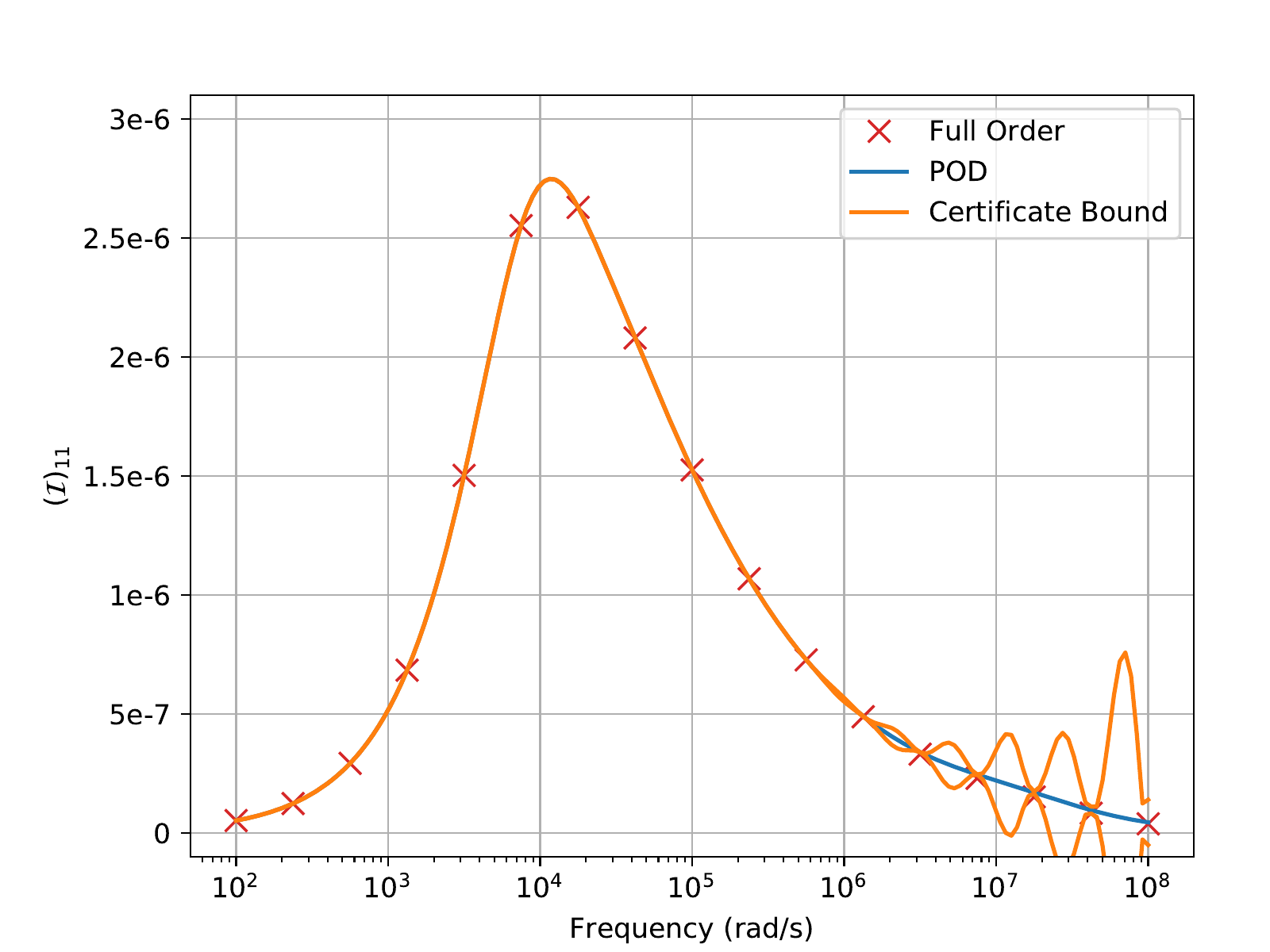} \\
\textrm{\footnotesize{(a) $(\mathcal{N}^0[\alpha B]+\mathcal{R}[\alpha B, \omega])_{11}$, $N=17$}} &
\textrm{\footnotesize{(b) $(\mathcal{I}[\alpha B, \omega])_{11}$, $N=17$} } \\
\includegraphics[width=0.5\textwidth, keepaspectratio]{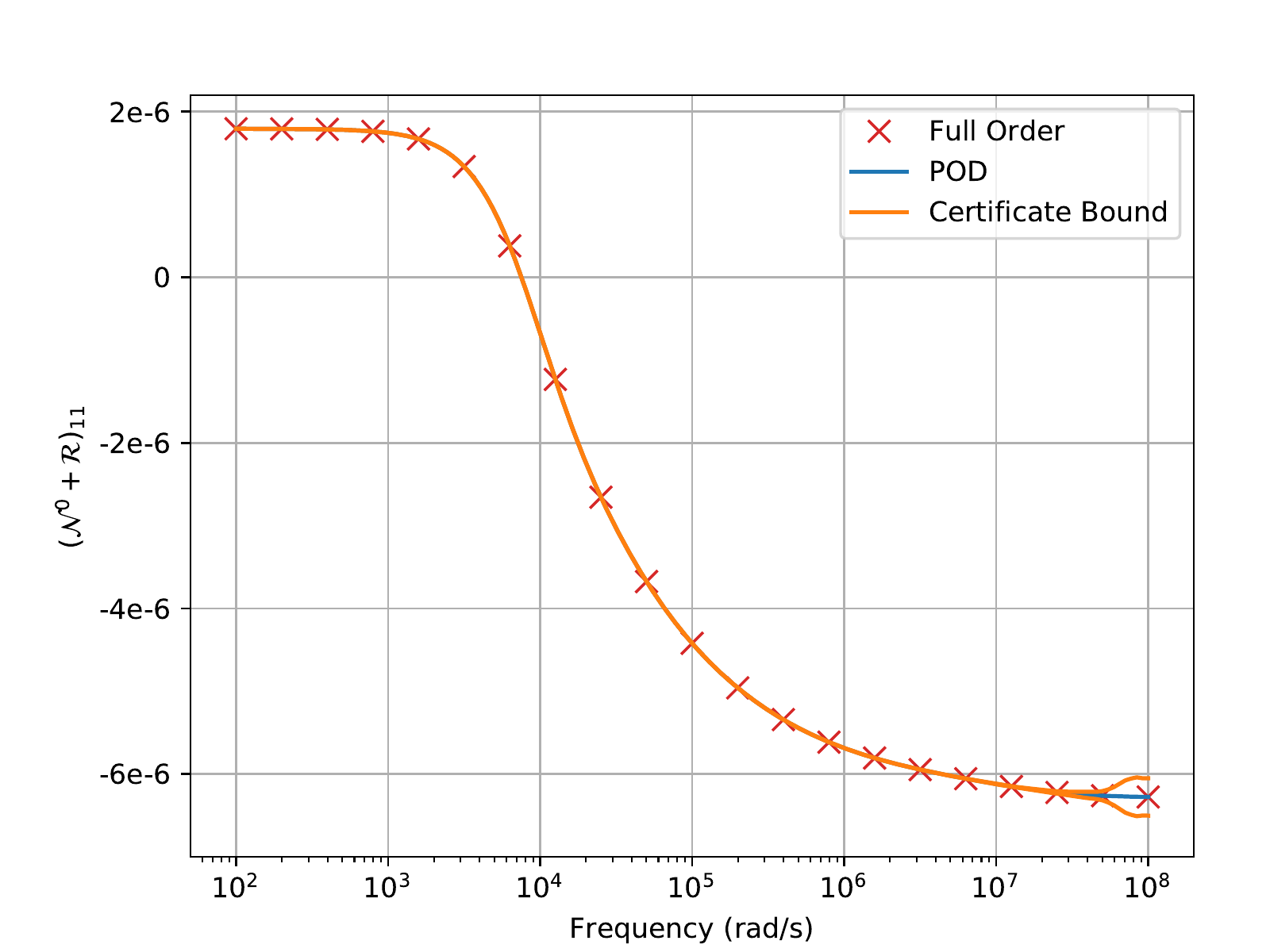} &
\includegraphics[width=0.5\textwidth, keepaspectratio]{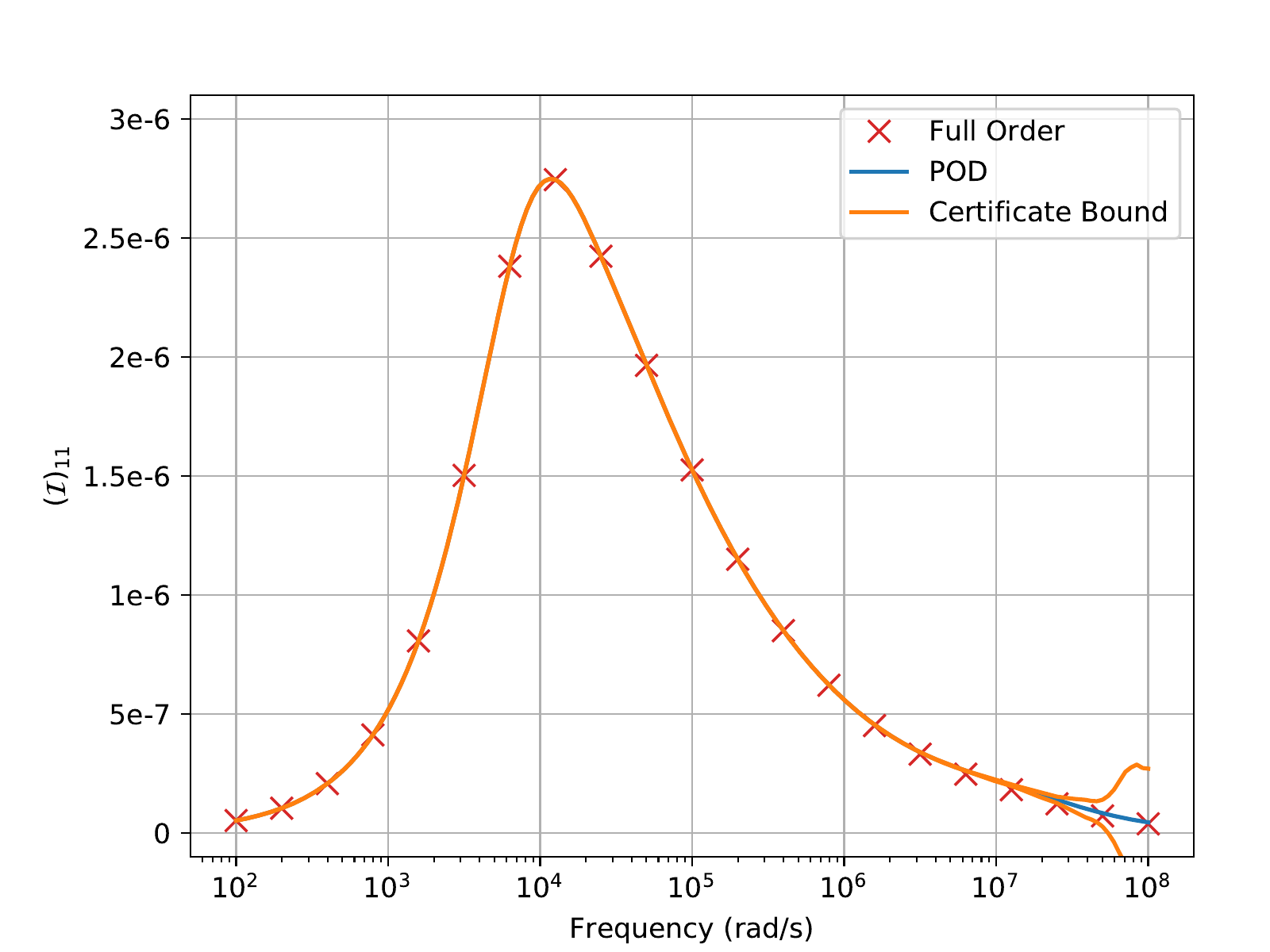}\\
\textrm{\footnotesize{(c) $(\mathcal{N}^0[\alpha B]+\mathcal{R}[\alpha B, \omega])_{11}$, $N=21$}} &
\textrm{\footnotesize{(d) $(\mathcal{I}[\alpha B, \omega])_{11}$, $N=21$} } 
\end{array}$$
\caption{Sphere with $\mu_r=1.5$, $\sigma_*=5.96\times10^6$ S/m, $\alpha=0.01$ m: PODP applied to the computation of $\mathcal{M}[\alpha B, \omega]$ with $TOL=1\times 10^{-6}$ showing the PODP solution, full order model solutions and output certificates $(\cdot ) \pm (\Delta [ \omega])_{11}$ for  $(a)$ $ (\mathcal{N}^0[\alpha B]+\mathcal{R}[\alpha B, \omega])_{11}$ using $N=17$,
$(b)$ $ (\mathcal{I}[\alpha B, \omega])_{11}$ using $N=17$,
  $(c)$ 
$(\mathcal{N}^0[\alpha B]+\mathcal{R}[\alpha B, \omega])_{11}$ using $N=21$ and $(d)$ 
$(\mathcal{I}[\alpha B, \omega])_{11}$ using $N=21$, each with $\omega$.}
\label{fig:sphereerror}
\end{figure}

The computational  speed-ups offered by using the PODP compared to a frequency sweep performed with the full order model are shown in Figure~\ref{fig:Speedup} where  $N=9,13,17$ and  logarithmically spaced snapshots are chosen with  $\omega_{min}= 1\times 10^2 \text{ rad/s}$, $\omega_{max}= 1\times 10^8 \text{ rad/s}$, as before. For the comparison, we vary the number of output points $N_0$ produced in a frequency sweep and measure the time taken to produce each of these frequency sweeps using a 2.9 GHz quad core Intel i5 processor and also show the percentage speed up offered by each of these PODP sweeps. Also shown is the break down of the computational time for the offline and online stages of the PODP for the case where $N=13$. Note, in particular, that the computational cost increases very slowly with $N_0$ and that the additional cost involved in computing the output certificates is negligible. The breakdown of computational costs for other $N$ is similar.

\begin{figure}[H]
$$\begin{array}{cc}
\includegraphics[width=0.5\textwidth, keepaspectratio]{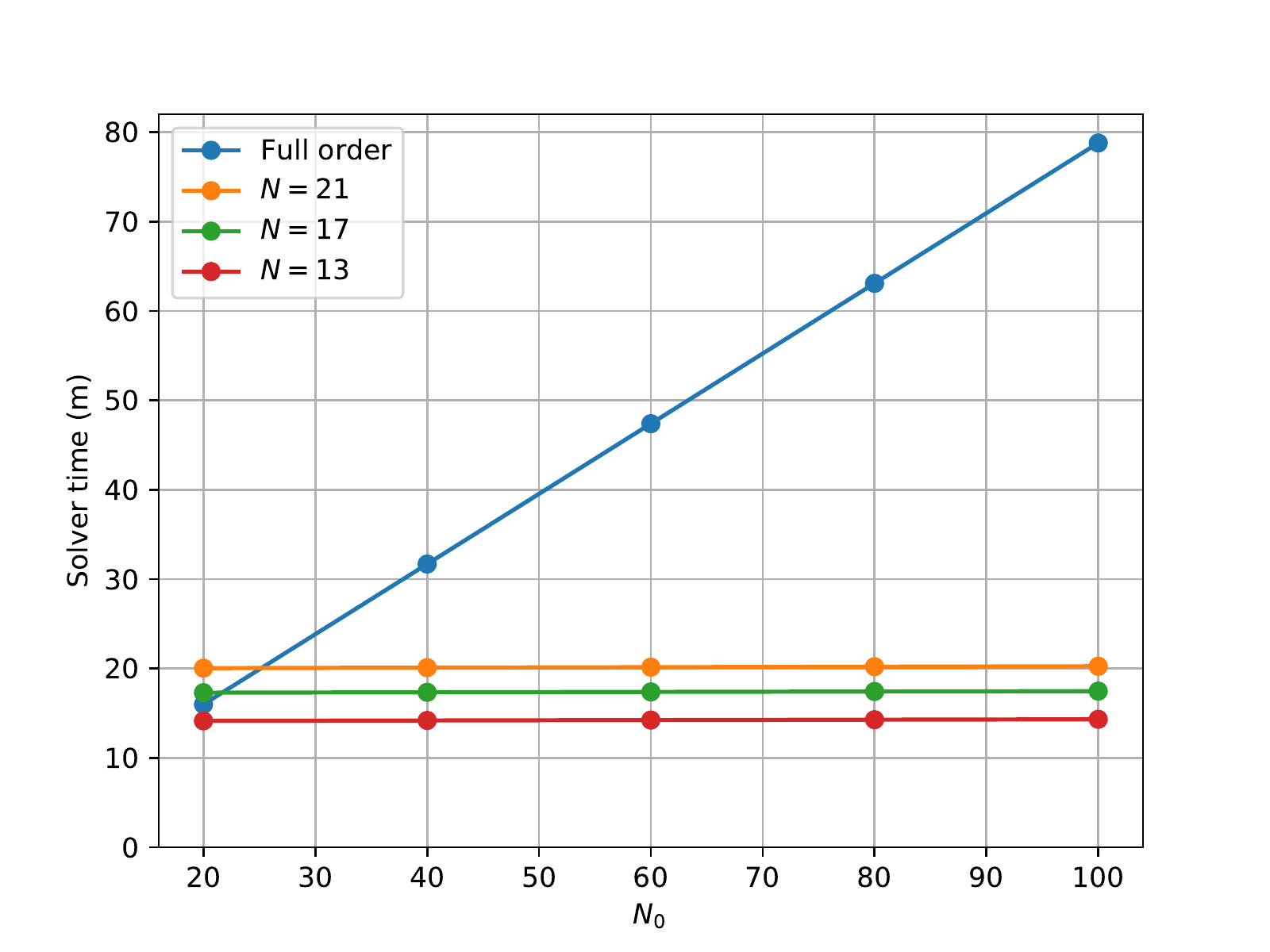} & \includegraphics[width=0.5\textwidth, keepaspectratio]{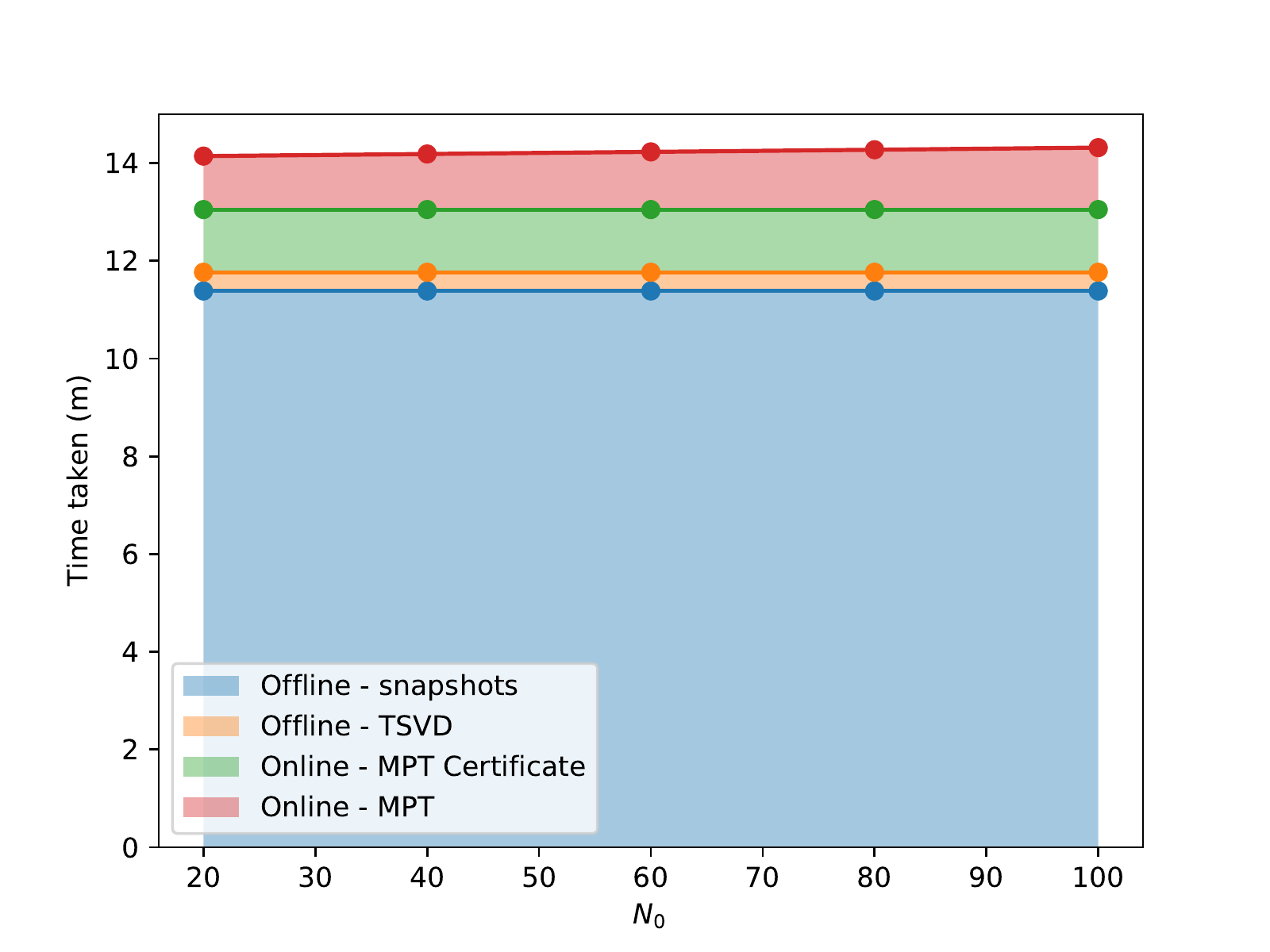}\\
\textrm{\footnotesize{(a) Sweep Time}} & \textrm{\footnotesize{(b) Break Down of PODP Timings}}
\end{array}$$
%\begin{figure}\label{LogvsLin}
\caption{Sphere with $\mu_r=1.5$, $\sigma_*=5.96\times10^6$ S/m, $\alpha=0.01$ m: PODP applied to the computation of $\mathcal{M}[\alpha B, \omega]$ with $N=13,17,21$ and $TOL=1\times 10^{-6}$ showing, for different numbers of outputs $N_0$, $(a)$ sweep computational time compared with full order and $(b)$ a typical break down of the offline and online 
computational times for $N=13$.}
\label{fig:Speedup}
\end{figure}

%%%%%%%%%%%%%%%%%%%%%%%%%%%%%%%%%%%%%%%%%%%%%%%%%%%%%%%%%%%%%%%%
\subsection{Conducing permeable torus}\label{sect:condtorus}
%%%%%%%%%%%%%%%%%%%%%%%%%%%%%%%%%%%%%%%%%%%%%%%%%%%%%%%%%%%%%%%%
Next, we consider $B_{\alpha}= \alpha B$ to be a torus where $B$ has major and minor radii  $a=2$  and $b=1$, respectively, $\alpha=0.01$ m and the object is permeable and  conducting with
 $\mu_r=1.5$, $\sigma_*=5\times10^5$ S/m.
 The object is centred at the origin so that it has rotational symmetry around the $\bm{e}_1$ axis and hence ${\mathcal M}[\alpha B, \omega]$ has independent coefficients $({\mathcal M}[\alpha B, \omega])_{11}$ and $({\mathcal M}[\alpha B, \omega])_{22}=({\mathcal M}[\alpha B, \omega])_{33}$, and thus $\mathcal{N}^0[\alpha B]$, $\mathcal{R}[\alpha B, \omega]$, $\mathcal{I}[\alpha B, \omega]$ each have $2$ independent eigenvalues. To compute the full order model, we 
 set $\Omega$ to be a sphere of radius 100, centred at the origin, such that it contains $B$ and
 discretise it by a mesh of  26142 unstructured tetrahedral elements, refined towards the object, and a polynomial order of $p=3$. This discretisation has already been found to produce an accurate representation of $\mathcal{M}[\alpha B, \omega]$ for the frequency range with $\omega_{min}=1 \times10^2 {\textrm{ rad}}{/ \textrm{s}}$ and $\omega_{max} = 1 \times10^8 {\textrm{ rad}} / {\textrm{s}}$ with the full order model. 
 
 The reduced order model is constructed using $N=13$ snapshots at logarithmically spaced frequencies with $TOL=1 \times 10^{-4} $. Figure~\ref{fig:Torus} shows the results for $\lambda_i(\mathcal{N}^0[\alpha B]+\mathcal{R}[\alpha B, \omega])$ and $\lambda_i(\mathcal{I}[\alpha B, \omega])$, each with $\omega$, for both the full order model and the PODP. The agreement is excellent in both cases.
\begin{figure}[H]
$$\begin{array}{cc}
\includegraphics[width=0.5\textwidth, keepaspectratio]{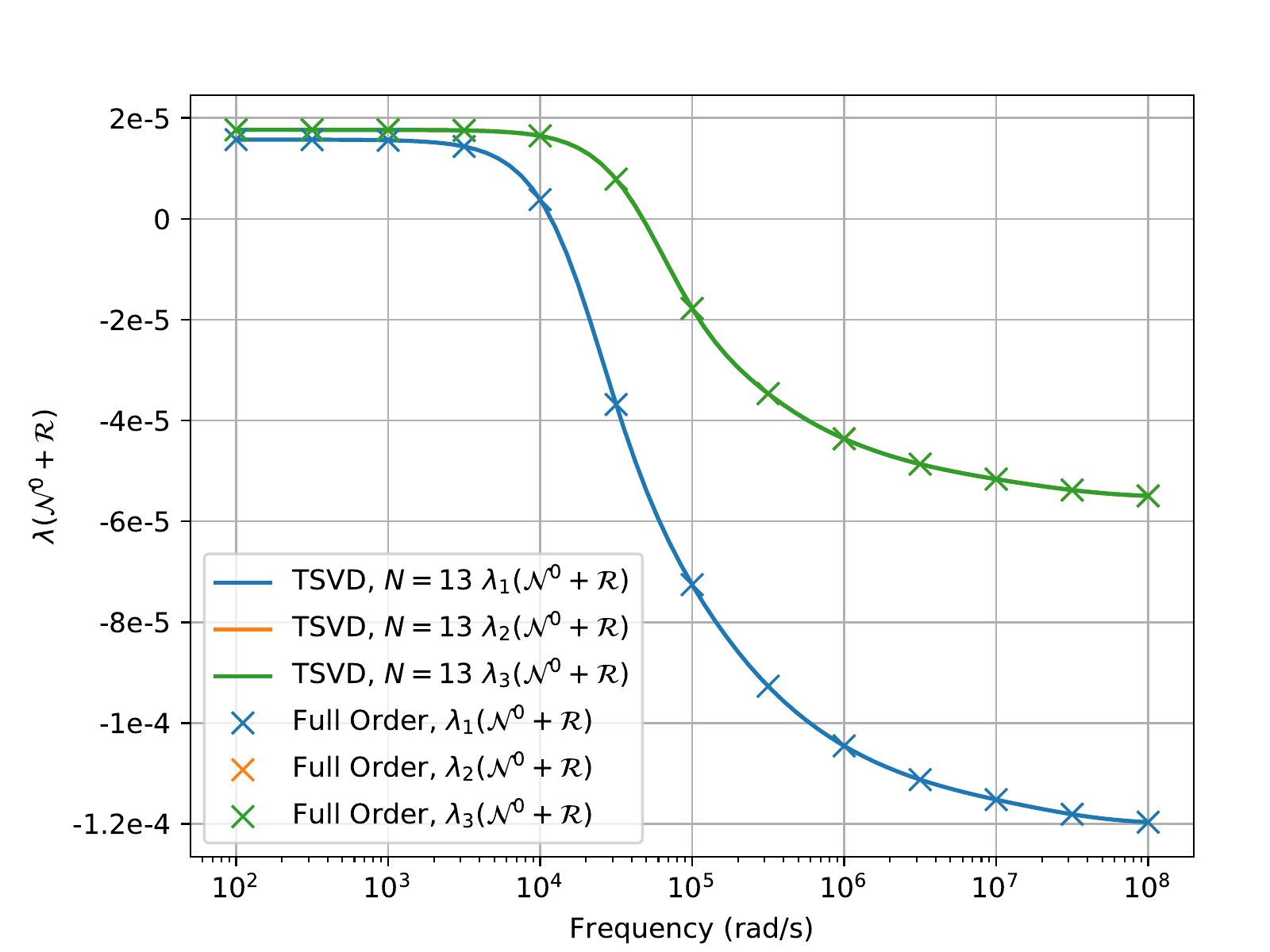} & \includegraphics[width=0.5\textwidth, keepaspectratio]{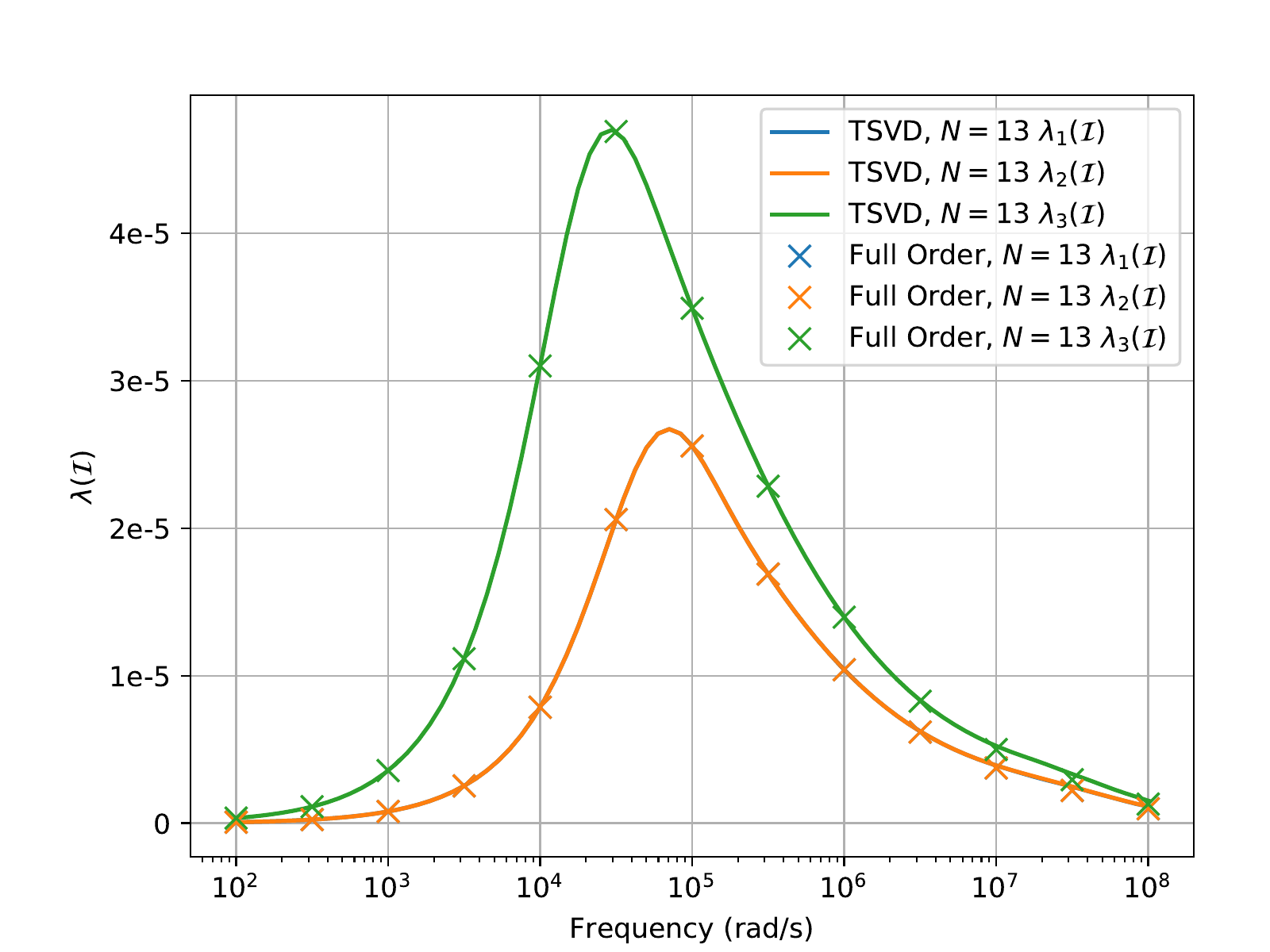}\\
\textrm{\footnotesize{(a) $\lambda_i(\mathcal{N}^0[\alpha B]+\mathcal{R}[\alpha B, \omega])$}} & \textrm{\footnotesize{(b) $\lambda_i(\mathcal{I}[\alpha B, \omega]$)}}
\end{array}$$
%\begin{figure}\label{LogvsLin}
\caption{Torus with major and minor radii of $a=2$ and $b=1$, respectively, and $\mu_r=1.5$, $\sigma_*=5\times10^5$ S/m, $\alpha=0.01$ m: 
PODP applied to the computation of $\mathcal{M}[\alpha B, \omega]$ $N=13$ and  $TOL=1 \times 10^{-4} $
 showing $(a)$ $\lambda_i(\mathcal{N}^0[\alpha B]+\mathcal{R}[\alpha B, \omega])$  and $(b)$ $\lambda_i(\mathcal{I}[\alpha B, \omega])$, each with $\omega$. %\textcolor{red}{remove titles from the graphs!}
 }
\label{fig:Torus}
\end{figure}

In Figure~\ref{fig:ErrorTorus} we show the output certificates $(\mathcal{R}^{PODP}[\alpha B, \omega]+{\mathcal N}^{0, PODP}[\alpha B ])_{ii}\pm (\Delta[\omega])_{ii}$ (no summation over repeated indices implied)  and $(\mathcal{I}^{PODP}[\alpha B, \omega])_{ii}\pm (\Delta[\omega])_{ii}$, each with $\omega$,
%In Figure~\ref{fig:ErrorTorus} we show the output certificates $(\mathcal{R}[\alpha B, \omega]+{\mathcal N}^{0}[\alpha B ])_{ii}\pm (\Delta[\omega])_{ii}$ (no summation over repeated indices implied)  and $(\mathcal{I}[\alpha B, \omega])_{ii}\pm (\Delta[\omega])_{ii}$, each with $\omega$,
obtained by applying the technique described in Section~\ref{sect:outputcert} for the case where $N=17$ and $TOL=1\times10^{-6}$.  Note that we increased the number of snapshots from $N=13$ to $N=17$ and have reduced the tolerance to ensure tight certificates bounds. 
%In practice, these would be very acceptable as the eddy current approximation on which the MPT characterisation is based  is known to break down for high frequencies~\cite{LedgerLionheart2016,LedgerLionheart2019}.

\begin{figure}[H]
$$\begin{array}{cc}
\includegraphics[width=0.5\textwidth, keepaspectratio]{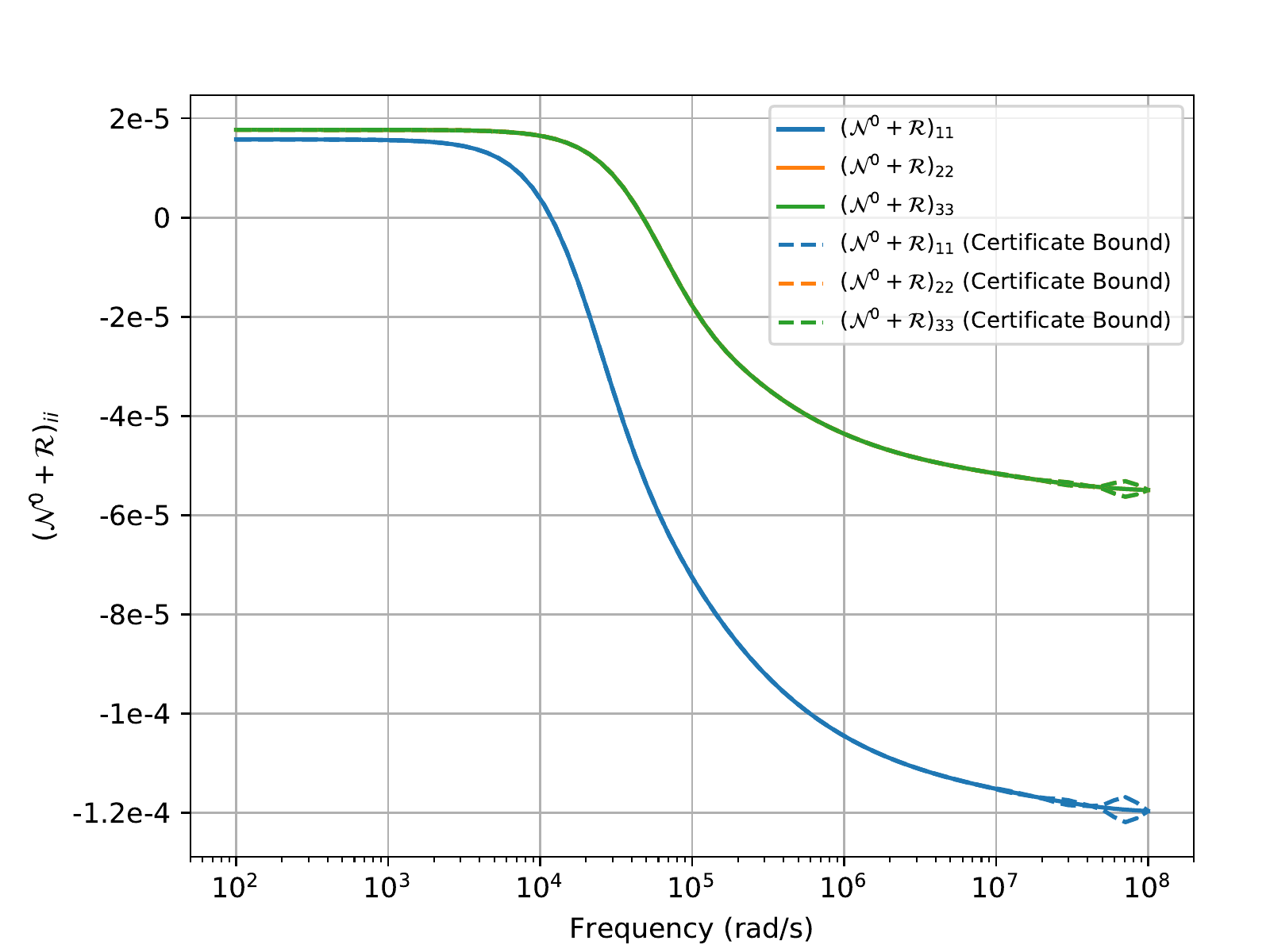} &
\includegraphics[width=0.5\textwidth, keepaspectratio]{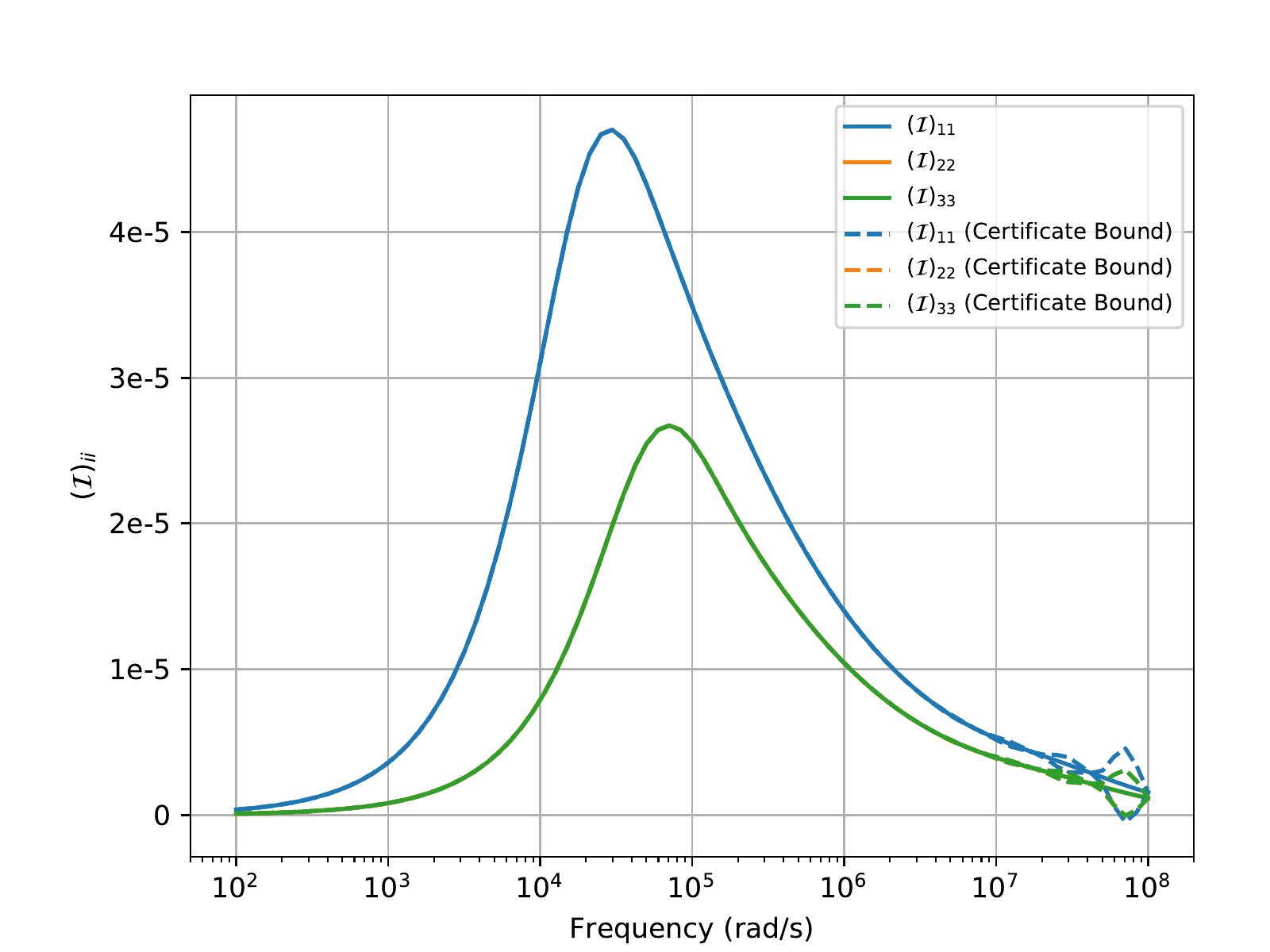} \\
\textrm{\footnotesize{(a) $(\mathcal{N}^0[\alpha B]+\mathcal{R}[\alpha B, \omega] )_{ii}$}} & \textrm{\footnotesize{(b) $(\mathcal{I}[\alpha B, \omega])_{ii} $}} 
\end{array}$$
\caption{
Torus with $\mu_r=1.5$,  $\sigma_*=5\times10^5$ S/m, $\alpha=0.01$ m: 
PODP applied to the computation of $\mathcal{M}[\alpha B, \omega]$ with $TOL=1\times 10^{-6}$ and $N=17$ showing the PODP solution and output certificates $(\cdot ) \pm( \Delta [\omega])_{ii}$ for  $(a)$ $ (\mathcal{N}^0[\alpha B]+\mathcal{R}[\alpha B, \omega])_{ii}$,
$(b)$ $ (\mathcal{I}[\alpha B, \omega])_{ii}$,
each with $\omega$. 
}
\label{fig:ErrorTorus}
\end{figure}

%%%%%%%%%%%%%%%%%%%%%%%%%%%%%%%%%%%%%%%%%%%%%%%%%%%%%%%%%%%%%%%%
\subsection{Conducting permeable tetrahedron}\label{sect:tetra}
%%%%%%%%%%%%%%%%%%%%%%%%%%%%%%%%%%%%%%%%%%%%%%%%%%%%%%%%%%%%%%%%
The third object considered is where $B_\alpha = \alpha B$ is a conducting permeable tetrahedron. The vertices of the tetrahedron $B$ are chosen to be at the locations 
\begin{equation}
v_1=\begin{pmatrix} 0\\0\\0\end{pmatrix},\  v_2=\begin{pmatrix} 7\\0\\0\end{pmatrix},\  v_3=\begin{pmatrix} 5.5\\4.6\\0\end{pmatrix}\ \textrm{and}\ v_4=\begin{pmatrix} 3.3\\2\\5\end{pmatrix}, \nonumber
\end{equation}
the object size is $\alpha = 0.01$ m and the tetrahedron is permeable  and conducting with $\mu_r=2$ and $\sigma_*=5.96\times10^6$ S/m. The object does not have rotational or reflectional symmetries and, hence, ${\mathcal M}[\alpha B,\omega]$ has $6$ independent coefficients and, thus,  $\mathcal{N}^0[\alpha B]$, $\mathcal{R}[\alpha B, \omega]$, $\mathcal{I}[\alpha B, \omega]$ each have $3$ independent eigenvalues. To compute the full order model, we  set $\Omega$ to be a cube with sides of length 200 centred about the origin and discretise it with a mesh of 21427 unstructured tetrahedral elements, refined towards the object, and a polynomial order of $p=3$. This discretisation has already been found to produce an accurate representation of $\mathcal{M}[\alpha B, \omega] $ for the frequency range with  $\omega_{min} = 1 \times 10^2 {\textrm{ rad}}  / { \textrm{s}}$ and $\omega_{max} = 1 \times 10^8 {\textrm{ rad}} / {\textrm{s}}$.

The reduced order model is constructed using $N=13 $ snapshots  at logarithmically spaced frequencies with $TOL=1 \times 10^{-4} $. Figure~\ref{fig:Tetra} shows the results for $\lambda_i(\mathcal{N}^0[\alpha B]+\mathcal{R}[\alpha B, \omega])$ and $\lambda_i(\mathcal{I}[\alpha B, \omega])$, each with $\omega$, for both the full order model and the PODP. The agreement is excellent in both cases.

\begin{figure}[H]
$$\begin{array}{cc}
\includegraphics[width=0.5\textwidth, keepaspectratio]{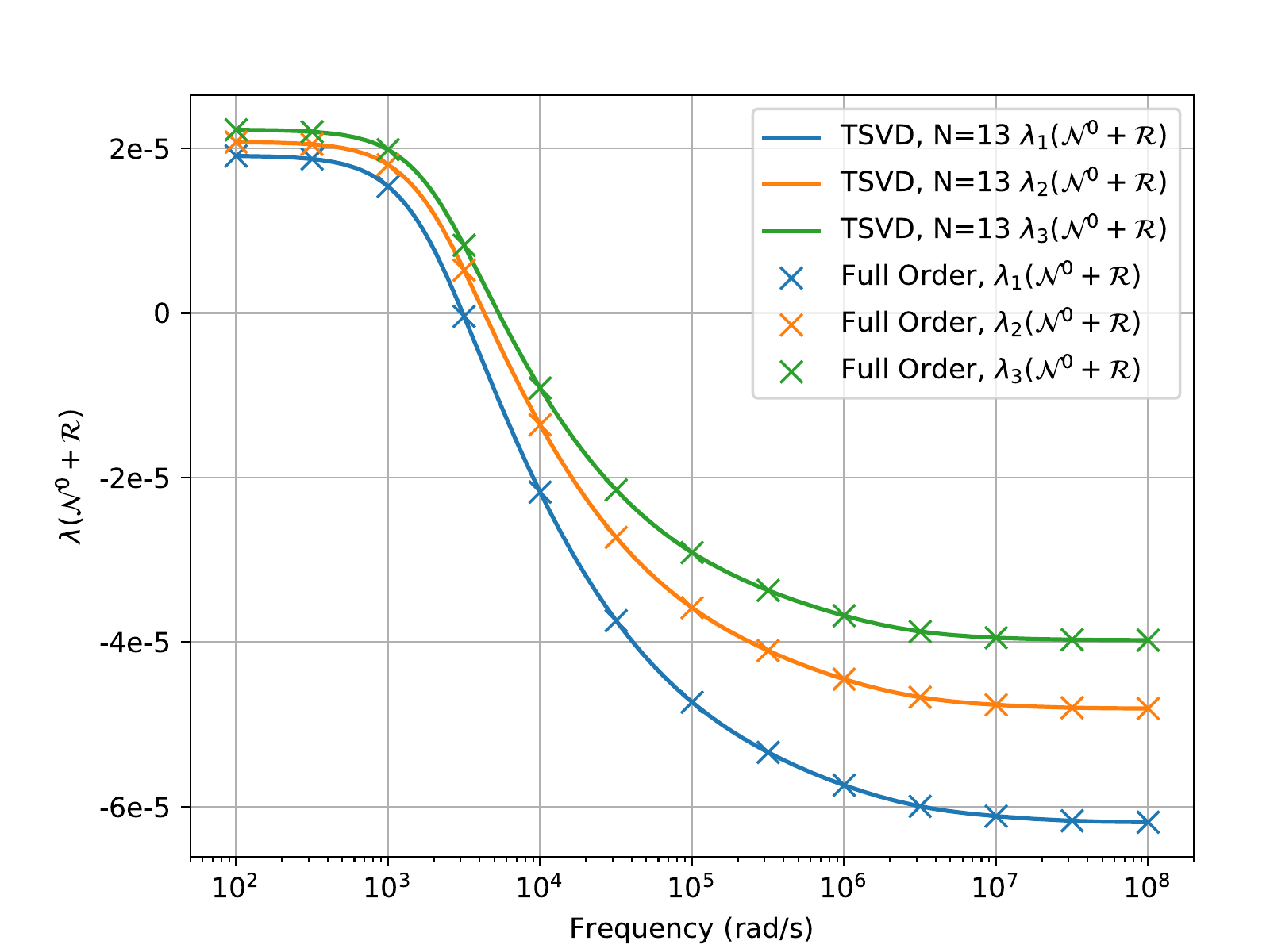} & \includegraphics[width=0.5\textwidth, keepaspectratio]{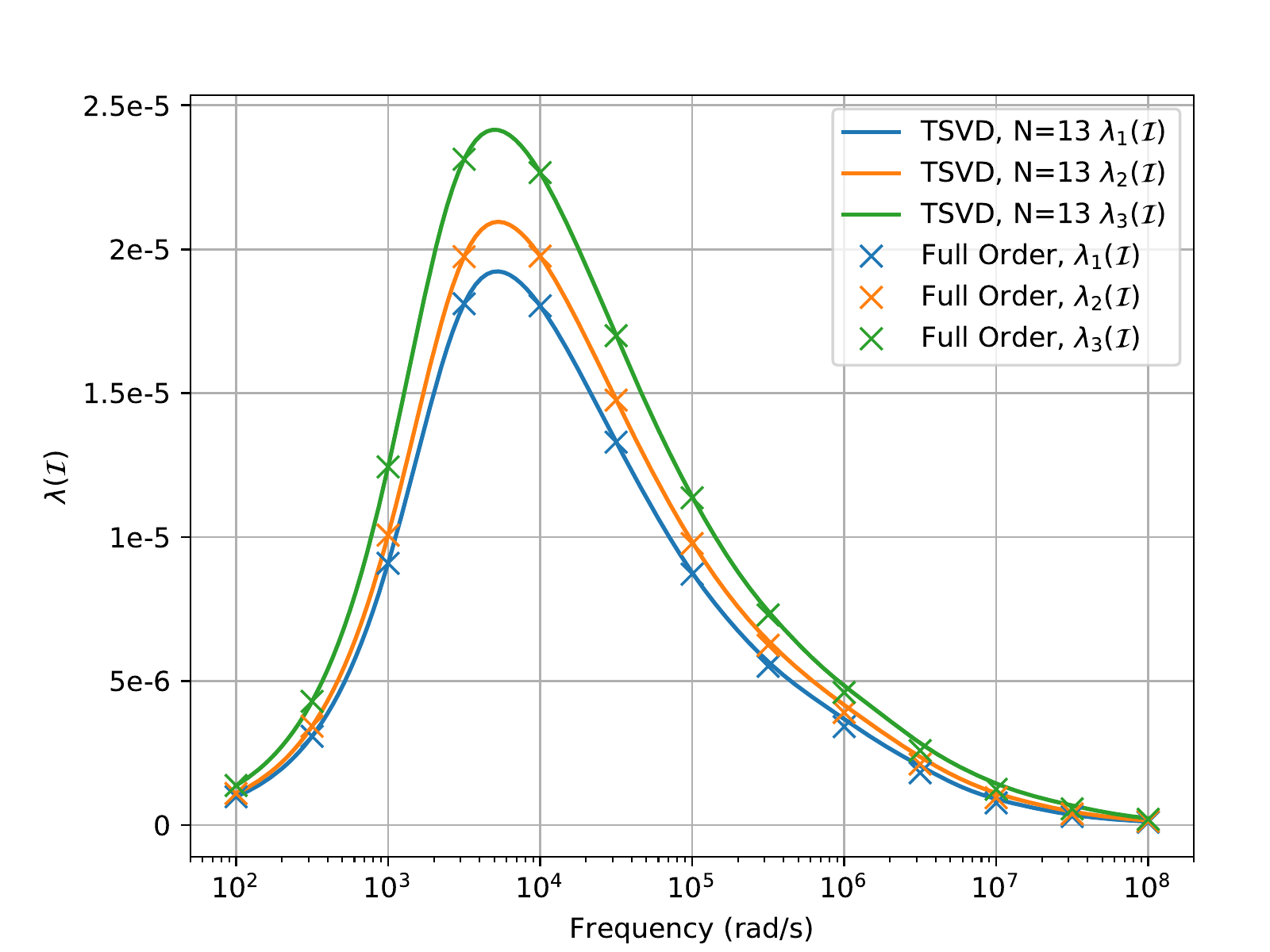}\\
\textrm{\footnotesize{(a) $\lambda_i(\mathcal{N}^0[\alpha B]+\mathcal{R}[\alpha B, \omega] )$}} & \textrm{\footnotesize{(b) $\lambda_i(\mathcal{I}[\alpha B, \omega] $)}}
\end{array}$$
\caption{
Irregular tetrahedron with $\mu_r=2$,  $\sigma_*=5.96\times10^6$ S/m, $\alpha=0.01$ m: 
PODP applied to the computation of $\mathcal{M}[\alpha B, \omega]$ $N=13$ and  $TOL=1 \times 10^{-4} $
 showing $(a)$ $\lambda_i(\mathcal{N}^0[\alpha B]+\mathcal{R}[\alpha B, \omega])$  and $(b)$ $\lambda_i(\mathcal{I}[\alpha B, \omega])$, each with $\omega$.}
\label{fig:Tetra}
\end{figure}

In Figure~\ref{fig:ErrorTetra} we show the output certificates $(\mathcal{R}^{PODP}[\alpha B, \omega]+{\mathcal N}^{0, PODP}[\alpha B ])_{ij}\pm (\Delta[\omega])_{ij}$  and \\ $(\mathcal{I}^{PODP}[\alpha B, \omega])_{ij}\pm (\Delta[\omega])_{ij}$, both with $\omega$, for $i=j$ and $i\ne j$
%obtained by applying the technique described in Section~\ref{sect:outputcert} for the case where $N=21$ and $TOL=1\times10^{-6}$.  Once again, we have
%In Figure~\ref{fig:ErrorTetra} we show the output certificates $(\mathcal{R}[\alpha B, \omega]+{\mathcal N}^{0}[\alpha B ])_{ij}\pm (\Delta[\omega])_{ij}$  and \\ $(\mathcal{I}[\alpha B, \omega])_{ij}\pm (\Delta[\omega])_{ij}$, both with $\omega$, for $i=j$ and $i\ne j$
obtained by applying the technique described in Section~\ref{sect:outputcert} for the case where $N=21$ and $TOL=1\times10^{-6}$.  Once again, we 
 increased the number of snapshots from $N=13$ to $N=21$ and have reduced the tolerance to ensure tight certificates bounds, except at large frequencies.

\begin{figure}[H]
$$\begin{array}{cc}
\includegraphics[width=0.5\textwidth, keepaspectratio]{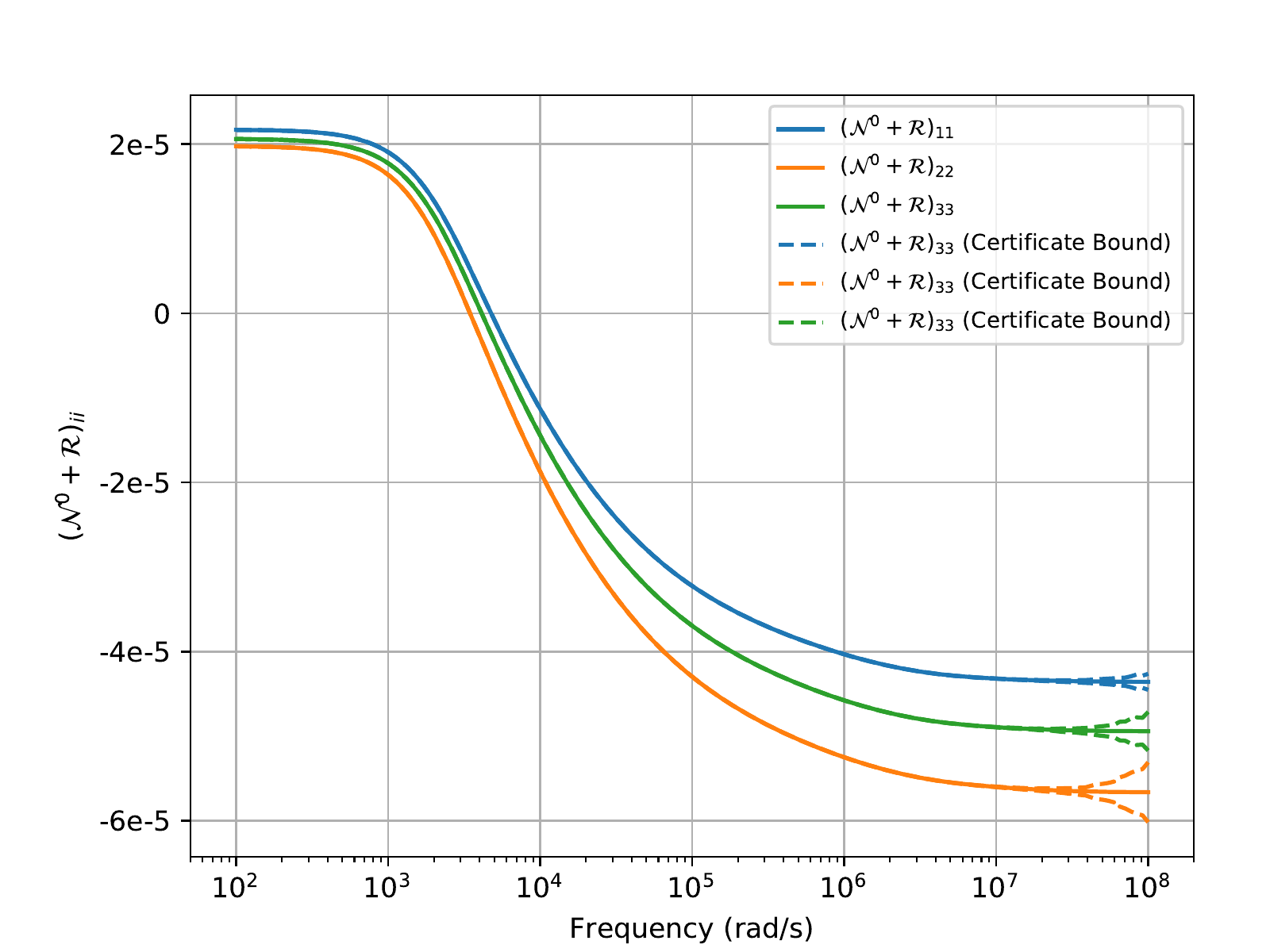} &
\includegraphics[width=0.5\textwidth, keepaspectratio]{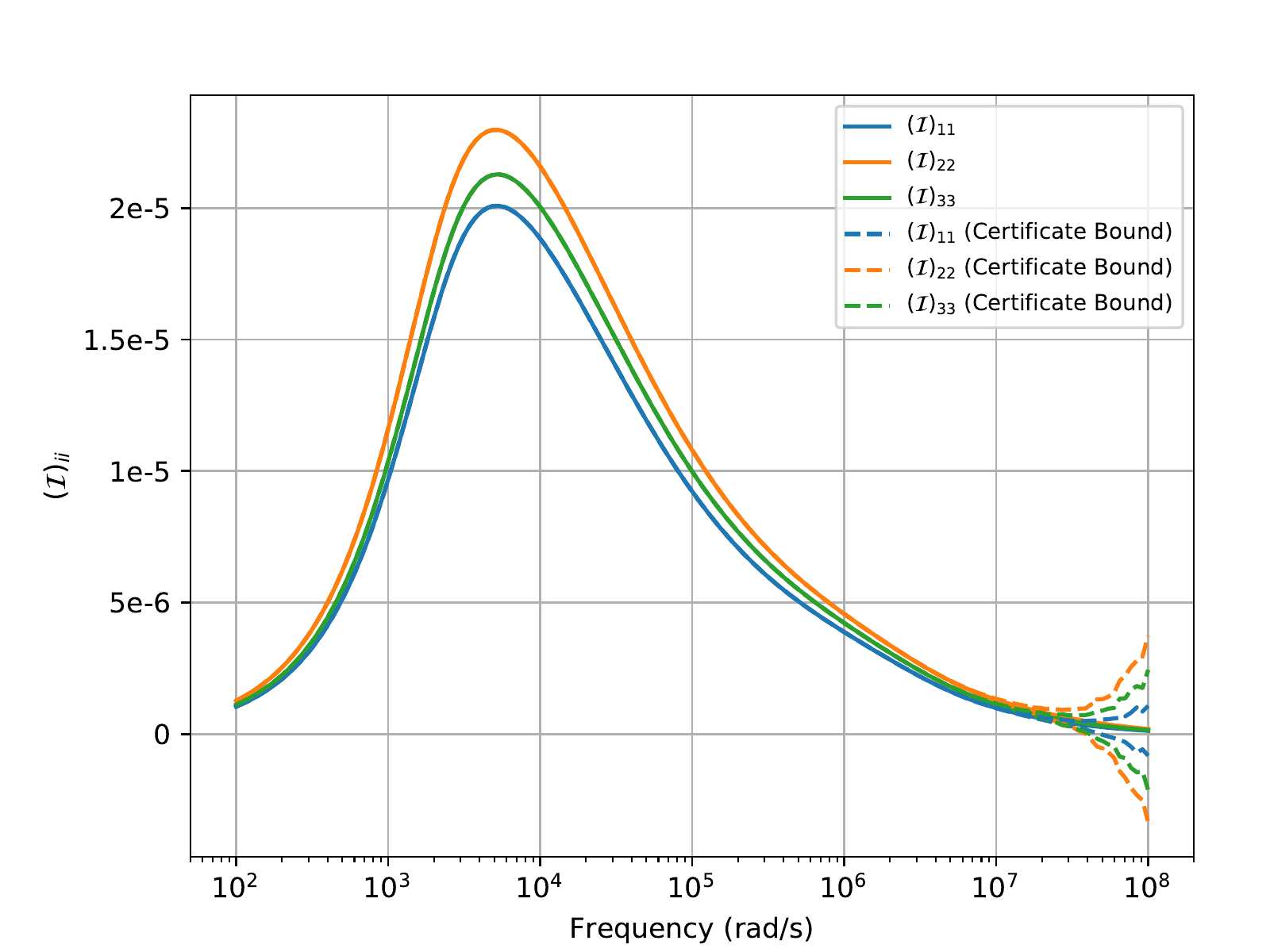} \\
\textrm{\footnotesize{(a) $(\mathcal{N}^0[\alpha B]+\mathcal{R}[\alpha B, \omega] )_{ii}$}} & \textrm{\footnotesize{(b) $(\mathcal{I}[\alpha B, \omega])_{ii} $}} \\
\includegraphics[width=0.5\textwidth, keepaspectratio]{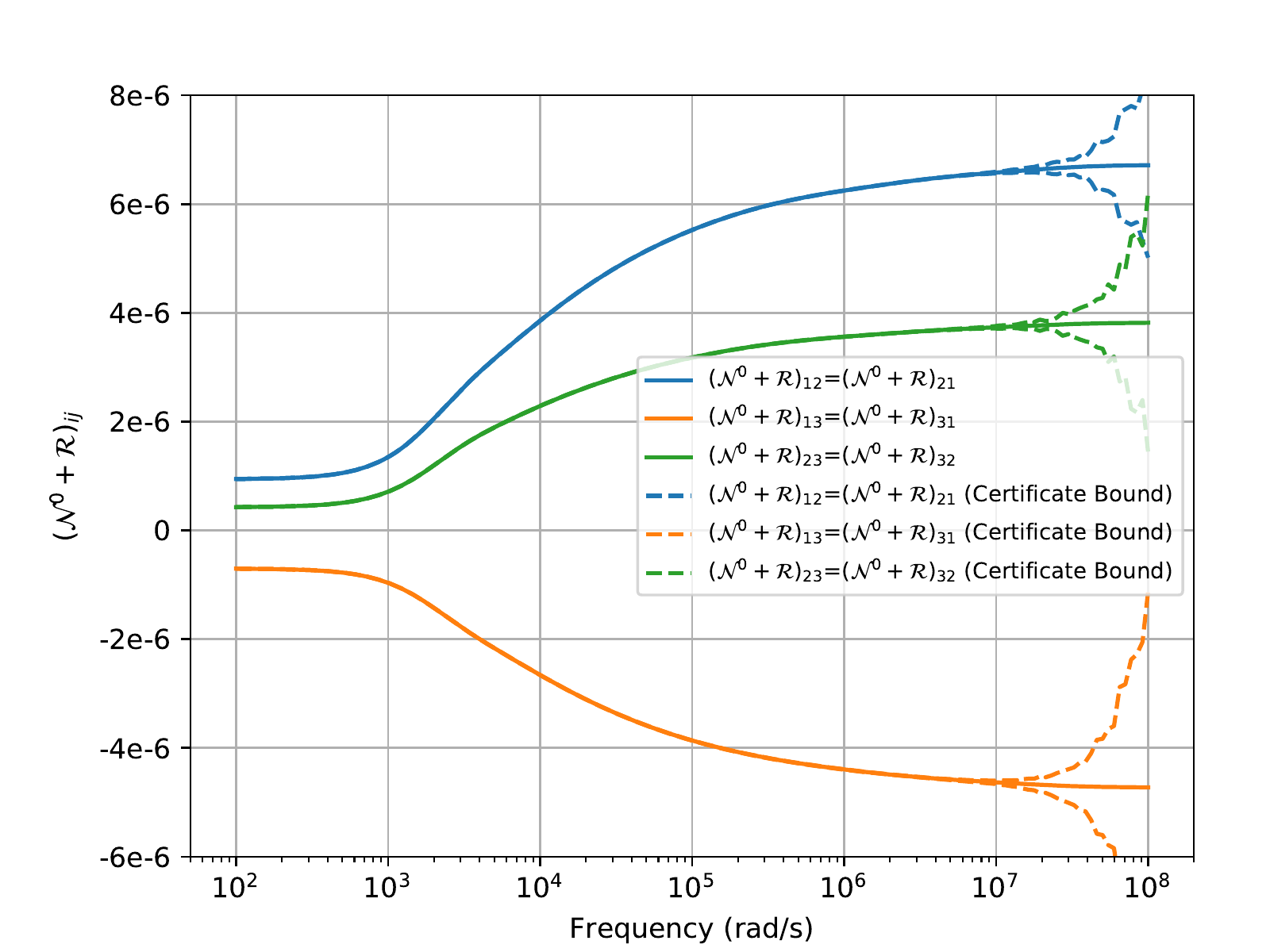} &
\includegraphics[width=0.5\textwidth, keepaspectratio]{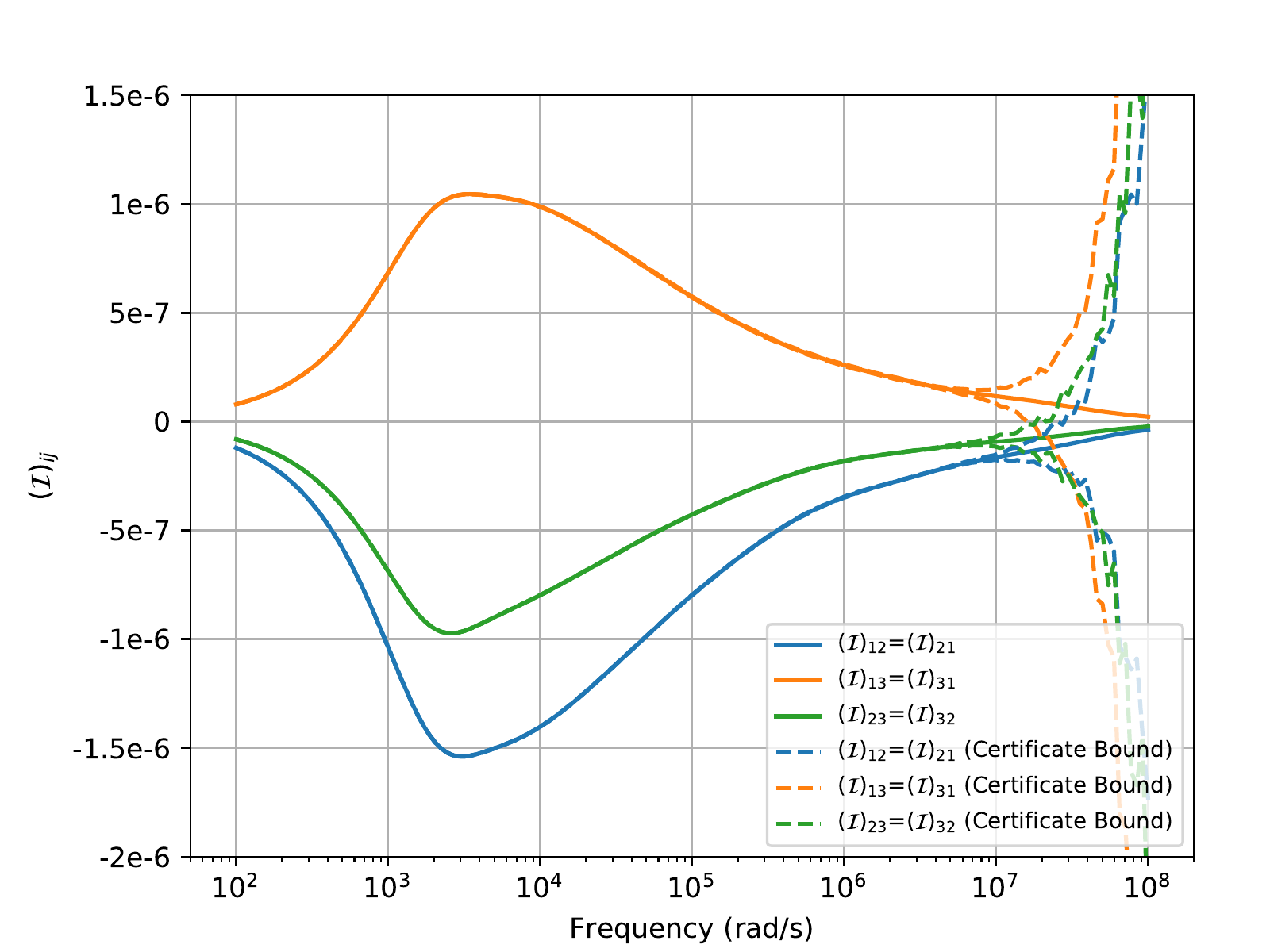} \\
\textrm{\footnotesize{(c) $ (\mathcal{N}^0[\alpha B]+\mathcal{R}[\alpha B, \omega] )_{ij}, i\ne j$}} & \textrm{\footnotesize{(d) $\mathcal{I}[\alpha B, \omega])_{ij}, i \ne j $}} 
\end{array}$$
\caption{
Irregular tetrahedron with $\mu_r=2$,  $\sigma_*=5.96\times10^6$ S/m, $\alpha=0.01$ m: 
 PODP applied to the computation of $\mathcal{M}[\alpha B, \omega]$ with $TOL=1\times 10^{-6}$ and $N=21$ showing the PODP solution and output certificates $(\cdot)\pm (\Delta [ \omega])_{ij}$ for  $(a)$ $ (\mathcal{N}^0[\alpha B]+\mathcal{R}[\alpha B, \omega])_{ij}$, with $i=j$
$(b)$ $ (\mathcal{I}[\alpha B, \omega])_{ij}$, with $i=j$, $(c)$ $ (\mathcal{N}^0[\alpha B]+\mathcal{R}[\alpha B, \omega])_{ij}$, with $i\ne j$,
$(d)$ $ (\mathcal{I}[\alpha B, \omega])_{ij}$ with $i\ne j$, each with $\omega$. 
}
\label{fig:ErrorTetra}
\end{figure}

%%%%%%%%%%%%%%%%%%%%%%%%%%%%%%%%%%%%%%%%%%%%%%%%%%%%%%%%%%%%%%%%
\subsection{Inhomogeneous conducting bar}
%%%%%%%%%%%%%%%%%%%%%%%%%%%%%%%%%%%%%%%%%%%%%%%%%%%%%%%%%%%%%%%%
As a final example  we consider $B_\alpha = \alpha B$ to the inhomogeneous conducting bar made up from two different conducting materials. The size, shape and materials of this object are the same as those presented in Section 6.1.3 of~\cite{LedgerLionheartamad2019}. This object has rotational and reflectional symmetries such that ${\mathcal M} [ \alpha B, \omega]$ has independent coefficients $({\mathcal M} [ \alpha B, \omega])_{11}$, $({\mathcal M} [ \alpha B, \omega])_{22}=({\mathcal M} [ \alpha B, \omega])_{33} $  and, thus, $\mathcal{N}^0[\alpha B]$, $\mathcal{R}[\alpha B, \omega]$, $\mathcal{I}[\alpha B, \omega]$ each have $2$ independent eigenvalues.
To compute the full order model, we  set $\Omega$ to be a sphere of radius 100 centred about the origin and discretise it with a mesh of 30209 unstructured tetrahedral elements, refined towards the object, and a polynomial order of $p=3$. This discretisation has already been found to produce an accurate representation of $\mathcal{M}[\alpha B, \omega] $ for the frequency range with  $\omega_{min} = 1 \times 10^2 {\textrm{ rad}}  / { \textrm{s}}$ and $\omega_{max} = 1 \times 10^8  {\textrm{ rad}} / {\textrm{s}}$.

The reduced order model is constructed using $N=13 $ snapshots  at logarithmically spaced frequencies with $TOL=1 \times 10^{-4} $. Figure~\ref{fig:Tetra} shows the results for $\lambda_i(\mathcal{N}^0[\alpha B]+\mathcal{R}[\alpha B, \omega])$ and $\lambda_i(\mathcal{I}[\alpha B, \omega])$, each with $\omega$, for both the full order model and the PODP. The agreement is excellent in both cases.
\begin{figure}[H]
$$\begin{array}{cc}
\includegraphics[width=0.48\textwidth, keepaspectratio]{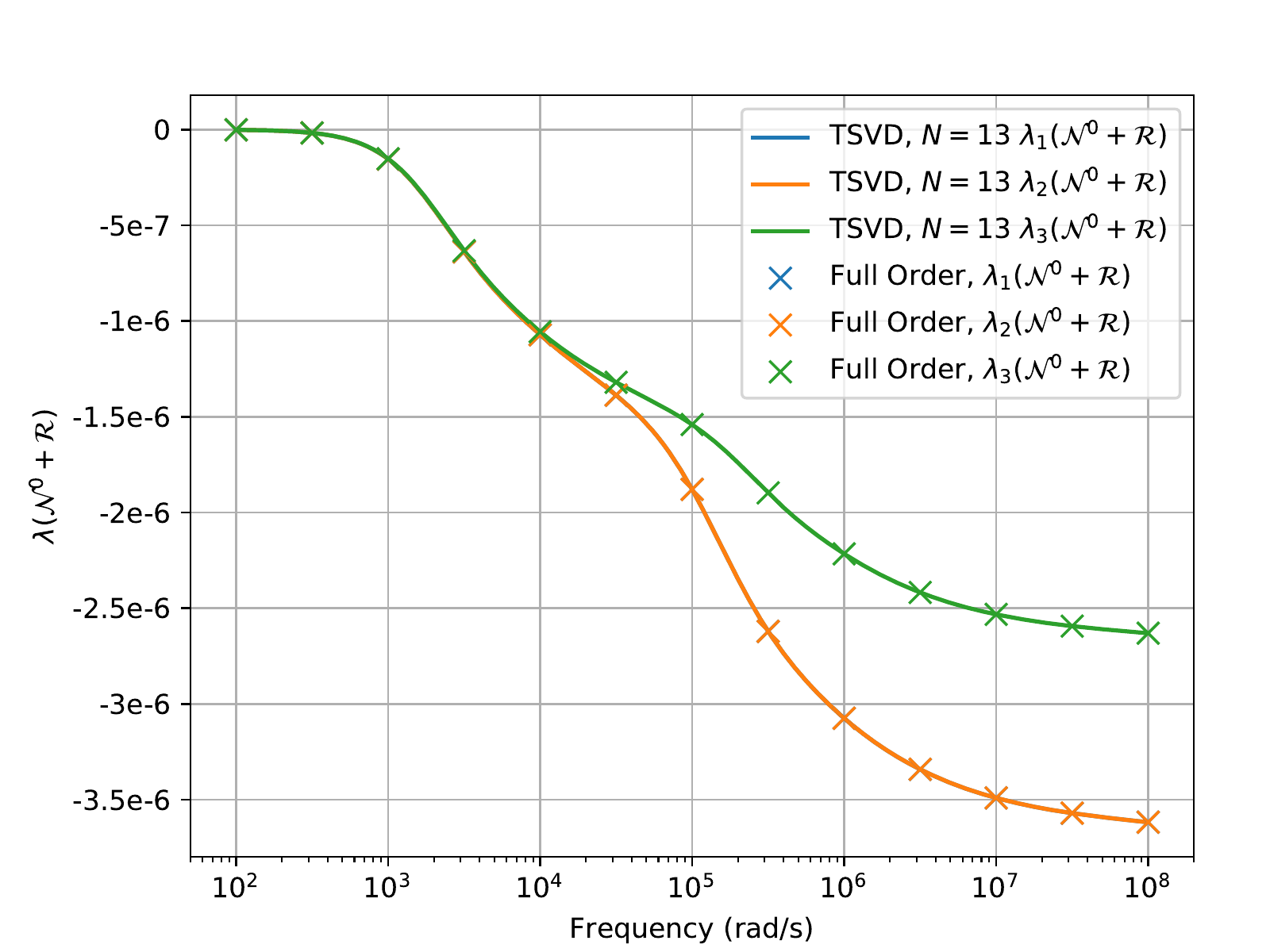} & \includegraphics[width=0.48\textwidth, keepaspectratio]{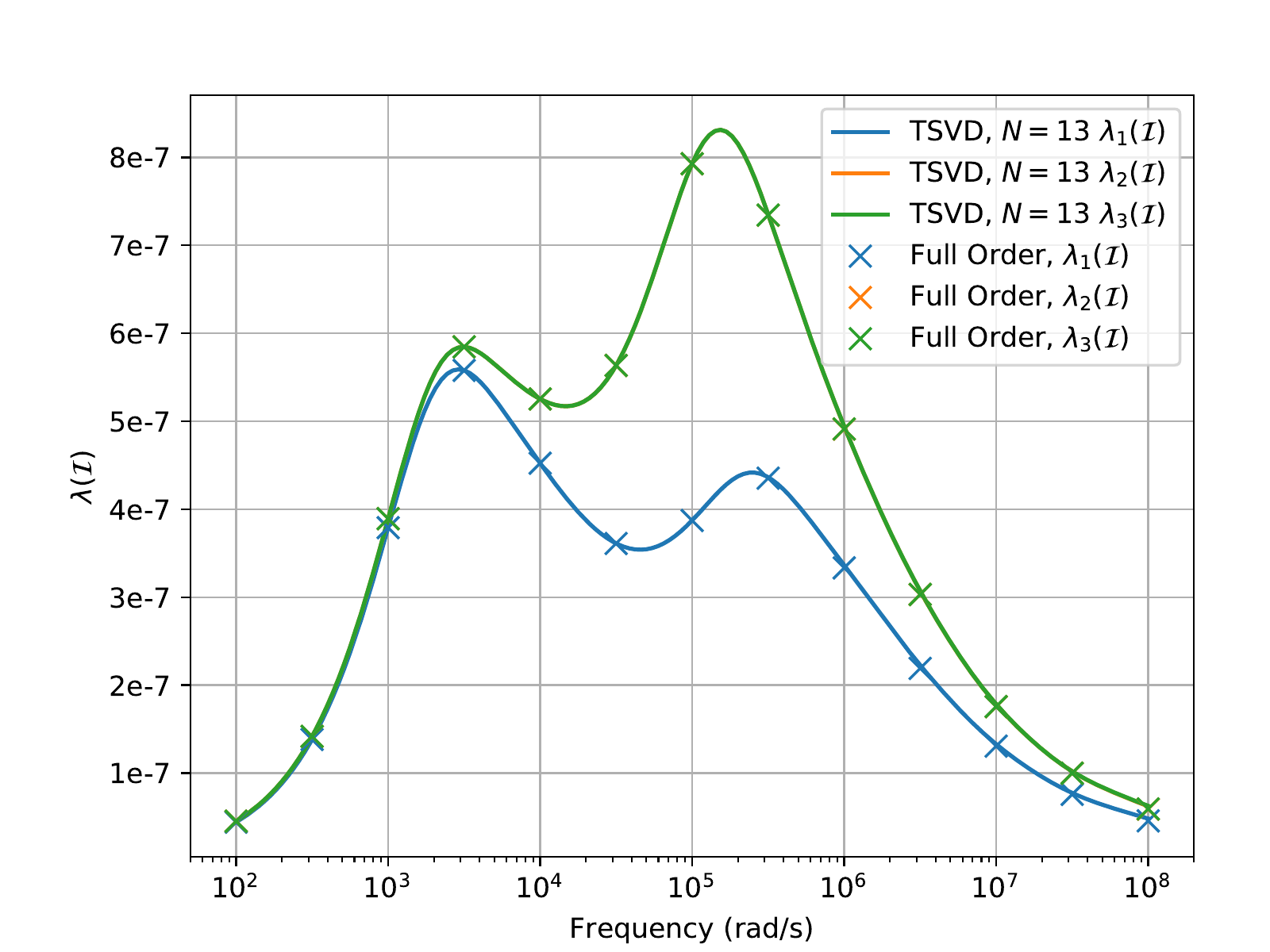}\\
\textrm{\footnotesize{(a) $\lambda_i(\mathcal{N}^0[\alpha B]+\mathcal{R}[\alpha B, \omega] )$}} & \textrm{\footnotesize{(b) $\lambda_i(\mathcal{I}[\alpha B, \omega] $)}}
\end{array}$$
\caption{Inhomogeneous bar with two distinct conductivities  (see Section 6.1.3 of~\cite{LedgerLionheartamad2019}): 
PODP applied to the computation of $\mathcal{M}[\alpha B, \omega]$ $N=13$ and  $TOL=1 \times 10^{-4} $
 showing $(a)$ $\lambda_i(\mathcal{N}^0[\alpha B]+\mathcal{R}[\alpha B, \omega])$   and $(b)$ $\lambda_i(\mathcal{I}[\alpha B, \omega])$, each with $\omega$.%\textcolor{red}{remove titles from the graphs}
 }
\label{fig:Bar}
\end{figure}

In Figure~\ref{fig:ErrorBar} we show the output certificates $(\mathcal{R}^{PODP}[\alpha B, \omega]+{\mathcal N}^{0, PODP}[\alpha B ])_{ii}\pm (\Delta[\omega])_{ii}$ (no summation over repeated indices implied) and $(\mathcal{I}^{PODP}[\alpha B, \omega])_{ii}\pm (\Delta[\omega])_{ii}$, both with $\omega$,
%In Figure~\ref{fig:ErrorBar} we show the output certificates $(\mathcal{R}[\alpha B, \omega]+{\mathcal N}^{0}[\alpha B ])_{ii}\pm (\Delta[\omega])_{ii}$ (no summation over repeated indices implied) and $(\mathcal{I}[\alpha B, \omega])_{ii}\pm (\Delta[\omega])_{ii}$, both with $\omega$,
obtained by applying the technique described in Section~\ref{sect:outputcert} for the case where $N=23$ and $TOL=1\times10^{-6}$.  Note that we increased the number of snapshots from $N=13$ to $N=23$ and have reduced the tolerance to ensure tight certificates bounds, except at large frequencies. 

\begin{figure}[H]
$$\begin{array}{cc}
\includegraphics[width=0.5\textwidth, keepaspectratio]{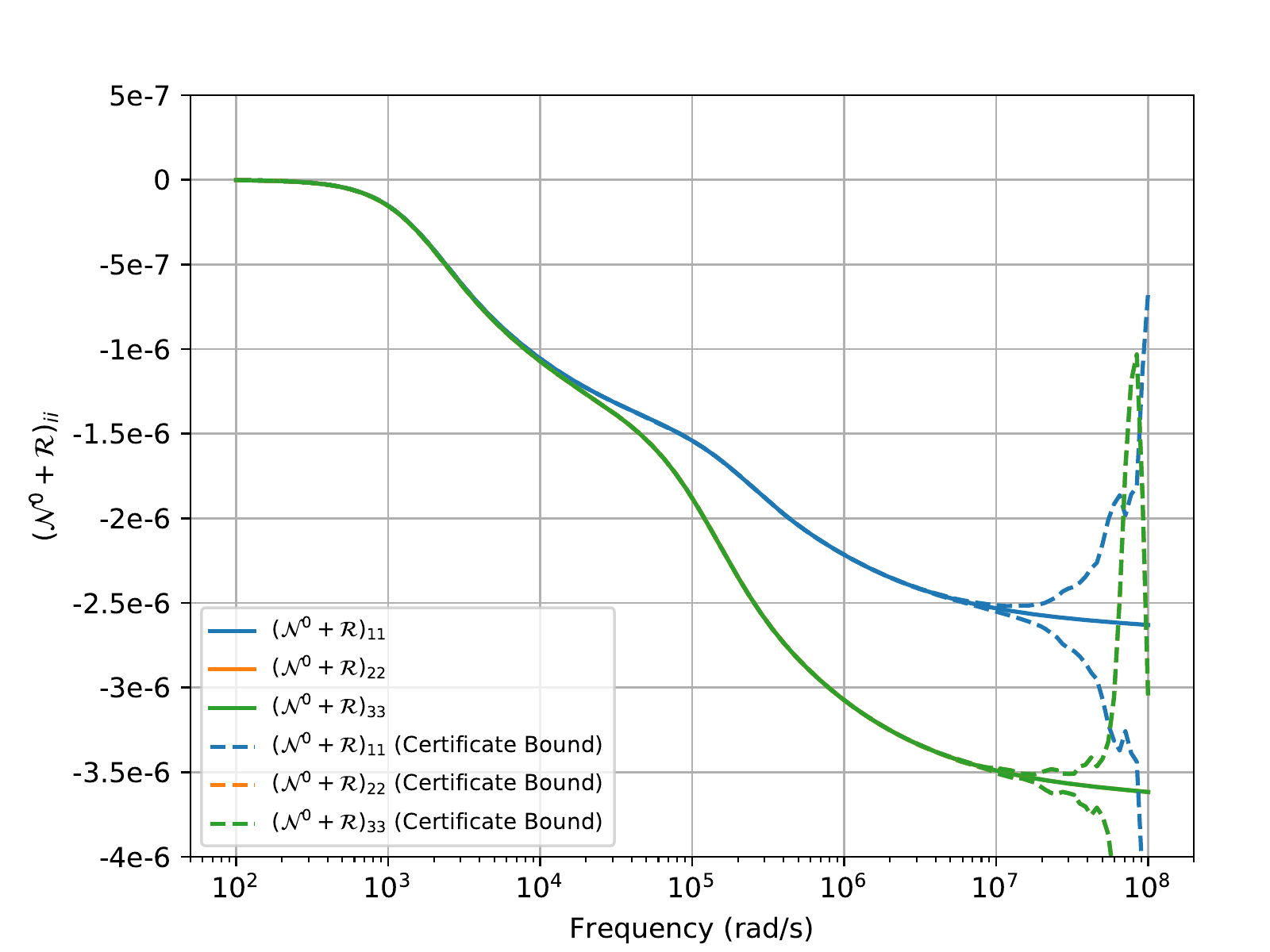} &
\includegraphics[width=0.5\textwidth, keepaspectratio]{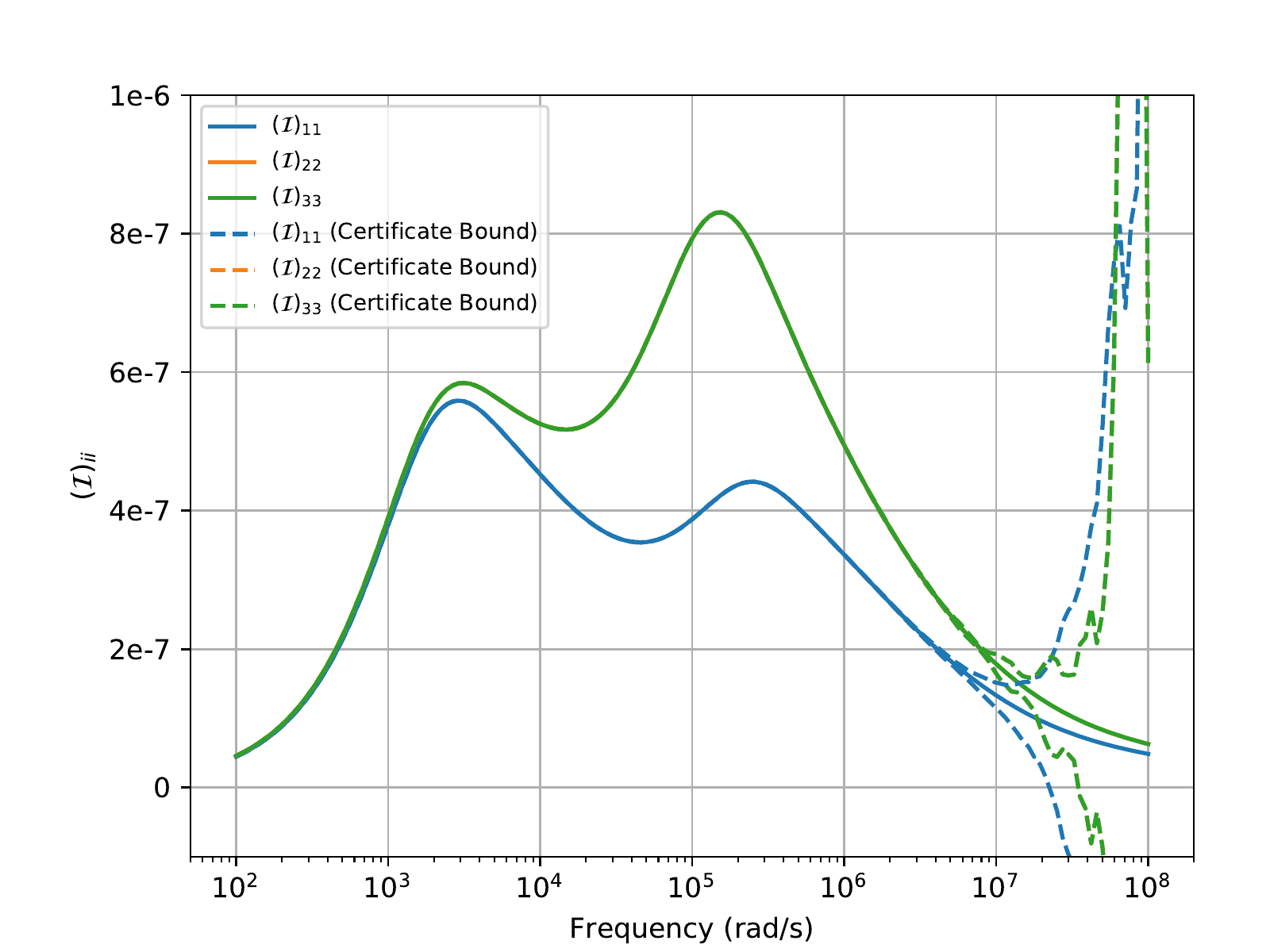} \\
\textrm{\footnotesize{(a) $(\mathcal{N}^0[\alpha B]+\mathcal{R}[\alpha B, \omega] )_{ii}$}} & \textrm{\footnotesize{(b) $(\mathcal{I}[\alpha B, \omega])_{ii} $}} 
\end{array}$$
\caption{
Inhomogeneous bar with two distinct conductivities  (see Section 6.1.3 of~\cite{LedgerLionheartamad2019}): 
PODP applied to the computation of $\mathcal{M}[\alpha B, \omega]$ with $N=23$ showing the PODP solution and output certificates $(\cdot ) \pm (\Delta [\omega])_{ii}$ for  $(a)$ $ (\mathcal{N}^0[\alpha B]+\mathcal{R}[\alpha B, \omega])_{ii}$,
$(b)$ $ (\mathcal{I}[\alpha B, \omega])_{ii}$,
each with $\omega$. 
}
\label{fig:ErrorBar}
\end{figure}

%%%%%%%%%%%%%%%%%%%%%%%%%%%%%%%%%%%%%%%%%%%%%%%%%%%%%%%%%%%%%%%%
\section{Numerical examples of scaling}\label{sect:examplesscale}
%%%%%%%%%%%%%%%%%%%%%%%%%%%%%%%%%%%%%%%%%%%%%%%%%%%%%%%%%%%%%%%%

In this section we illustrate the application of the results presented in Section~\ref{sect:scaling}.

%%%%%%%%%%%%%%%%%%%%%%%%%%%%%%%%%%%%%%%%%%%%%%%%%%%%%%%%%%%%%%%%
\subsection{Scaling of conductivity}
%%%%%%%%%%%%%%%%%%%%%%%%%%%%%%%%%%%%%%%%%%%%%%%%%%%%%%%%%%%%%%%%
As an illustration of Lemma~\ref{lemma:condscale}, we consider a conducting permeable sphere $B_\alpha=\alpha B $ where  $\alpha=0.01$~m with materials properties  $\mu_r=1.5$ and $\sigma_{*}^{(1)}=1 \times 10^7$ S/m and a second object, which is the same as the first except that $\sigma_*^{(2)} = s \sigma_*^{(1)} = 10 \sigma_*^{(1)}$. In Figure~\ref{fig:SphereSigma}, we compare the full order computations of ${\mathcal M} [ \alpha B , \omega, \mu_r,\sigma_*^{(1)}]$ and ${\mathcal M} [ \alpha B , \omega, \mu_r,\sigma_*^{(2)}]$ with that obtained from (\ref{eqn:condscale}). We observe that the translation predicted by (\ref{eqn:condscale}) is in excellent agreement with the full order model solution for  ${\mathcal M} [ \alpha B, \omega, \mu_r,\sigma_*^{(2)}] $.
\begin{figure}[H]
$$\begin{array}{cc}
\includegraphics[width=0.5\textwidth, keepaspectratio]{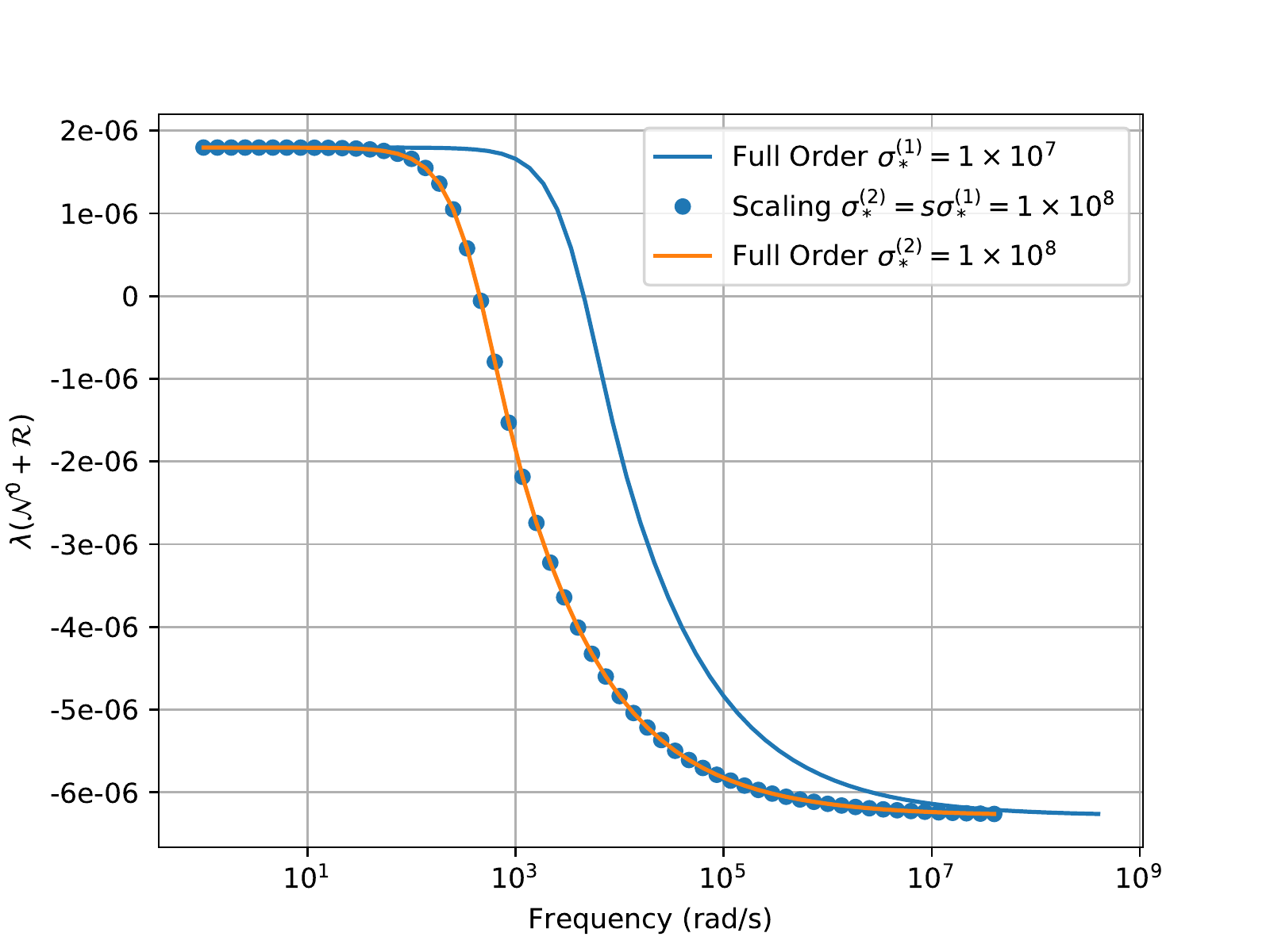} & \includegraphics[width=0.5\textwidth, keepaspectratio]{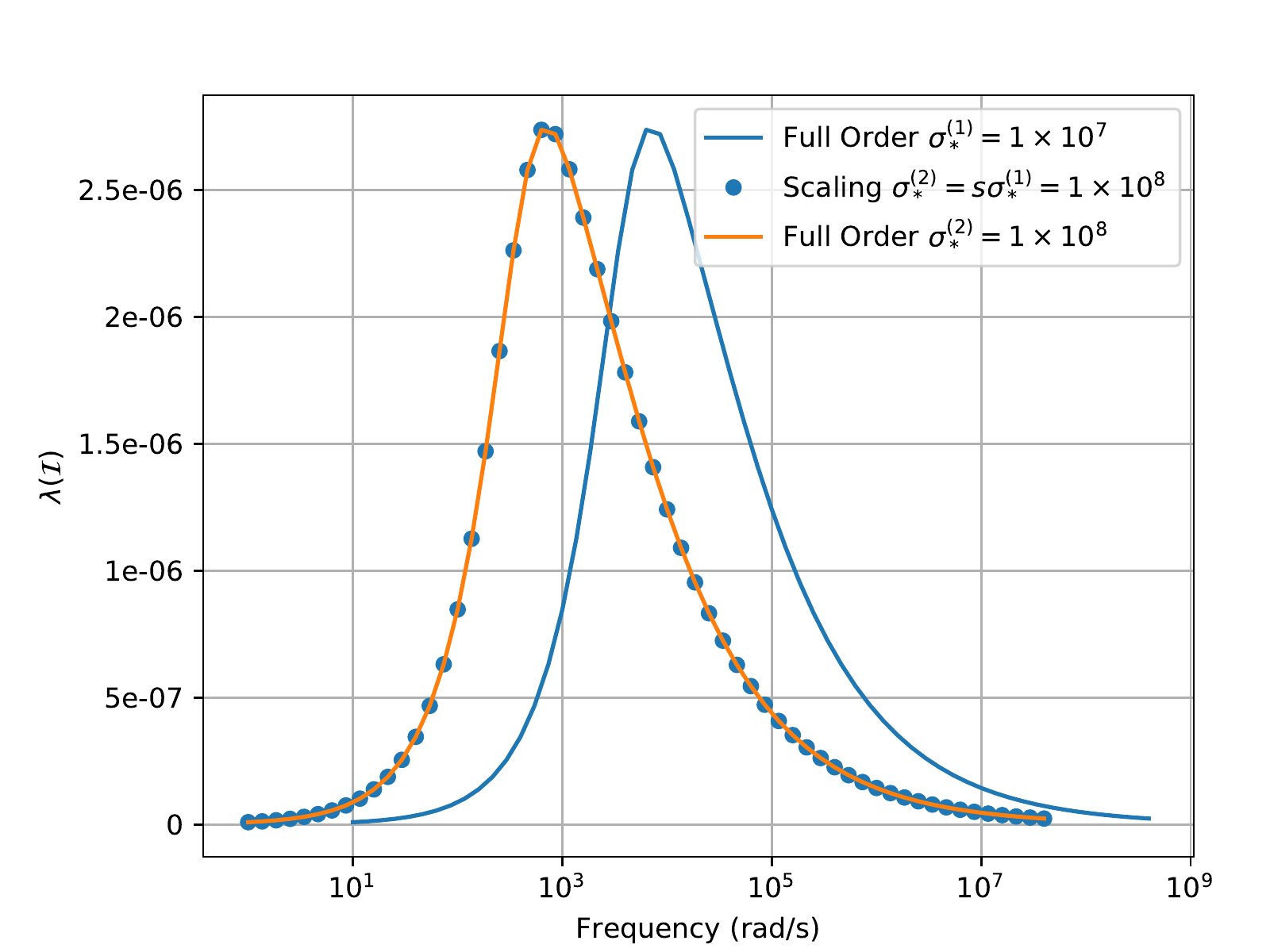}\\
\textrm{\footnotesize{(a) $\lambda_i(\mathcal{N}^0[\alpha B,\mu_r]+\mathcal{R}[\alpha B, \omega,\mu_r,\sigma_*])$}} & \textrm{\footnotesize{(b) $\lambda_i(\mathcal{I}[\alpha B, \omega,\mu_r,\sigma_*])$}}
\end{array}$$
%\begin{figure}\label{LogvsLin}
\caption{Sphere with $\mu_r=1.5$, $\sigma_*^{(1)}=1\times10^7$ S/m , $\alpha=0.01$ m and second sphere, which is the same as the first except that $\sigma_*^{(2)} = s \sigma_*^{(1)} = 10 \sigma_*^{(1)}$:
showing the translation predicted by (\ref{eqn:condscale})  compared with  the full order model solutions for $(a)$ $\lambda_i(\mathcal{N}^0[\alpha B,\mu_r]+\mathcal{R}[\alpha B, \omega,\mu_r,\sigma_*])$ and $ (b)$ $\lambda_i(\mathcal{I}[\alpha B, \omega,\mu_r,\sigma_*])$.}
\label{fig:SphereSigma}
\end{figure}
%%%%%%%%%%%%%%%%%%%%%%%%%%%%%%%%%%%%%%%%%%%%%%%%%%%%%%%%%%%%%%%%
\subsection{Scaling of object size}
%%%%%%%%%%%%%%%%%%%%%%%%%%%%%%%%%%%%%%%%%%%%%%%%%%%%%%%%%%%%%%%%
To illustrate Lemma~\ref{lemma:alphascale},  we consider a conducting permeable tetrahedron $B_\alpha^{(1)}=\alpha^{(1)} B=0.01B $ with vertices as described in Section~\ref{sect:tetra} and   material properties $\mu_r=1.5$ and $\sigma_*=1 \times 10^6$ S/m. Then, we consider a second object $B_\alpha^{(2)} = \alpha^{(2)}B = s\alpha^{(1)} B=0.015B$, which, apart from its size, is otherwise the same as $B_\alpha^{(1)}$.
 In Figure~\ref{fig:TetraAlpha}, we compare the full order computations of ${\mathcal M} [ \alpha^{(1)} B, \omega, \mu_r,\sigma_*]$ and ${\mathcal M} [  \alpha^{(2)} B, \omega, \mu_r,\sigma_*]$ with that obtained from (\ref{eqn:alphascale}). We observe that the translation and scaling predicted by (\ref{eqn:alphascale}) is in excellent agreement with the full order model solution for  ${\mathcal M} [ \alpha^{(2)} B, \omega, \mu_r,\sigma_*]$.

\begin{figure}[H]
$$\begin{array}{cc}
\includegraphics[width=0.5\textwidth, keepaspectratio]{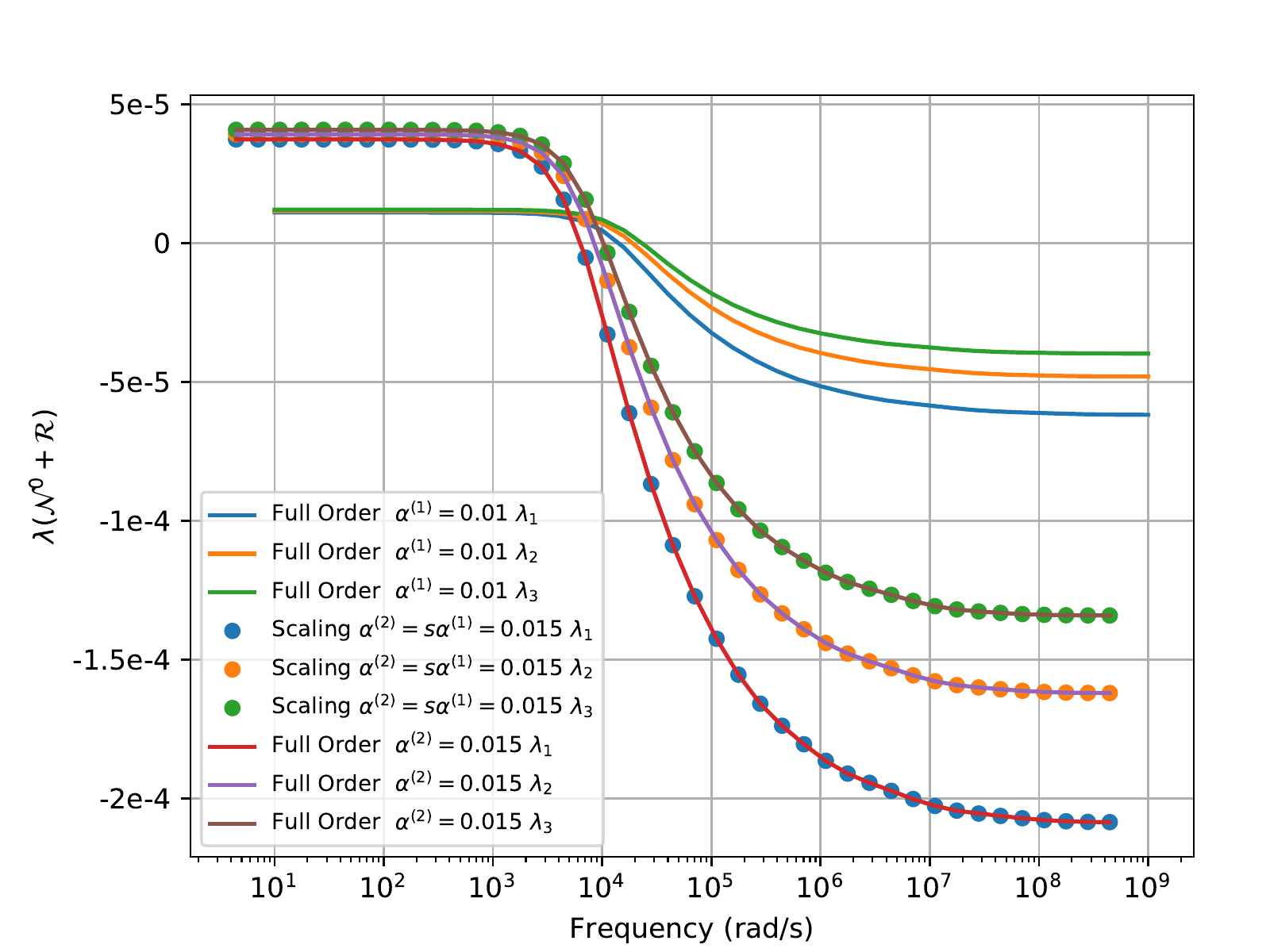} & \includegraphics[width=0.5\textwidth, keepaspectratio]{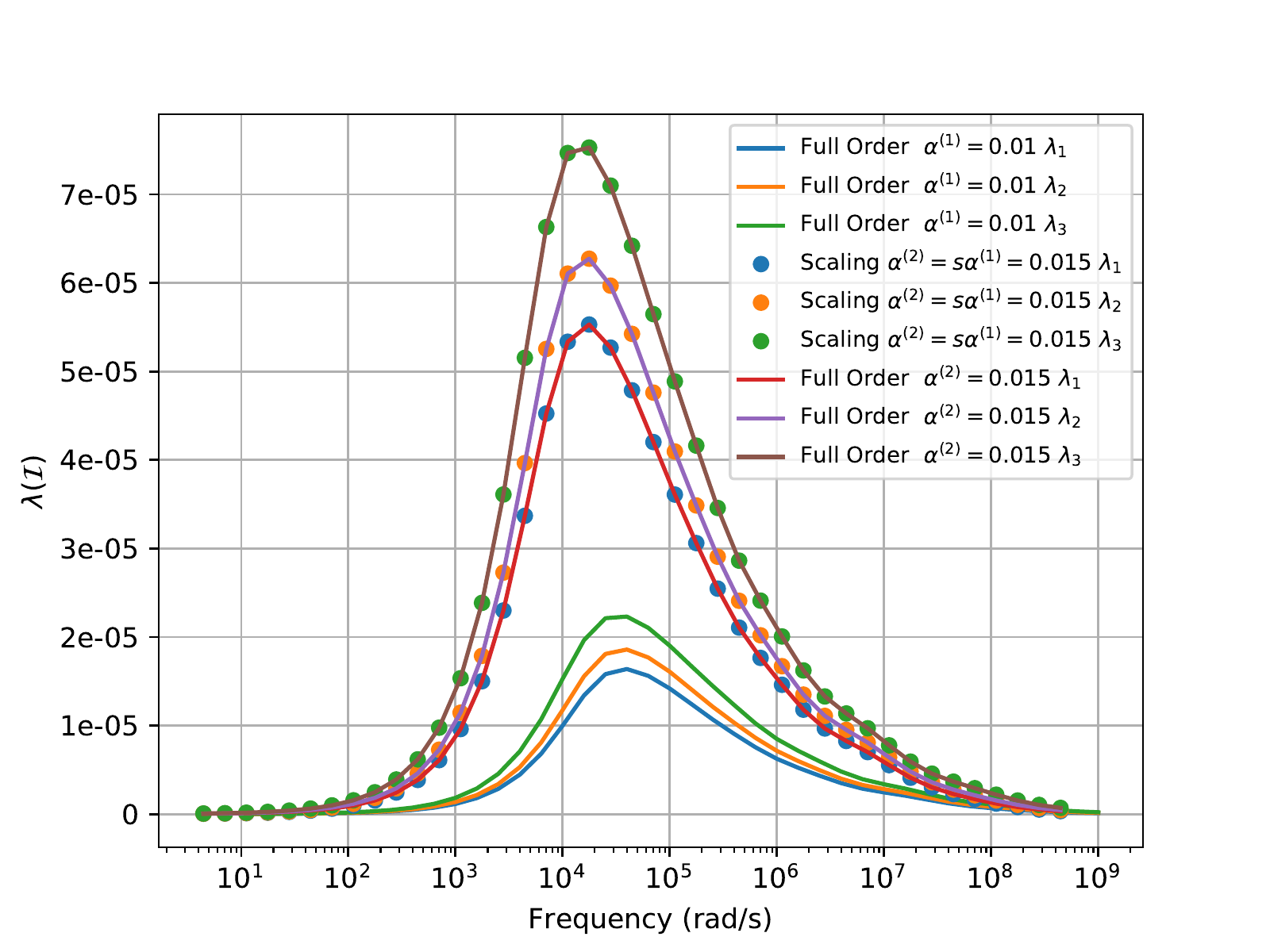}\\
\textrm{\footnotesize{(a) $\lambda_i(\mathcal{N}^0[\alpha B,\mu_r]+\mathcal{R}[\alpha B, \omega,\mu_r,\sigma_*])$}} & \textrm{\footnotesize{(b) $\lambda_i(\mathcal{I}[\alpha B, \omega,\mu_r,\sigma_*])$}}
\end{array}$$
\caption{Tetrahedron $B_\alpha^{(1)}=\alpha^{(1)} B=0.01B $  with $\mu_r=1.5$ and $\sigma_*=1 \times 10^6$~S/m, $\alpha=0.01$ m  and a second tetrahedron, which is the same as the first except  that $B_\alpha^{(2)} = \alpha^{(2)}B = s\alpha^{(1)} B=0.015B$: showing the translation and scaling predicted by (\ref{eqn:alphascale})  compared with  the full order model solutions for $(a)$ $\lambda_i(\mathcal{N}^0[\alpha B,\mu_r]+\mathcal{R}[\alpha B, \omega,\mu_r,\sigma_*])$ and $ (b)$ $\lambda_i(\mathcal{I}[\alpha B, \omega,\mu_r,\sigma_*])$.}
\label{fig:TetraAlpha}
\end{figure}

\section{Conclusions}

An application of a ROM using PODP for the efficient computation of the spectral signature of the MPT has been studied in this paper. The full order model has been approximated by ${\bm H}(\hbox{curl})$ conforming discretisation using the \texttt{NGSolve} finite element package. The offline stage of the ROM involves computing a small number of snapshots of the full order model at logarithmically spaced frequencies, then in the online stage, the spectral signature of the MPT is rapidly and accurately predicted to arbitrarily fine fidelity using PODP. Output certificates have been derived and can be computed in the online stage at negligible computational cost and ensure accuracy
of the ROM prediction. If desired, these output certificates could be used to drive an adaptive procedure for choosing new snapshots, in a similar manner to the approach presented in~\cite{hesthaven2016}. However, by choosing the frequency snapshots logarithmically, accurate spectral signatures of the MPT were already obtained with tight certificate bounds.
In addition, simple scaling results, which enable the MPT spectral signature to be easily computed from an existing set of coefficients under the scaling of an object's conductivity or object size, have been derived.  A series of numerical examples have been presented to demonstrate the accuracy and efficiency of our approach for homogeneous and inhomogeneous conducting permeable objects. Future work involves applying the presented approach to generate a dictionary of MPT spectral signatures for different objects for the purpose of metallic object identification using a classifier.

\section*{Acknowledgements}
B.A. Wilson gratefully acknowledges the financial support received from EPSRC in the form of a DTP studentship with project reference number 2129099. P.D. Ledger gratefully acknowledges the financial support received from EPSRC in the form of grant EP/R002134/1. The authors are grateful to Professors W.R.B. Lionheart and A.J. Peyton from The University of Manchester and Professor T. Betcke from University College London for research discussions at project meetings and to the group of Professor J. Sch\"oberl from the Technical University of Vienna for their technical support  on \texttt{NGSolve}.  {\bf EPSRC Data Statement:} All data is provided in Section~\ref{sect:examplespodp}. This paper does not have any conflicts of interest.

\bibliographystyle{plain}
\bibliography{paperbib}

\end{document}